\documentclass[12pt]{amsart}
\DeclareRobustCommand{\gobblefive}[5]{}
\newcommand*{\SkipTocEntry}{\addtocontents{toc}{\gobblefive}}
\usepackage[margin=3.3cm]{geometry}
\usepackage{amscd,amsmath,amsfonts,amssymb,latexsym}
\usepackage{hyperref,graphicx}
\usepackage{xcolor}
\hypersetup{
    colorlinks=true,
    citecolor=blue,
    linkcolor=blue,
    urlcolor=blue,
}
\usepackage{tikz}
\usepackage[all]{xy}
\usepackage{slashed}
\usepackage{soul}
\usepackage{comment}
\usepackage{extarrows}
\usepackage{mathtools}
\usepackage{url}
\usepackage{lipsum}

\numberwithin{equation}{section}
\theoremstyle{plain}
\newtheorem{lemma}{Lemma}[section]
\newtheorem{proposition}[lemma]{Proposition}

\newtheorem{theorem}[lemma]{Theorem}
\newtheorem{corollary}[lemma]{Corollary}

\theoremstyle{definition}
\newtheorem{definition}[lemma]{Definition}
\newtheorem{remark}[lemma]{Remark}

\let\C\relax
\newcommand{\C}{{\mathbb C}}
\newcommand{\R}{{\mathbb R}}

\newcommand{\Hh}{{\mathcal H}}
\newcommand{\Ll}{{\mathcal L}}
\newcommand{\Mm}{{\mathcal M}}

\newcommand{\id}{{\rm id}}
\newcommand{\Om}{{\Omega}}
\newcommand{\om}{{\omega}}

\newcommand{\la}{\langle}
\newcommand{\ra}{\rangle}

\newcommand{\vol}{\mbox{\rm vol}}
\newcommand{\rmspan}{\mbox{\rm span}}

\newcommand{\Sp}{{\text{\rm Spin}(7)}}
\renewcommand{\sp}{{\mathfrak {spin} (7)}}
\renewcommand{\Im}{{ \rm Im \,}}

\newcommand{\Aa}{{\mathcal A}}
\newcommand{\Bb}{{\mathcal B}}
\newcommand{\Ee}{{\mathcal E}}
\newcommand{\Ff}{{\mathcal F}}
\newcommand{\wFf}{{\widehat{\mathcal F}_{\Sp}}} 
\newcommand{\Gg}{{\mathcal G}}

\newcommand{\Pp}{{\mathcal P}}
\newcommand{\Ss}{{\mathcal S}}

\newcommand{\Uu}{{\mathcal U}}
\newcommand{\Vv}{{\mathcal V}}
\newcommand{\Ww}{{\mathcal W}}

\newcommand{\ja}{{\bf j_1}}
\newcommand{\jb}{{\bf j_2}}
\newcommand{\jc}{{\bf j_3}}

\renewcommand{\i}{{\sqrt{-1}}}
\renewcommand{\l}{{\ell}}

\newcommand{\n}{{\nabla}}

\begin{document}

%\date{\today}

\title[Deformation of $\Sp$-dDT]
{Deformation theory of deformed Donaldson--Thomas connections for 
$\Sp$-manifolds
}

\author{Kotaro Kawai}
\address{Department of Mathematics, Faculty of Science, Gakushuin University, 1-5-1 Mejiro, Toshima-ku, Tokyo 171-8588, Japan}
\email{kkawai@math.gakushuin.ac.jp}

\author{Hikaru Yamamoto}
\address{Department of Mathematics, Faculty of Pure and Applied Science, University of Tsukuba, 1-1-1 Tennodai, Tsukuba, Ibaraki 305-8577, Japan}
\email{hyamamoto@math.tsukuba.ac.jp}

\thanks{The first named author is supported by 
JSPS KAKENHI Grant Number JP17K14181, 
and the second named author is supported by JSPS KAKENHI Grant Number 
JP18K13415 and Osaka City University Advanced Mathematical Institute (MEXT Joint Usage/Research Center on Mathematics and Theoretical Physics)}
%%%%%%%%%%%%%%%%%%%%%%%%%%%%%%%%%%%%%%%%%%%%%%%%%%%%%%%%%
\begin{abstract}
A deformed Donaldson--Thomas connection for a 
manifold with a ${\rm Spin}(7)$-structure, 
which we call a ${\rm Spin}(7)$-dDT connection, 
is a Hermitian connection on a Hermitian line bundle $L$ over a manifold with a ${\rm Spin}(7)$-structure 
defined by fully nonlinear PDEs. 
It was first introduced by Lee and Leung 
as a mirror object of a Cayley cycle obtained by the real Fourier--Mukai transform 
and its alternative definition was suggested in our other paper. 
As the name indicates, a ${\rm Spin}(7)$-dDT connection can also be considered as 
an analogue of a Donaldson--Thomas connection (${\rm Spin}(7)$-instanton). 

In this paper, using our definition, 
we show that 
the moduli space $\mathcal{M}_{{\rm Spin}(7)}$ 
of ${\rm Spin}(7)$-dDT connections has similar properties to these objects. 
That is, we show the following for an open subset 
$\mathcal{M}'_{{\rm Spin}(7)} \subset \mathcal{M}_{{\rm Spin}(7)}$. 
(1) 
Deformations of elements of $\mathcal{M}'_{{\rm Spin}(7)}$ 
are controlled by a subcomplex of 
the canonical complex defined by Reyes Carri\'on by introducing 
a new ${\rm Spin}(7)$-structure from the initial ${\rm Spin}(7)$-structure and a ${\rm Spin}(7)$-dDT connection. 
(2)
The expected dimension of $\mathcal{M}'_{{\rm Spin}(7)}$ is finite. 
It is $b^1$, the first Betti number of the base manifold, 
if the initial ${\rm Spin}(7)$-structure is torsion-free. 
(3) Under some mild assumptions, $\mathcal{M}'_{{\rm Spin}(7)}$ is smooth if we perturb the initial ${\rm Spin}(7)$-structure generically. 
(4)
The space $\mathcal{M}'_{{\rm Spin}(7)}$ admits a canonical orientation 
if all deformations are unobstructed.
\end{abstract}

%%%%%%%%%%%%%%%%%%%%%%%%%%%%%%%%%%%%%%%%%%%%%%%%%%%%%%%%%
% [Memo]
% 53C07  	Special connections and metrics on vector bundles
% 58D27  	Moduli problems for differential geometric structures
% 58H15  	Deformations of structures
% 53D37  	Mirror symmetry, symplectic aspects; homological mirror symmetry; Fukaya category
% 53C25  	Special Riemannian manifolds
% 53C38  	Calibrations and calibrated geometries
\maketitle

\tableofcontents
%%%%%%%%%%%%%%%%%%%%%%%%%%%%%%%%%%%%%%%%%%%%%%%%%%%%%%%%%
%%%%%%%%%%%%%%%%%%%%%%%%%%%%%%%%%%%%%%%%%%%%%%%%%%%%%%%%%
%%%%%%%%%%%%%%%%%%%%%%%%%%%%%%%%%%%%%%%%%%%%%%%%%%%%%%%%%
%%%%%%%%%%%%%%%%%%%%%%%%%%%%%%%%%%%%%%%%%%%%%%%%%%%%%%%%%
%%%%%%%%%%%%%%%%%%%%%%%%%%%%%%%%%%%%%%%%%%%%%%%%%%%%%%%%%
\section{Introduction}
A deformed Donaldson--Thomas connection for a manifold with a ${\rm Spin}(7)$-structure, 
which we call a ${\rm Spin}(7)$-dDT connection, 
was first introduced by Lee and Leung \cite[Section 4.2.1]{LL} 
as a mirror object of a Cayley cycle (a Cayley submanifold with an ASD connection over it) 
via the real Fourier--Mukai transform. 
In \cite{KYFM}, we gave the following alternative definition of a $\Sp$-dDT connection 
which seems to be more appropriate by carefully computing the real Fourier--Mukai transform again. 

\begin{definition} \label{def:Spin7dDTMain}
Let $X^8$ be a manifold with a ${\rm Spin}(7)$-structure $\Phi$ 
and $L \to X$ be a smooth complex line bundle with a Hermitian metric $h$.
Denote by $\Om^k_\l \subset \Om^k$ the subspace of the space of $k$-forms 
corresponding to the $\l$-dimensional irreducible representation of $\Sp$. 
Let $\pi^k_\l : \Om^k \rightarrow \Om^k_\l$ be the projection. 
A Hermitian connection $\nabla$ of $(L,h)$ satisfying 
\begin{align}\label{eq:Spin7dDT intro}
\pi^2_{7} \left( F_\nabla + \frac{1}{6} * F_\nabla^3 \right) = 0 
\quad \mbox{and} \quad 
\pi^{4}_{7}(F_{\nabla}^2) &= 0
\end{align}
is called a \emph{deformed Donaldson--Thomas connection 
for a manifold with a ${\rm Spin}(7)$-structure} (a $\Sp$-dDT connection). 
Here, we regard the curvature 2-form $F_\n$ of $\n$ as a $\i \R$-valued closed 2-form on $X$. 
\end{definition}

\SkipTocEntry \subsection*{Motivation}
As mentioned above, a $\Sp$-dDT connection is the mirror object of a Cayley cycle. 
As the name indicates, it can also be considered as 
an analogue of a Donaldson--Thomas connection (${\rm Spin}(7)$-instanton). 
Thus, it is natural to expect  that ${\rm Spin}(7)$-dDT connections 
would have similar properties to these objects. 
In this paper, we especially study whether the moduli space 
$\Mm_\Sp$ of ${\rm Spin}(7)$-dDT connections has
similar properties, focusing on the deformation theory and the orientability. 

It is known that deformations of Cayley submanifolds are obstructed in general by \cite{McLean}. 
It is known that the moduli space of $\Sp$-instantons is a smooth manifold for a generic $\Sp$-structure (see \cite{MS2}) 
and it admits a canonical orientation by \cite{CL, CGJ, MS1}. 
Thus, we expect these properties for $\Mm_\Sp$.

%Under these motivations, it is natural to ask whether a $\Sp$-dDT connection inherits the same properties. 

In Proposition \ref{prop:F1toF2}, 
we show that the second equation of \eqref{eq:Spin7dDT intro} follows from the first equation 
when $* F^4_\n/24 \neq 1$, by which we mean that $* F^4_\n/24 - 1$ is nowhere vanishing. 
Thus, if we assume $* F^4_\n/24 \neq 1$, the theory of $\Sp$-dDT connections becomes more tractable. 
For technical reasons explained in Remark \ref{rem:prob}, 
we focus on the moduli space $\Mm'_\Sp \subset \Mm_\Sp$ of such connections and obtain the following.

\SkipTocEntry \subsection*{Main Results and implication}
\begin{theorem}[Theorems \ref{thm:moduli MSpin7}, \ref{thm:moduli generic} 
and Corollary \ref{cor:oriSpin7}] \label{Main1}

Suppose that 
$X^8$ is a compact connected manifold with a $\Sp$-structure $\Phi$ 
and $L \to X$ is a smooth complex line bundle with a Hermitian metric $h$. 
Let $\Mm'_{\Sp}$ be the set of all $\Sp$-dDT connections $\n$ of $L$ 
with $* F^4_\n/24 \neq 1$ divided by the $U(1)$-gauge action. Suppose that $\Mm'_{\Sp} \neq \emptyset$. 

\begin{enumerate}
\item
If $H^2(\#_\nabla) = \{\, 0 \,\}$ for $\nabla \in \Mm'_\Sp$, 
where $H^2(\#_\nabla)$ is the second cohomology of the complex $(\#_\nabla)$ 
for $\nabla$ defined in Subsection \ref{sec:infi deform}, 
the moduli space $\Mm'_\Sp$ is a finite dimensional smooth manifold around $\nabla$. 
If $\Phi$ is torsion-free, its dimension is $b^1$,  where $b^1$ is the first Betti number of $X$.

\item
Let $\n$ be a $\Sp$-dDT connection for $\Phi$ with $* F_\n^4/24 \neq 1$. 
Suppose that there exist $\Sp$-dDT connections $\n_\Psi$
satisfying $*_\Psi F_{\n_\Psi}^4/24 \neq 1$ for every $\Sp$-structure $\Psi$ sufficiently close to $\Phi$, 
where $*_\Psi$ is the Hodge star induced from $\Psi$. 
If 
\begin{enumerate}
\item
$F_\n \neq 0$ on a dense set of $X$, or 
\item
$\n$ is flat and the $\Sp$-structure $\Phi$ is torsion-free, 
\end{enumerate}
then for every generic $\Sp$-structure $\Psi$ close to $\Phi$, 
the subset of elements of 
the moduli space $\mathcal{M}'_{\Sp, \Psi}$ 
of $\Sp$-dDT connections $\n'$ for $\Psi$ with $*_\Psi F^4_{\n'}/24 \neq 1$ 
close to $[\n]$ is a finite dimensional smooth manifold (or empty).

\item
Suppose that $H^2(\#_\nabla) = \{\, 0 \,\}$ for any $[\n] \in \Mm'_{\Sp}$. 
Then, $\Mm'_{\Sp}$ is a manifold and it admits a canonical orientation 
if we choose an orientation of $\det D$, 
where $D$ is a Fredholm operator defined by \eqref{cancpxfordDT3}. 
\end{enumerate}
\end{theorem}

Note that the complex $(\#_\nabla)$ 
can be regarded as a subcomplex of the canonical complex 
introduced by Reyes Carri\'on \cite{Reyes} 
with respect to a newly introduced $\Sp$-structure from the initial $\Sp$-structure and a $\Sp$-dDT connection. 
See Section \ref{sec:expdim}. 
The $\Sp$-structure $\Phi$ need not to be torsion-free, 
which implies that there are many explicit examples for which we can apply this theorem.

%As far as the authors know, it is not known whether 
%the moduli space od Cayley submanifolds is orientable or not. 
%Theorem \ref{Main1} implies that the moduli space will admit a canonical orientation via mirror symmetry. 
%\cite[Remark 1.13]{CGJ}?

Theorem \ref{Main1} indicates that 
the moduli space $\Mm'_{\Sp}$ is similar to each of that of Cayley submanifolds and that of $\Sp$-instantons. 
A challenging problem in $\Sp$-geometry is 
to define an enumerative invariant of $\Sp$-manifolds 
by counting Cayley submanifolds or $\Sp$-instantons. 
By the similarities of moduli spaces, 
we might hope to define an enumerative invariant by counting $\Sp$-dDT connections 
if we can manage the case $*F^4_\n/24=1$ and find an appropriate compactification of the moduli space.  
This could be a challenging future work. 
It would be nice that the expected dimension of $\Mm'_{\Sp}$ is 0 
if $\Phi$ is torsion-free and $b^1=0$, 
which includes the case that a $\Sp$-manifold $(X^8, \Phi)$ has full holonomy $\Sp$.

\SkipTocEntry \subsection*{Ideas used in the proof}
The moduli space $\Mm'_{\Sp}$ is defined 
as a zero set of the so-called deformation map $\Ff_\Sp$. 
We introduce a new $\Sp$-structure from the initial $\Sp$-structure and a $\Sp$-dDT connection
and describe the linearization of $\Ff_\Sp$ in a ``nice way''. 
Using this, we show that each deformation is controlled by a subcomplex of the canonical complex 
introduced by Reyes Carri\'on \cite{Reyes}. 
The idea is similar to that of \cite{KY}, but we need more sophisticated calculations in the $\Sp$ case. 
We derive some formulas for $\Sp$-geometry in Section \ref{sec:basic} and Appendix \ref{sec:new Spin7 str}, 
which do not seem to appear in the literature.  
These formulas themselves would be interesting and useful for the study of $\Sp$-geometry.

Orienting the moduli space is considered to be important to define enumerative invariants. 
It is because we must count ``with signs" to define enumerative invariants 
and we need a canonical orientation of the moduli space to determine the signs. 
The point in our case is that our new $\Sp$-structures are connected in the space of $\Sp$-structures to the initial $\Sp$-structure 
by construction. 
Investigating this fact in more detail, we see that a family of Fredholm operators describing infinitesimal deformations is 
homotopic to the trivial family in the space of Fredholm operators. 
Then, the homotopy property of vector bundles implies that the determinant line bundle is trivial, and hence, 
the moduli space is orientable.

%%%%%%%%%%%%%%%%%%%%%%%%%%%%%%%%%%%%%%%%%%%%%%%%%%%%%%%%%
\SkipTocEntry \subsection*{Organization of this paper}
This paper is organized as follows. 
Section \ref{sec:basic} gives basic identities and some decompositions 
of the spaces of differential forms in $G_2$- and ${\rm Spin}(7)$-geometry that are used in this paper. 
In Section \ref{sec:suggestSpin7dDT}, 
we give some properties of $\Sp$-dDT connections analogous to \cite{KY}. 
Section \ref{sec:defofSpin7dDT} is devoted to the study of the deformation theory of $\Sp$-dDT connections $\n$ 
with $*F_\n^4/24 \neq 1$. We prove Theorems \ref{Main1} here. 
In Appendix \ref{sec:new Spin7 str}, 
we show that the newly induced $\Sp$-structure introduced in Section \ref{sec:defofSpin7dDT} 
is useful to describe the linearization of the deformation map ``nicely". 
Appendix \ref{app:notation} is the list of notation in this paper. 
%%%%%%%%%%%%%%%%%%%%%%%%%%%%%%%%%%%%%%%%%%%%%%%%%%%%%%%%%

\section{Basics on $G_2$- and ${\rm Spin}(7)$-geometry} \label{sec:basic}
In this section, 
we collect some basic definitions and equations on $G_2$- and ${\rm Spin}(7)$-geometry 
which we need in the calculations in this paper. 
Some formulas do not seem to appear in the literature and 
these formulas themselves would be interesting and useful for the study of $\Sp$-geometry.

%%%%%%%%%%%%%%%%%%%%%%%%%%%%%%%%%%%%%%%%%%%%%%%%%%%%%%%%%
\subsection{The Hodge-$\ast$ operator}
Let $V$ be an $n$-dimensional oriented real vector space with a scalar product $g$. 
By an abuse of notation, we also denote by $g$ the 
induced inner product on $\Lambda^k V^*$ from $g$. 
Let $\ast$ be the Hodge-$\ast$ operator.
The following identities are frequently used throughout this paper. 

For $\alpha, \beta \in \Lambda^k V^*$ and $v \in V$, we have 
\[
\begin{aligned}
\ast^2|_{\Lambda^k V^*} &= (-1)^{k(n-k)} {\rm id}_{\Lambda^k V^*}, & 
g(\ast \alpha, \ast \beta) &= g(\alpha, \beta), \\ 
i(v) \ast \alpha &= (-1)^k \ast (v^\flat \wedge \alpha), & 
\ast( i(v) \alpha) &= (-1)^{k+1} v^\flat \wedge \ast \alpha. 
\end{aligned}
\]

%%%%%%%%%%%%%%%%%%%%%%%%%%%%%%%%%%%%%%%%%%%%%%%%%%%%%%%%%

%%%%%%%%%%%%%%%%%%%%%%%%%%%%%%%%%%%%%%%%%%%%%%%%%%%%%%%%%
%%%%%%%%%%%%%%%%%%%%%%%%%%%%%%%%%%%%%%%%%%%%%%%%%%%%%%%%%
%%%%%%%%%%%%%%%%%%%%%%%%%%%%%%%%%%%%%%%%%%%%%%%%%%%%%%%%%
%%%%%%%%%%%%%%%%%%%%%%%%%%%%%%%%%%%%%%%%%%%%%%%%%%%%%%%%%
%%%%%%%%%%%%%%%%%%%%%%%%%%%%%%%%%%%%%%%%%%%%%%%%%%%%%%%%%
%%%%%%%%%%%%%%%%%%%%%%%%%%%%%%%%%%%%%%%%%%%%%%%%%%%%%%%%%

\subsection{Basics on $G_2$-geometry} \label{sec:G2 geometry}
Let $V$ be an oriented $7$-dimensional vector space. A \emph{$G_2$-structure} on $V$ is 
a 3-form $\varphi \in \Lambda^3 V^*$ such that there is a positively oriented basis 
$\{\, e_i \,\}_{i=1}^7$ of $V$ 
with the dual basis $\{\, e^i \,\}_{i=1}^7$ of $V^\ast$ satisfying 
\begin{equation} \label{varphi}
\varphi = e^{123} + e^{145} + e^{167} + e^{246} - e^{257} - e^{347} - e^{356},
\end{equation}
where $e^{i_1 \cdots i_k}$ is short for $e^{i_1} \wedge \cdots \wedge e^{i_k}$. Setting $\vol := e^{1 \cdots 7}$, 
the 3-form $\varphi$ uniquely determines an inner product $g_\varphi$ via 
\begin{equation} \label{eq:form-1def}
g_\varphi(u,v)\; \vol = \dfrac16 i(u) \varphi \wedge i(v) \varphi \wedge \varphi 
\end{equation}
for $u,v \in V$. 
It follows that any oriented basis $\{\, e_i \,\}_{i=1}^7$ for which \eqref{varphi} holds is orthonormal with respect to $g_\varphi$. Thus, the Hodge-dual of $\varphi$ with respect to $g_\varphi$ is given by 
\begin{equation} \label{varphi*}
\ast \varphi = e^{4567} + e^{2367} + e^{2345} + e^{1357} - e^{1346} - e^{1256} - e^{1247}.
\end{equation}
The stabilizer of $\varphi$ is known to be the exceptional $14$-dimensional simple Lie group 
$G_2 \subset {\rm GL}(V)$. The elements of $G_2$ preserve both $g_\varphi$ and $\vol$, that is, 
$G_2 \subset {\rm SO}(V, g_\varphi)$.

We summarize important well-known facts about the decomposition of tensor products of $G_2$-modules into irreducible summands. 
Denote by $V_k$ the $k$-dimensional irreducible $G_2$-module if there is a unique such module. For instance, $V_7$ is the irreducible $7$-dimensional $G_2$-module $V$ from above, 
and $V_7^* \cong V_7$. For its exterior powers, we obtain the decompositions
\begin{equation} \label{eq:DiffForm-V7}
\begin{array}{rlrl}
\Lambda^0 V^* \cong \Lambda^7 V^* \cong V_1, \quad
& \Lambda^2 V^*  \cong \Lambda^5 V^* \cong V_7 \oplus  V_{14},\\[2mm]
\Lambda^1 V^* \cong \Lambda^6 V^* \cong V_7, \quad
& \Lambda^3 V^* \cong \Lambda^4 V^* \cong V_1 \oplus V_7 \oplus V_{27},
\end{array}
\end{equation}
where $\Lambda^k V^* \cong \Lambda^{7-k} V^*$ due to the $G_2$-invariance of the Hodge isomorphism $\ast: \Lambda^k V^* \to \Lambda^{7-k} V^*$. We denote by $\Lambda^k_\l V^* \subset \Lambda^k V^*$ the subspace isomorphic to $V_\l$. 
Let 
\[
\pi^k_\l: \Lambda^k V^* \rightarrow \Lambda^k_\l V^*
\]
be the canonical projection. 
Identities for these spaces we need in this paper are as follows. 
\begin{equation*}
\begin{aligned}
\Lambda^2_7 V^* =& \{\, i(u) \varphi \mid u \in V \,\}
= \{\, \alpha \in \Lambda^2 V^* \mid * (\varphi \wedge \alpha) = 2 \alpha \,\},\\
\Lambda^2_{14} V^* =& \{\, \alpha \in \Lambda^2 V^* \mid \ast \varphi \wedge \alpha = 0\,\} 
= \{\, \alpha \in \Lambda^2 V^* \mid * (\varphi \wedge \alpha) = - \alpha \,\},\\
\Lambda^3_1 V^* =& \R \varphi, \\
\Lambda^3_7 V^* =& \{\, i(u) * \varphi \in \Lambda^3 V^* \mid u \in V \,\}. 
\end{aligned}
\end{equation*}
The following equations are well-known and useful in this paper. 

\begin{lemma} \label{lem:G2 identities}
For any $u \in V$, we have the following identities. 
\[
\begin{aligned}
\varphi \wedge i(u) * \varphi &= -4 * u^{\flat}, 
\\
* \varphi \wedge i(u) \varphi &= 3 * u^{\flat}, \\
\varphi \wedge i(u) \varphi &= 2 * (i(u) \varphi) = 2 u^{\flat} \wedge * \varphi. 
\end{aligned}
\]
\end{lemma}

\begin{definition}
Let $X$ be an oriented 7-manifold. A \emph{$G_2$-structure} on $X$ is a $3$-form $\varphi \in \Om^3$ 
such that at each $p \in X$ there is a positively oriented basis 
$\{\, e_i \,\}_{i=1}^7$ of $T_p X$ such that $\varphi_p \in \Lambda^3 T^*_p X$ is of the form \eqref{varphi}. As noted above, $\varphi$ determines a unique Riemannian metric $g = g_\varphi$ on $X$ by \eqref{eq:form-1def}, 
and the basis $\{\, e_i \,\}_{i=1}^7$ is orthonormal with respect to $g$.
A $G_2$-structure $\varphi$ is called \emph{torsion-free} if 
it is parallel with respect to the Levi-Civita connection of $g=g_\varphi$. 
A manifold with a torsion-free $G_2$-structure is called a \emph{$G_2$-manifold}. 
\end{definition}

A manifold $X$ admits a $G_2$-structure if and only if 
its frame bundle is reduced to a $G_2$-subbundle. 
Hence, considering its associated subbundles, 
we see that 
$\Lambda^* T^* X$ has the same decomposition as in \eqref{eq:DiffForm-V7}. 
The algebraic identities above also hold. 

%%%%%%%%%%%%%%%%%%%%%%%%%%%%%%%%%%%%%%%%%%%%%%%%%%%%%%%%%
\subsection{${\rm Spin}(7)$-geometry} \label{sec:Spin7 geometry}
Let $W$ be an $8$-dimensional oriented real vector space. 
A \emph{$\Sp$-structure on $W$} is 
a 4-form $\Phi \in \Lambda^4 W^*$ such that there is a positively oriented basis 
$\{\, e_i \,\}_{i=0}^7$ of $W$ 
with dual basis $\{\, e^i \,\}_{i=0}^7$ of $W^\ast$ satisfying 
\begin{equation}\label{Phi4}
\begin{aligned}
\Phi := & e^{0123} + e^{0145} + e^{0167} + e^{0246} - e^{0257} - e^{0347} - e^{0356}\\
& + e^{4567} + e^{2367} + e^{2345} + e^{1357} - e^{1346} - e^{1256} - e^{1247}, 
\end{aligned}
\end{equation}
where $e^{i_1 \cdots i_k}$ is short for $e^{i_1} \wedge \cdots \wedge e^{i_k}$.
Defining forms 
$\varphi$ and $\ast_7 \varphi$ on $V := \rmspan \{\, e_i \,\}_{i=1}^7 \subset W$ 
as in \eqref{varphi} and \eqref{varphi*}, where $*_7$ stands for the Hodge star operator on $V$, 
we have 
\[
\Phi = e^0 \wedge \varphi + \ast_7 \varphi. 
\]
Note that $\Phi$ is self dual, that is, $*_8 \Phi = \Phi$, where 
$*_8$ is the Hodge star operator on $W$. 
It is known that $\Phi$ uniquely determines an inner product $g_\Phi$ and the volume form 
and the subgroup of ${\rm GL}(W)$ preserving $\Phi$ is isomorphic to $\Sp$. 

Denote by $W_k$ the $k$-dimensional irreducible $\Sp$-module if there is a unique such module. 
For example, $W_8$ is the irreducible $8$-dimensional $\Sp$-module from above, and $W_8^* \cong W_8$. 
The group $\Sp$ acts irreducibly 
on $W_7 \cong \R^7$ as the double cover of ${\rm SO}(7)$. 
For its exterior powers, we obtain the decompositions 
\begin{equation}\label{decom1-L-W8}
\begin{aligned}
\Lambda^0 W^* &\cong \Lambda^8 W^* \cong W_1, \quad
& \Lambda^2 W^*  \cong \Lambda^6 W^* &\cong W_7 \oplus  W_{21},\\
\Lambda^1 W^* &\cong \Lambda^7 W^* \cong W_8, \quad
& \Lambda^3 W^* \cong \Lambda^5 W^* &\cong W_8 \oplus W_{48},\\
\Lambda^4 W^* &\cong W_1 \oplus W_7 \oplus W_{27} \oplus W_{35}
\end{aligned}
\end{equation}
where $\Lambda^k W^* \cong \Lambda^{8-k} W^*$ due to the $\Sp$-invariance of the Hodge isomorphism $\ast_8: \Lambda^k W^* \to \Lambda^{8-k} W^*$. Again, we denote by $\Lambda^k_\l W^* \subset \Lambda^k W^*$ the subspace isomorphic to $W_\l$ in the above notation.

The space $\Lambda^k_7 W^*$ for $k = 2,4,6$ 
is explicitly given as follows. 
For the explicit descriptions of the other irreducible summands, 
see for example \cite[(4.7)]{KLS}. 

\begin{lemma}[{\cite[Lemma 4.2]{KLS}, \cite[Lemma 3.3]{KYFM}}] \label{lem:lambdas}
Let $e^0 \in W^*$ be a unit vector. 
Set $V^* = (\R e^0)^\perp$, the orthogonal complement of $\R e^0$. 
The group $\Sp$ acts irreducibly on $V^*$ as the double cover of ${\rm SO}(7)$, 
and hence, we have the identification $V^* \cong W_7$. 
Then, the following maps are $\Sp$-equivariant isometries. 
\begin{equation} \label{def:lambda}
\lambda^k: V^* \longrightarrow \Lambda^k_7 W^*, \quad 
\begin{array}{ll} 
\lambda^2(\alpha) := \frac{1}{2} \left( e^0 \wedge \alpha + i (\alpha^\sharp) \varphi \right), \\[2mm]
\lambda^4(\alpha) := 
\frac{1}{\sqrt{8}} \left( e^0 \wedge i (\alpha^\sharp) \ast_7 \varphi - \alpha \wedge \varphi \right), \\[2mm]
\lambda^6(\alpha) := \frac{1}{3} \Phi \wedge \lambda^2(\alpha) = \ast_8 \lambda^2(\alpha). 
\end{array}
\end{equation}
Here, $\ast_8$ and $\ast_7$ are the Hodge star operators on $W^*$ and $V^*$, respectively.
\end{lemma}

By Lemma \ref{lem:lambdas}, we obtain the following.

\begin{lemma} \label{lem:lambda221}
$$
\Lambda^2_{21} W^* = \{ -e^0 \wedge *_7 (\beta \wedge *_7 \varphi) + \beta 
\mid \beta \in \Lambda^2 V^* \}. 
$$
\end{lemma}
\begin{proof}
Suppose that $e^0 \wedge \gamma + \beta \in \Lambda^2_{21} W^*$ 
for $\gamma \in V^*$ and $\beta \in \Lambda^2 V^*$. 
Then, for any $\alpha \in V^*$, we have 
$$
0
= \la e^0 \wedge \gamma + \beta, 2 \lambda^2(\alpha) \ra
= \la \gamma, \alpha \ra + \la \beta, i(\alpha^\sharp) \varphi \ra. 
$$
Since 
\begin{align*}
\la \beta, i(\alpha^\sharp) \varphi \ra
&= 
*_7 \left( \beta \wedge *_7 \left( i(\alpha^\sharp) \varphi \right) \right) \\
&=
*_7 \left( \beta \wedge \alpha \wedge *_7 \varphi \right)
=
\la \alpha, *_7 \left( \beta \wedge *_7 \varphi \right) \ra, 
\end{align*}
the proof is completed. 
\end{proof}

\begin{remark}
Since the group $\Sp$ acts irreducibly on $\Lambda^2 V^*$ as the double cover of ${\rm SO}(7)$, 
we have $\Lambda^2 V^* \cong W_{21}$. 
However, the map $\Lambda^2 V^* \ni \beta \mapsto -e^0 \wedge *_7 (\beta \wedge *_7 \varphi) + \beta 
\in \Lambda^2_{21} W^*$ is not $\Sp$-equivariant. 
The $\Sp$-equivariant isomorphism $\mu: \Lambda^2 V^* \rightarrow \Lambda^2_{21} W^*$ 
is given by 
$$
\mu (\beta) = -e^0 \wedge *_7 (\beta \wedge *_7 \varphi) - \beta + *_7 (\varphi \wedge \beta). 
$$
This can be proved by the same method as in the proof of 
Lemma 4.1 of the version 2 of the arXiv version of \cite{KLS} 
and using the action of an element of $\Sp \setminus G_2$. 
\end{remark}

We give some formulas about projections onto  some irreducible summands. 
Denote by 
\begin{align} \label{eq:proj}
\pi^k_\l : \Lambda^k W^* \rightarrow \Lambda^k_\l W^*
\end{align}
the canonical projection. 
When $k=2, 4, 6$ and $\l=7$, 
Lemma \ref{lem:lambdas} implies that 
\begin{equation} \label{eq:Spin7 proj 1}
\pi^k_\l (\alpha^k) = \sum_{\mu=1}^7 \la \alpha^k, \lambda^k (e^\mu) \ra \cdot \lambda^k (e^\mu)
\end{equation}
for $\alpha^k \in \Lambda^k W^*$, 
where $\{\, e^\mu\,\}_{\mu=1}^7$ is an orthonormal basis of $V^*$. 

We give other descriptions of $\pi^k_\l$ for $k=2$. 
Since the map 
$\Lambda^2 W^* \ni \alpha^2 \mapsto *_8 (\Phi \wedge \alpha^2) \in \Lambda^2 W^*$ 
is $\Sp$-equivariant, the simple computation and Schur's lemma give the following: 
\[
\begin{aligned}
\Lambda^2_7 W^* 
=& \{\, \alpha^2 \in \Lambda^2 W^* \mid \alpha^2 \wedge \Phi = 3 *_8 \alpha^2 \,\}, \\
\Lambda^2_{21} W^*
=& 
\{\, \alpha^2 \in \Lambda^2 W^* \mid \alpha^2 \wedge \Phi = - *_8 \alpha^2 \,\}.  
\end{aligned}
\]
Since 
$\alpha^2 = \pi^2_7 (\alpha^2) + \pi^2_{21} (\alpha^2)$ and 
$*_8 (\Phi \wedge \alpha^2) = 3 \pi^2_7 (\alpha^2) - \pi^2_{21} (\alpha^2)$
for a 2-form $\alpha^2 \in \Lambda^2 W^*$, it follows that 
\begin{equation*}
\pi^2_7 (\alpha^2) = \frac{\alpha^2 + *_8 (\Phi \wedge \alpha^2)}{4}, \quad
\pi^2_{21} (\alpha^2) = \frac{3 \alpha^2 - *_8 (\Phi \wedge \alpha^2)}{4}. 
\end{equation*}

Let $S^k W_\l$ be the space of symmetric $k$-tensors on $W_\l$. 
We use the following irreducible decompositions in this paper. 
\begin{align} \label{eq:decomp Spin7}
\Lambda^2 W_7    &= W_{21} = \sp, & 
\Lambda^2 W_{21} &= W_{21} \oplus W_{189}, \\
\nonumber
S^2 W_7 &= W_1 \oplus W_{27}, & 
S^2 W_8 &= W_1 \oplus W_{35}, \\
S^2 W_{21} &= W_1 \oplus W_{27} \oplus W_{35} \oplus W_{168}, & 
W_7 \otimes W_{21} &= W_7 \oplus W_{35} \oplus W_{105}. 
\nonumber
\end{align}

By \eqref{eq:decomp Spin7}, we can deduce some equations 
about the wedge product of two 2-forms. 
Indeed, by \eqref{eq:decomp Spin7} and Schur's lemma, 
we see that 
\begin{align} \label{eq:F2decomp}
\begin{split}
\Lambda^2_7 W^* \wedge \Lambda^2_7 W^* 
&\subset \Lambda^4_1 W^*  \oplus \Lambda^4_{27} W^*, \\
\Lambda^2_7 W^* \wedge \Lambda^2_{21} W^* 
&\subset \Lambda^4_7 W^*  \oplus \Lambda^4_{35} W^*, \\
\Lambda^2_{21} W^* \wedge \Lambda^2_{21} W^* 
&\subset \Lambda^4_1 W^*  \oplus \Lambda^4_{27} W^* \oplus \Lambda^4_{35} W^*. 
\end{split}
\end{align}
In fact, all of these inclusions are equalities. 
See \cite[Proposition 2.1]{MS2}.

We use the following lemmas in Appendix \ref{sec:new Spin7 str}. 
\begin{lemma} \label{lem:2form cong}
For any 2-form $F \in \Lambda^2 W^*$, 
there exists $h \in \Sp$ such that 
\begin{align} \label{eq:2form cong}
h^* F = 2 \lambda^2(\alpha) + 
\mu_1 e^{01} + \mu_2 e^{23} + \mu_3 e^{45} + \mu_4 e^{67}, 
\end{align}
where $\alpha = \alpha_1 e^1 + \alpha_3 e^3 + \alpha_5 e^5 + \alpha_7 e^7$, 
$\sum_{j=1}^4 \mu_j =0$ for 
$\alpha_i, \mu_j \in \R$ 
and $\mu_1 \leq \mu_2 \leq \mu_3 \leq \mu_4$.  
\end{lemma}

\begin{proof}
By the decomposition (\ref{decom1-L-W8}), set 
\[
F = F_7 + F_{21} = 2 \lambda^2 (\alpha') + F_{21} \in  \Lambda^2_7 W^* \oplus \Lambda^2_{21} W^*, 
\]
where $\alpha' \in V^* = (\R^7)^*$. 
Recall that every element in $\mathfrak{spin}(7)$ is ${\rm Ad}(\Sp)$-conjugate to an element of a Cartan subalgebra. Then, 
there exists $h' \in \Sp$ such that 
\[
(h')^* F_{21} = \mu_1 e^{01} + \mu_2 e^{23} + \mu_3 e^{45} + \mu_4 e^{67} 
\]
for $\mu_j \in \R$ such that $\sum_{j=1}^4 \mu_j =0$. 
Since we can permute $\{ \mu_j \}$ by the ${\rm SU}(4)$-action, 
we may assume that $\mu_1 \leq \mu_2 \leq \mu_3 \leq \mu_4$. 

Now, we show that there exists 
$h^{\prime\prime} \in \Sp$ such that 
$(h^{\prime\prime})^* (h')^* F_{21} = (h')^* F_{21}$ and  
$(h^{\prime\prime})^* (h')^* F_7 = 2 \lambda^2(\alpha)$ for 
$\alpha = \alpha_1 e^1 + \alpha_3 e^3 + \alpha_5 e^5 + \alpha_7 e^7$. 
Then, $h=h' h^{\prime\prime}$ satisfies the desired property.

Denote by $\om=e^{01}+e^{23}+e^{45}+e^{67}$ 
the standard K\"ahler form on $W$. 
Set 
$f^1=e^0+\i e^1, f^2=e^2+\i e^3, f^3=e^4+\i e^5, f^4=e^6+\i e^7$. 
Then, it is straightforward to see that 
\[
\begin{aligned}
2\lambda^2(e^1) &= \om, \\
2\left(\lambda^2(e^2) + \i \lambda^2(e^3) \right)&= f^{12} + \overline{f^{34}}, \\
2\left(\lambda^2(e^4) + \i \lambda^2(e^5) \right)&= f^{13} + \overline{f^{42}}, \\
2\left(\lambda^2(e^6) + \i \lambda^2(e^7) \right)&= f^{14} + \overline{f^{23}}. 
\end{aligned}
\]
Then, 
setting 
$
h^{\prime\prime} = 
{\rm diag} \left( e^{\i \theta_1}, e^{\i \theta_2}, e^{\i \theta_3}, e^{\i \theta_4} \right)
\in {\rm SU}(4) \subset \Sp 
$
for $\theta_j \in \R$ with $e^{\i (\theta_1 + \theta_2 + \theta_3 + \theta_4)} =1$, 
we see that $(h^{\prime\prime})^{*} (h')^* F_{21} = (h')^* F_{21}$, 
$(h^{\prime\prime})^* \lambda^2(e^1) = \lambda^2(e^1)$
and 
\[
\begin{aligned}
(h^{\prime\prime})^* \left( \lambda^2(e^2) + \i \lambda^2(e^3) \right) &= 
e^{\i (\theta_1 + \theta_2)} \left( \lambda^2(e^2) + \i \lambda^2(e^3) \right), \\
(h^{\prime\prime})^* \left( 
\lambda^2(e^4) + \i \lambda^2(e^5) \right) &= 
e^{\i (\theta_1 + \theta_3)} \left( \lambda^2(e^4) + \i \lambda^2(e^5) \right), \\
(h^{\prime\prime})^* \left(
\lambda^2(e^6) + \i \lambda^2(e^7) \right) &= 
e^{\i (\theta_1 + \theta_4)} \left( \lambda^2(e^6) + \i \lambda^2(e^7) \right). 
\end{aligned}
\]
In particular, given $\tau_1, \tau_2, \tau_3 \in \R$, set 
$
\theta_1 = (\tau_1+\tau_2+\tau_3)/2,  
\theta_2 = (\tau_1-\tau_2-\tau_3)/2,
\theta_3 = (-\tau_1+\tau_2-\tau_3)/2,
\theta_4 = (-\tau_1-\tau_2+\tau_3)/2.  
$
Then, $\sum_{j=1}^4 \theta_j =0$ and 
$
\left( e^{\i (\theta_1 + \theta_2)}, e^{\i (\theta_1 + \theta_3)}, 
e^{\i (\theta_1 + \theta_4)} \right)
= \left( e^{\i \tau_1}, e^{\i \tau_2}, e^{\i \tau_3} \right).  
$
This implies that 
$h^{\prime\prime}$ gives rotations on planes 
\[{\rm span} \{\lambda^2(e_2), \lambda^2(e_3) \}, \ {\rm span} \{\lambda^2(e_4), \lambda^2(e_5) \} \ \mbox{ and } \ {\rm span} \{\lambda^2(e_6), \lambda^2(e_7) \}. \]
Then, the proof is completed. 
\end{proof}

\begin{proposition} \label{prop:2form norm}
For any $\beta \in \Lambda^2_7 W^*$ and $\gamma \in \Lambda^2_{21} W^*$, we have 
\begin{align*}
\beta^3 &= \frac{3}{2} |\beta|^2 *_8 \beta, \\
|\beta|^4 &= \frac{2}{3} |\beta^2|^2, \\
|\gamma|^4 &= |\gamma^2|^2 - \frac{1}{3} *_8 \gamma^4, \\
|\beta|^2 |\gamma|^2 &= 2 |\beta \wedge \gamma|^2. 
\end{align*}
\end{proposition}

\begin{proof}
Set $\beta = 2 \lambda^2(\alpha)$, where 
$\lambda^2$ is defined in \eqref{def:lambda} and $\alpha \in V^*$. 
By the $\Sp$-action, we may assume that 
$\beta = 2 |\alpha| \lambda^2(e^1) = |\alpha| \om$, 
where $\om$ is the standard K\"ahler form on $W$. 
Note that $|\beta|=2 |\alpha|$. 
Then, 
$$
\beta^3=|\alpha|^3 \om^3 = 6 |\alpha|^3 *_8 \om = \frac{3}{2} |\beta|^2 *_8 \beta. 
$$

Since 
$\Lambda^4_1 W^*  \oplus \Lambda^4_7 W^* \oplus \Lambda^4_{27} W^*$ 
and $\Lambda^4_{35} W^*$ are the spaces of 
self dual 4-forms and anti self dual 4-forms, respectively, 
\eqref{eq:F2decomp} implies that $\beta^2$ is self dual. 
Then, $\beta^3 =  (3/2) |\beta|^2 *_8 \beta$ implies $|\beta|^4 = (2/3) |\beta^2|^2$.

Next, we prove the third equation. 
As in Lemma \ref{lem:2form cong}, we may assume that 
\[
\gamma = \mu_1 e^{01} + \mu_2 e^{23} + \mu_3 e^{45} + \mu_4 e^{67} 
\]
for $\mu_j \in \R$ such that $\sum_{j=1}^4 \mu_j =0$. 
Then, we have 
\[
\begin{aligned}
\gamma^2
=2 \left(\mu_1 \mu_2 e^{0123} + \mu_1 \mu_3 e^{0145} + \mu_1 \mu_4 e^{0167} 
+ \mu_2 \mu_3 e^{2345} 
+ \mu_2 \mu_4 e^{2367} + \mu_3 \mu_4 e^{4567} \right). 
\end{aligned}
\]
It is straightforward to obtain 
\[
\begin{aligned}
|\gamma^2|^2-|\gamma|^4
=&
4 (\mu_1^2 \mu_2^2 + \mu_1^2 \mu_3^2 + \mu_1^2 \mu_4^2 
+ \mu_2^2 \mu_3^2 + \mu_2^2 \mu_4^2 + \mu_3^2 \mu_4^2) - (\mu_1^2 + \mu_2^2 + \mu_3^2 + \mu_4^2)^2\\
=&
8 \mu_1 \mu_2 \mu_3 \mu_4
=
\frac{1}{3} *_8 \gamma^4, 
\end{aligned}
\]
where the second equality follows from $\sum_{j=1}^4 \mu_j =0$. 

Finally, we prove the fourth equation. As in Lemma \ref{lem:2form cong}, 
we may assume that 
$\beta = 2 \lambda^2(\alpha)$ and 
$\gamma = \mu_1 e^{01} + \mu_2 e^{23} + \mu_3 e^{45} + \mu_4 e^{67}$, 
where 
$\alpha = \alpha_1 e^1 + \alpha_3 e^3 + \alpha_5 e^5 + \alpha_7 e^7$, 
$\sum_{j=1}^4 \mu_j =0$ for 
$\alpha_i, \mu_j \in \R$. 
Then, since 
\begin{align*}
\beta \wedge \gamma 
=& \left( e^0 \wedge \alpha + i (\alpha^\sharp) \varphi \right) \wedge 
\left( \mu_1 e^{01} + \mu_2 e^{23} + \mu_3 e^{45} + \mu_4 e^{67} \right) \\
=&
e^0 \wedge 
\left( \alpha \wedge \left(\mu_2 e^{23} + \mu_3 e^{45} + \mu_4 e^{67} \right) 
+ \mu_1 e^1 \wedge i (\alpha^\sharp) \varphi 
\right) \\
&+ 
i (\alpha^\sharp) \varphi \wedge \left(\mu_2 e^{23} + \mu_3 e^{45} + \mu_4 e^{67} \right), 
\end{align*}
we have 
$$
|\beta \wedge \gamma|^2 = I_1 + I_2, 
$$
where 
\begin{align*}
I_1=& 
\left| \alpha \wedge \left(\mu_2 e^{23} + \mu_3 e^{45} + \mu_4 e^{67} \right)
+ \mu_1 e^1 \wedge i (\alpha^\sharp) \varphi 
\right|^2, \\
I_2=& \left| i (\alpha^\sharp) \varphi \wedge \left(\mu_2 e^{23} + \mu_3 e^{45} 
+ \mu_4 e^{67} \right) \right|^2.  
\end{align*}
Then, it is straightforward to obtain 
\begin{align*}
I_1 =& 
\alpha_1^2 
\left((\mu_1 + \mu_2)^2 + (\mu_1 + \mu_3)^2 + (\mu_1 + \mu_4)^2 \right)
+
\alpha_3^2 \left(2 \mu_1^2 + \mu_3^2 + \mu_4^2 \right) \\
&+ 
\alpha_5^2 \left(2 \mu_1^2 + \mu_2^2 + \mu_4^2 \right)
+
\alpha_7^2 \left(2 \mu_1^2 + \mu_2^2 + \mu_3^2 \right), \\
I_2 =& 
\alpha_1^2 
\left((\mu_2 + \mu_3)^2 + (\mu_2 + \mu_4)^2 + (\mu_3 + \mu_4)^2 \right)
+
\alpha_3^2 \left(2 \mu_2^2 + \mu_3^2 + \mu_4^2 \right) \\
&+ 
\alpha_5^2 \left(\mu_2^2 + 2 \mu_3^2 + \mu_4^2 \right)
+
\alpha_7^2 \left(\mu_2^2 + \mu_3^2 + 2 \mu_4^2 \right) \\
\end{align*}
and 
$$
|\beta \wedge \gamma|^2 
= I_1 + I_2
= 2|\alpha|^2 \sum_{j=1}^4 \mu_j^2 = \frac{1}{2} |\beta|^2 |\gamma|^2, 
$$
where we use $\sum_{j=1}^4 \mu_j =0$ several times. 
\end{proof}

By the fourth equation of Proposition \ref{prop:2form norm}, 
we obtain the following. This is also proved in \cite[Remark 2.2]{MS2}. 

\begin{corollary} \label{cor:2form norm}
For $\beta \in \Lambda^2_7 W^*$ and $\gamma \in \Lambda^2_{21} W^*$, 
$\beta \wedge \gamma=0$ implies 
$\beta =0$ or $\gamma=0$. 
\end{corollary}
%%%%%%%%%%%%%%%%%%%%%%%%%%%%%%%%%%%%%%%%%
\subsection{A manifold with a $\Sp$-structure} \label{sec:Spin7str}

\begin{definition}
Let $X$ be an oriented 8-manifold. A \emph{$\Sp$-structure on $X$} is a 
$4$-form $\Phi \in \Om^4$ such that at each $p \in X$ there is a positively oriented basis 
$\{ e_i \}_{i=0}^7$ of $T_p X$ such that $\Phi_p \in \Lambda^4 T^*_p X$ is of the form (\ref{Phi4}). 
The $4$-form $\Phi$ determines a unique Riemannian metric $g = g_\Phi$ on $X$ 
and the basis $\{ e_i \}_{i=0}^7$ is orthonormal with respect to $g$.

A $\Sp$-structure $\Phi$ is called \emph{torsion-free} if 
it is parallel with respect to the Levi-Civita connection of $g=g_\Phi$. 
\end{definition}

A manifold $X$ admits a $\Sp$-structure if and only if 
its frame bundle is reduced to the $\Sp$-subbundle. 
Hence considering its associated subbundle, we see that 
the $\Lambda^* T^* X$ has the same decomposition as (\ref{decom1-L-W8}).

The following case is important in this paper. 
For any $\sigma \in {\rm GL}(W)$, $\sigma^* \Phi$ defines a $\Sp$-structure on $W$. 
Then, $\sigma^* \Phi$ induces the decomposition 
$\Lambda^k W^* = \oplus_\l \Lambda^k_{\l, \sigma^* \Phi} W^*$ as in \eqref{decom1-L-W8}. 
More explicitly, we have 
$$
\Lambda^k_{\l, \sigma^* \Phi} W^* 
= 
\sigma^* \left( \Lambda^k_\l W^* \right). 
$$
Denote by $\pi^k_{\l, \sigma^* \Phi} : \Lambda^k W^* \rightarrow \Lambda^k_{\l, \sigma^* \Phi} W^*$ 
the projection. Using \eqref{eq:proj}, we can describe this more explicitly as 
\begin{align} \label{eq:proj2}
\pi^k_{\l, \sigma^* \Phi} = \sigma^* \circ \pi^k_\l \circ (\sigma^{-1})^*. 
\end{align}

%%%%%%%%%%%%%%%%%%%%%%%%%%%%%%%%%%%%%%%%%
\subsection{Flows of $\Sp$-structures} \label{sec:flow}

In this subsection, we consider the variation of $\Sp$-structures 
and deduce some formulas for Subsection \ref{sec:moduli generic}. 

First, define a $\Sp$-equivariant map 
$\ja: W^* \otimes W^* \to \Lambda^4 W^*$ by 
\begin{align} \label{eq:j1}
\ja (\alpha \otimes \beta) = \alpha \wedge i(\beta^\sharp) \Phi.  
\end{align}
Denote by $g, S^2 W^*, S^2_0 W^*$ 
the standard inner product on $W=\R^8$, 
the space of symmetric 2-tensors on $W^*$, 
the space of traceless symmetric 2-tensors on $W^*$, respectively. 
Then, we have 
\begin{align*}
W^* \otimes W^* &= S^2 W^* \oplus \Lambda^2 W^*, \\
S^2 W^* &= \R g \oplus S^2_0 W^* \cong W_1\oplus W_{35}, \\
\Lambda^2 W^* &= \Lambda^2_7 W^* \oplus \Lambda^2_{21} W^* 
\cong W_7\oplus W_{21}. 
\end{align*}
In particular, 
we have $0=\ja|_{\Lambda^2_{21} W^*}: \Lambda^2_{21} W^* \to \Lambda^4 W^*$. 
By \cite[Corollary 2.6]{Kar}, 
there are isomorphisms 
\begin{align*}
\ja|_{\R g}&: \R g \to \Lambda^4_1 W^*, \\
\ja|_{S^2_0 W^*}&: S^2_0 W^* \to \Lambda^4_{35} W^*, \\
\ja|_{\Lambda^2_7 W^*}&: \Lambda^2_7 W^* \to \Lambda^4_{7} W^*. 
\end{align*}
For $A \in W^* \otimes W^*$, 
define $A^\sharp \in {\rm End}(W) = W^* \otimes W$ and 
$A^* \in W \otimes W$ by 
\begin{align} \label{eq:Asharp}
g(A^\sharp (u),v) = A(u,v), \qquad A^* (u^\flat, v^\flat) = A(u,v)
\end{align}
for any $u,v \in W$. Note that 
$g^\sharp = \id_W$ and $g^*$ is the induced inner product on $W^*$ from $g$. 

Decompose $A \in W^* \otimes W^*$ as  
$$
A = \frac{1}{8} {\rm tr} (A^\sharp) g + A_0 +A_7 +A_{21} 
\in \R g \oplus S^2_0 W^* \oplus \Lambda^2_7 W^* \oplus \Lambda^2_{21} W^*. 
$$
We also decompose $B \in W^* \otimes W^*$ similarly. 
Then, by \cite[Proposition 2.5]{Kar}, we have 
\begin{align} \label{eq:j1 norm}
\la \ja (A), \ja (B) \ra = \frac{7}{2} {\rm tr} (A^\sharp) {\rm tr} (B^\sharp) 
+4 {\rm tr} (A_0^\sharp B_0^\sharp) 
-16 {\rm tr} (A_7^\sharp B_7^\sharp).  
\end{align}

Let $\Pp$ be the space of $\Sp$-structures on $W$. 
It is known that the tangent space $T_\Phi \Pp$ of $\Pp$ at $\Phi$ is given by 
$$
T_\Phi \Pp 
= \Lambda^4_1 W^* \oplus \Lambda^4_7 W^* \oplus \Lambda^4_{35} W^*
= \ja (S^2 W^* \oplus \Lambda^2_7 W^*). 
$$ 
Let $\{ \Phi_t \}_{t \in (- \epsilon, \epsilon)}$ be 
a path in $\Pp$ such that 
$$
\Phi_0 = \Phi
\quad  \mbox{and} \quad 
\left. \frac{d \Phi_t}{dt} \right|_{t=0} = \ja (A), 
$$
where $A = h + \beta \in S^2 W^* \oplus \Lambda^2_7 W^*$. 
Denote by $g_t, \vol_t$ the inner product and the volume form 
induced from $\Phi_t$, respectively. 
Set $\vol=\vol_0$. 
Define $g^*_t$ by \eqref{eq:Asharp} for $g_t$, 
where the operator $\flat$ is defined for $g$. 
By \cite[Proposition 3.1 and Corollary 3.2]{Kar}, we have 
\begin{align} \label{eq:diff etc}
\left. \frac{d g_t}{dt} \right|_{t=0} =2h, \quad
\left. \frac{d g^*_t}{dt} \right|_{t=0} =-2h^*, \quad
\left. \frac{d \vol_t}{dt} \right|_{t=0} ={\rm tr} (h^\sharp) \vol. 
\end{align}
Using these formula, 
we now derive the formula of the differential of Hodge star $*_t$ induced from $\Phi_t$. 

\begin{lemma} \label{lem:diff Hodge}
For any $1 \leq p \leq 8$ and $h \in S^2 W^*$, define 
$\jb (h): \Lambda^p W^* \to \Lambda^p W^*$ by 
\begin{align} \label{eq:j2}
\jb (h) (\alpha_1 \wedge \cdots \wedge \alpha_p) = 
\sum_{j=1}^p \alpha_1 \wedge \cdots \wedge \alpha_{j-1} \wedge h^*(\alpha_j, \,\cdot\,)^\flat 
\wedge \alpha_{j+1} \wedge \cdots \wedge \alpha_p, 
\end{align}
where $\alpha_1, \cdots, \alpha_p \in W^*$. 
Then, for $\beta \in \Lambda^p W^*$, we have 
$$
\left. \frac{d}{dt} *_t \beta \right|_{t=0} = * \left( -2 \jb (h) (\beta) + {\rm tr}(h^\sharp) \beta \right). 
$$
\end{lemma}

\begin{proof}
Let $\{ e_i \}_{i=0}^7$ be the orthonormal basis of $W$ 
with respect to $g=g_\Phi$. Denote by $\{ e^i \}_{i=0}^7$ its dual. 
Set 
$$
*_t \beta = \frac{1}{p!} \sum_{i_1, \cdots, i_p} \beta_{i_1 \cdots i_p} * e^{i_1 \cdots i_p}, 
$$
where $\beta_{i_1 \cdots i_p} = \la *_t \beta, * e^{i_1 \cdots i_p} \ra$. 
Then, since 
$$
\la e^{i_1 \cdots i_p}, \beta \ra_t \vol_t
=
e^{i_1 \cdots i_p} \wedge *_t \beta = \beta_{i_1 \cdots i_p} \vol, 
$$
we have 
$$
*_t \beta = \frac{1}{p!} \sum_{i_1, \cdots, i_p} 
* \left( \la e^{i_1 \cdots i_p}, \beta \ra_t \vol_t \right)
* e^{i_1 \cdots i_p}. 
$$
Then, by \eqref{eq:diff etc}, it follows that 
\begin{align} \label{eq:diff Hodge1}
\left. \frac{d}{dt} *_t \beta \right|_{t=0} 
= 
\frac{1}{p!} \sum_{i_1, \cdots, i_p} 
\left. \frac{d}{dt}  \la e^{i_1 \cdots i_p}, \beta \ra_t \right|_{t=0} * e^{i_1 \cdots i_p}
+ 
{\rm tr}(h^\sharp) * \beta. 
\end{align}
Without loss of generality, we may 
assume that $\beta = \beta_1 \wedge \cdots \wedge \beta_p$ for $\beta_j \in W^*$. 
Set 
$$
v_{j,t} = {}^t\! \left( g_t^* (e^{i_1}, \beta_j), \cdots, g_t^* (e^{i_p}, \beta_j) \right) \in \R^p, \qquad 
v_j = v_{j,0}.
$$ 
Then, by \eqref{eq:diff etc} again, we obtain  
\begin{align*}
\left. \frac{d}{dt}  \la e^{i_1 \cdots i_p}, \beta \ra_t \right|_{t=0} 
&=
\left. \frac{d}{dt}  \det (v_{1,t} \cdots v_{p,t}) \right|_{t=0} \\
&=
\sum_j \det \left(v_{1} \cdots \left. \frac{d}{dt} v_{j,t} \right|_{t=0}  \cdots v_{p} \right) \\
&=
\sum_j 
\left \la e^{i_1 \cdots i_p}, \beta_1 \wedge \cdots \wedge (-2 h^* (\beta_j, \cdot))^\flat \wedge 
\cdots \wedge \beta_p \right \ra \\
&=
-2 \left \la e^{i_1 \cdots i_p}, \jb (h) (\beta) \right \ra. 
\end{align*}
This together with \eqref{eq:diff Hodge1} implies Lemma \ref{lem:diff Hodge}. 
\end{proof}

We prove the following lemmas 
for the application in Subsection \ref{sec:moduli generic}. 
\begin{lemma} \label{lem:j2ad}
For any $1 \leq p \leq 8$, define an ${\rm O}(8)$-equivariant map 
$\jc: \Lambda^p W^* \otimes \Lambda^p W^* \to S^2 W^*$ by 
\begin{align} \label{eq:j3}
\begin{split}
&\jc \left(\alpha_1 \wedge \cdots \wedge \alpha_p 
\otimes \beta_1 \wedge \cdots \wedge \beta_p \right) \\
=& 
\sum_{j,k=1}^p (-1)^{j+k}
\la 
\alpha_1 \wedge \stackrel{j}{\hat \cdots} \wedge \alpha_p, 
\beta_1 \wedge \stackrel{k}{\hat \cdots}  \wedge \beta_p 
\ra 
( \alpha_j \otimes \beta_k + \beta_k \otimes \alpha_j), 
\end{split}
\end{align}
where we set 
$\alpha_1 \wedge \stackrel{j}{\hat \cdots} \wedge \alpha_p 
= \alpha_1 \wedge \cdots \wedge \alpha_{j-1} \wedge \alpha_{j+1} \wedge \cdots \wedge \alpha_p. 
$
Then, for any $h \in S^2_0 W^*$ and $\alpha, \beta \in \Lambda^p W^*$, we have 
$$
\la \jb (h) (\alpha), \beta \ra = \frac{1}{8} \left \la \ja (h), (\ja \circ \jc) (\alpha \otimes \beta) \right \ra. 
$$
\end{lemma}

\begin{proof}
First, we prove this for $p=1$. 
By the definition of $\jb$ in Lemma \ref{lem:diff Hodge}, we have 
$
\la \jb (h) (\alpha), \beta \ra = h^*(\alpha,\beta). 
$
For $\alpha, \beta \in W^*$, set 
$$
B= \alpha \otimes \beta + \beta \otimes \alpha 
= \frac{1}{8} {\rm tr} (B^\sharp) g + B_0 \in \R g \oplus S^2_0 W^*. 
$$
Since $B^\sharp = \alpha \otimes \beta^\sharp + \beta \otimes \alpha^\sharp$, 
it follows that 
$$
{\rm tr}(h^\sharp B^\sharp) 
= 
{\rm tr}(\alpha \otimes h(\beta^\sharp, \cdot)^\sharp + \beta \otimes h(\alpha^\sharp,  \cdot)^\sharp) 
= 
2 h( \alpha^\sharp, \beta^\sharp)
= 
2 h^* (\alpha, \beta). 
$$
Since $h \in S^2_0 W^*$, we have 
${\rm tr}(h^\sharp B^\sharp) = {\rm tr}(h^\sharp B^\sharp_0). $
Then, by \eqref{eq:j1 norm}, it follows that 
$
{\rm tr}(h^\sharp B^\sharp_0) = \la \ja (h), \ja (B) \ra/4.  
$
Hence, we obtain 
\begin{align} \label{eq:j2ad 1}
h^*(\alpha,\beta) 
= \frac{1}{8} \la \ja (h), \ja (B) \ra
= \frac{1}{8} \la \ja (h), \ja (\alpha \otimes \beta + \beta \otimes \alpha) \ra, 
\end{align}
which implies Lemma \ref{lem:j2ad} for $p=1$. 

Next, suppose that 
$\alpha = \alpha_1 \wedge \cdots \wedge \alpha_p$ and 
$\beta = \beta_1 \wedge \cdots \wedge \beta_p$ for $\alpha_j, \beta_k \in W^*$. 
Then, we have 
\begin{align*}
\la \jb (h) (\alpha), \beta \ra 
=&
\sum_{j=1}^p 
\la \alpha_1 \wedge \cdots \wedge h^*(\alpha_j, \,\cdot\,)^\flat \wedge \cdots \wedge \alpha_p, 
\beta \ra \\
=&
\sum_{j=1}^p (-1)^{j-1}
\la \alpha_1 \wedge \stackrel{j}{\hat \cdots} \wedge \alpha_p, 
i \left( h^*(\alpha_j, \,\cdot\,) \right) \beta \ra \\
=&
\sum_{j,k=1}^p (-1)^{j+k}
h^* (\alpha_j, \beta_k) 
\la 
\alpha_1 \wedge \stackrel{j}{\hat \cdots} \wedge \alpha_p, 
\beta_1 \wedge \stackrel{k}{\hat \cdots}  \wedge \beta_p 
\ra.  
\end{align*}
Then, by \eqref{eq:j2ad 1}, the proof is completed. 
\end{proof}

\begin{lemma} \label{lem:pi35}
For any $\gamma, \gamma' \in \Lambda^2_7 W^*$ and $\delta \in \Lambda^2_{21} W^*$, we have 
\begin{align*}
(\ja \circ \jc) (* \gamma \otimes * \gamma') &=
6 \la \gamma, \gamma' \ra \Phi, \\
(\ja \circ \jc) (* \gamma \otimes * \delta) = 
(\ja \circ \jc) (* \delta \otimes * \gamma) 
&= -4 \pi^4_{35} (\gamma \wedge \delta). 
\end{align*}
\end{lemma}

\begin{proof}
First note that 
$$
{\bf k}: 
\Lambda^2_7 W^* \otimes \Lambda^2_7 W^* \to S^2 W^*, \qquad 
\gamma \otimes \gamma' \mapsto \jc (* \gamma \otimes * \gamma')
$$
is $\Sp$-equivariant. 
Since 
\begin{align} \label{eq:pi35 1}
\Lambda^2_7 W^* \otimes \Lambda^2_7 W^* 
=W_1 \oplus W_{21} \oplus W_{27}
\end{align} 
and 
$S^2 W^* = \R g \oplus S^2_0 W^* = W_1 \oplus W_{35}$ by \eqref{eq:decomp Spin7}, 
the image of ${\bf k}$ is contained in $\R g$. 
Since $\ja: W^* \otimes W^* \to \Lambda^4 W^*$ is also $\Sp$-equivariant 
and $\Lambda^4 W^* = W_1 \oplus W_7 \oplus W_{27} \oplus W_{35}$ by \eqref{decom1-L-W8}, 
it follows that 
the image of $\ja \circ {\bf k}$ is contained in $\Lambda^4_1 W^*$.

By \eqref{eq:pi35 1}, the space of $\Sp$-equivariant linear maps from 
$\Lambda^2_7 W^* \otimes \Lambda^2_7 W^*$ to $\Lambda^4_1 W^*$ 
is 1-dimensional.  
Then, there exists $C \in \R$, 
we have 
$$
(\ja \circ \jc) (* \gamma \otimes * \gamma')
=
C \la \gamma, \gamma' \ra \Phi  
$$
for any $\gamma, \gamma' \in \Lambda^2_7 W^*$. 

Now suppose that $\gamma = \gamma' = e^{01} +e^{23} +e^{45} +e^{67}$. 
Then, 
$$
\la \gamma, \gamma' \ra \Phi = 4 \Phi. 
$$
Since 
$* \gamma = * \gamma' = e^{012345} + e^{012367} + e^{014567} + e^{234567}$, 
it is straightforward to obtain 
\begin{align*}
\jc (* \gamma \otimes * \gamma')
=&
\jc (e^{012345} \otimes e^{012345}) 
+\jc (e^{012367} \otimes e^{012367}) \\
&+\jc (e^{014567} \otimes e^{014567}) 
+ \jc (e^{234567} \otimes e^{234567}) 
\end{align*}
and 
\begin{align*}
\jc (e^{012345} \otimes e^{012345}) &= 2 \sum_{i \neq 6,7} e^i \otimes e^i, &
\jc (e^{012367} \otimes e^{012367}) &= 2 \sum_{i \neq 4,5} e^i \otimes e^i,\\
\jc (e^{014567} \otimes e^{014567}) &= 2 \sum_{i \neq 2,3} e^i \otimes e^i, &
\jc (e^{234567} \otimes e^{234567}) &= 2 \sum_{i \neq 0,1} e^i \otimes e^i. 
\end{align*}
Then, 
$\jc (* \gamma \otimes * \gamma') = 6 \sum_i e^i \otimes e^i = 6 g$ and 
\begin{align*}
(\ja \circ \jc) (* \gamma \otimes * \gamma')
= 24 \Phi.
\end{align*}
Hence, we obtain $C=6$.

Similarly, note that 
$$
{\bf k'}: 
\Lambda^2_7 W^* \otimes \Lambda^2_{21} W^* \to S^2 W^*, \qquad 
\gamma \otimes \delta \mapsto \jc (* \gamma \otimes * \delta)
$$
is $\Sp$-equivariant. 
Since 
\begin{align} \label{eq:pi35 2}
\Lambda^2_7 W^* \otimes \Lambda^2_{21} W^* 
=W_7 \otimes W_{21} = W_7 \oplus W_{35} \oplus W_{105}
\end{align} 
and 
$S^2 W^* = \R g \oplus S^2_0 W^* = W_1 \oplus W_{35}$ by \eqref{eq:decomp Spin7}, 
the image of ${\bf k'}$ is contained in $S^2_0 W^*$. 
Since $\ja: W^* \otimes W^* \to \Lambda^4 W^*$ is also $\Sp$-equivariant 
and $\Lambda^4 W^* = W_1 \oplus W_7 \oplus W_{27} \oplus W_{35}$ by \eqref{decom1-L-W8}, 
it follows that 
the image of $\ja \circ {\bf k'}$ is contained in $\Lambda^4_{35} W^*$.

By \eqref{eq:pi35 2}, the space of $\Sp$-equivariant linear maps from 
$\Lambda^2_7 W^* \otimes \Lambda^2_{21} W^*$ to $\Lambda^4_{35} W^*$ 
is 1-dimensional.  
Then, there exists $C' \in \R$, 
we have 
$$
(\ja \circ \jc) (* \gamma \otimes * \delta)
=
C' \pi^4_{35} (\gamma \wedge \delta) 
$$
for any $\gamma \in \Lambda^2_7 W^*$ and $\delta \in \Lambda^2_{21} W^*$. 

Now suppose that $\gamma = e^{01} +e^{23} +e^{45} +e^{67}$ 
and $\delta = e^{01} -e^{23}$. Then, 
$$
\pi^4_{35} (\gamma \wedge \delta) 
= \gamma \wedge \delta
= (e^{01} -e^{23}) \wedge (e^{45} +e^{67}). 
$$
Since 
$* \gamma = e^{012345} + e^{012367} + e^{014567} + e^{234567}$
and 
$* \delta = - e^{014567} + e^{234567}$, 
it is straightforward to obtain 
\begin{align*}
\jc (* \gamma \otimes * \delta)
=&
- \jc (e^{014567} \otimes e^{014567}) 
+ \jc (e^{234567} \otimes e^{234567}) \\
=&
2 (-e^0 \otimes e^0 - e^1 \otimes e^1 + e^2 \otimes e^2 + e^3 \otimes e^3)
\end{align*}
and 
\begin{align*}
(\ja \circ \jc) (* \gamma \otimes * \delta)
=
4 (-e^{0145}-e^{0167}+e^{2345}+e^{2367}). 
\end{align*}
Hence, we obtain $C'=-4$. 
\end{proof}
%%%%%%%%%%%%%%%%%%%%%%%%%%%%%%%%%%%%%%%%%%%%%%%%%%%%%%%%%
%%%%%%%%%%%%%%%%%%%%%%%%%%%%%%%%%%%%%%%%%%%%%%%%%%%%%%%%%
%%%%%%%%%%%%%%%%%%%%%%%%%%%%%%%%%%%%%%%%%%%%%%%%%%%%%%%%%
%%%%%%%%%%%%%%%%%%%%%%%%%%%%%%%%%%%%%%%%%%%%%%%%%%%%%%%%%
%%%%%%%%%%%%%%%%%%%%%%%%%%%%%%%%%%%%%%%%%%%%%%%%%%%%%%%%%
%%%%%%%%%%%%%%%%%%%%%%%%%%%%%%%%%%%%%%%%%%%%%%%%%%%%%%%%%
%%%%%%%%%%%%%%%%%%%%%%%%%%%%%%%%%%%%%%%%%%%%%%%%%%%%%%%%%
%%%%%%%%%%%%%%%%%%%%%%%%%%%%%%%%%%%%%%%%%%%%%%%%%%%%%%%%%
%%%%%%%%%%%%%%%%%%%%%%%%%%%%%%%%%%%%%%%%%%%%%%%%%%%%%%%%%
%%%%%%%%%%%%%%%%%%%%%%%%%%%%%%%%%%%%%%%%%%%%%%%%%%%%%%%%%
%%%%%%%%%%%%%%%%%%%%%%%%%%%%%%%%%%%%%%%%%%%%%%%%%%%%%%%%%
\section{Some properties of ${\rm Spin}(7)$-dDT connections} \label{sec:suggestSpin7dDT}
We continue to use the notation (and identities) of Subsection \ref{sec:Spin7 geometry}. 
Let 
$X^8$ be a compact connected 8-manifold with a ${\rm Spin}(7)$-structure $\Phi$ 
and $L \to X$ be a smooth complex line bundle with a Hermitian metric $h$.
Set 
\[
\begin{aligned}
\mathcal{A}_{0}=&\{\,\nabla \mid \mbox{a Hermitian connection of }(L,h) \,\} \\
=& \nabla + \i \Om^1 \cdot \id_L, 
\end{aligned}
 \]
where $\nabla \in \Aa_{0}$ is any fixed connection. 
We regard the curvature 2-form $F_\n$ of $\n$ as a $\i \R$-valued closed 2-form on $X$. 

Define maps 
$\Ff^1_{{\rm Spin}(7)}:\Aa_{0} \rightarrow \i \Om^{2}_7$ and 
$\Ff^2_{{\rm Spin}(7)}:\Aa_{0} \rightarrow \Om^{4}_7$ by 
\begin{equation*}
\begin{aligned}
\Ff^{1}_{{\rm Spin}(7)}(\nabla)
&= 
\pi^2_{7} \left( F_\nabla + \frac{1}{6} * F_\nabla^3 \right) \\
&= 
\frac{1}{4}
\left(
F_\nabla + \frac{1}{6} * F_\nabla^3 + 
* \left( \left( F_\nabla + \frac{1}{6} * F_\nabla^3  \right) \wedge \Phi \right) 
\right), \\
\Ff^{2}_{{\rm Spin}(7)}(\nabla)
&=
\pi^{4}_{7}(F_{\nabla}^2). 
\end{aligned}
\end{equation*}

\begin{definition}\label{def:Spin7dDT}
A Hermitian connection $\nabla$ of $(L,h)$ satisfying 
$$
\Ff^{1}_{{\rm Spin}(7)}(\nabla)=0 \quad \mbox{and} \quad \Ff^{2}_{{\rm Spin}(7)}(\nabla)=0 
$$
is called a \emph{deformed Donaldson--Thomas connection}. 
For the rest of this paper, 
we call this a $\Sp$-dDT connection for short. 
\end{definition}

\begin{remark}
Though the dDT connection for a ${\rm Spin}(7)$-manifold was first introduced in \cite[Section 4.2.1]{LL} 
as a Hermitian connection $\nabla$ satisfying $*F_\nabla + F\wedge \Phi+F_\nabla^3/6=0$, 
we use the definition above introduced in \cite[Definition 1.3]{KYFM}. 
\end{remark}

As a generalization of the last statement of \cite[Chapter I\hspace{-.1em}V,Theorem 2.20]{HL}, 
we can show the following. 
The proof is given in Proposition \ref{prop:1+F 00 Spin7}.
Thus, if we assume $* F^4_\n/24 \neq 1$, the theory of $\Sp$-dDT connections 
becomes more tractable. 

\begin{proposition} \label{prop:F1toF2}
For $\n \in \Aa_0$ satisfying $* F_\n^4/24 \neq 1$, 
$\Ff^1_{{\rm Spin}(7)} (\n) = 0$
implies 
$\Ff^2_{{\rm Spin}(7)} (\n) = 0$. 
\end{proposition}

The following is an analogue of \cite{KY}. 
Set $(F_\nabla)_j = \pi^2_j (F_\n) \in \i \Om^2_j$ for $j=7,21$. 
If $\n$ is a $\Sp$-dDT connection, 
the norms of $(F_\nabla)_7$ and $(F_\nabla)_{21}$ satisfy the following relation. 
Though we do not use Proposition \ref{prop:dDT norm} in this paper, 
we believe that this decomposition will be helpful 
for the further research of $\Sp$-dDT connections. 
The statement follows from Corollary \ref{cor:dDT estimate}.

\begin{proposition} \label{prop:dDT norm}
Suppose that $\n$ is a $\Sp$-dDT connection. 
\begin{enumerate}
\item
If $(F_\nabla)_{21}=0$, we have $(F_\nabla)_7=0$ or $|(F_\nabla)_7| = 2$. 
\item
We have 
\[
|(F_\nabla)_7| \leq 2 \sqrt{\frac{|(F_\nabla)_{21}|^2 + 4}{3}} 
\cos \left( \frac{1}{3} \arccos \left( \frac{|(F_\nabla)_{21}|^3}{(|(F_\nabla)_{21}|^2+4)^{3/2}} \right) \right). 
\]
\end{enumerate}
\end{proposition}

Now, we define the moduli space of $\Sp$-dDT connections. 
Let $\Gg_U$
be the group of unitary gauge transformations of $(L,h)$. 
Precisely, 
\[
\Gg_U= \{\, f \cdot \id_L \mid f: X \rightarrow \C, \ |f|=1 \,\} \cong C^\infty(X, S^1).
\]
The action $\Gg_U \times \Aa_0 \rightarrow \Aa_0$ is defined by 
$(\lambda, \nabla) \mapsto \lambda^{-1} \circ \nabla \circ \lambda$.
When $\lambda=f \cdot \id_L$, we have 
\begin{equation}\label{finvdf}
\lambda^{-1} \circ \nabla \circ \lambda = \nabla + f^{-1}df \cdot \id_L. 
\end{equation}
Thus, the $\Gg_U$-orbit through $\n \in \Aa_0$ is given by 
\begin{align} \label{eq:Gu orbit Spin7}
\nabla + \{ f^{-1} d f \in \i \Om^1 \mid f: X \rightarrow \C, \ |f|=1 \} \cdot \id_L. 
\end{align}
Since the curvature 2-form $F_\nabla$ is invariant under the action of $\Gg_U$, 
the maps $\Ff^1_\Sp$ and $\Ff^2_\Sp$ reduces to 
\[
\left( \underline{\Ff^1_\Sp},\underline{\Ff^2_\Sp} \right): 
\Aa_{0}/\Gg_U \rightarrow \i \Om^6_7 \oplus \Om^4_7. 
\]

\begin{definition}
The {\em moduli space} $\Mm$ of deformed Donaldson-Thomas connections of $(L,h)$ 
is given by 
\begin{align*}
\Mm_\Sp 
=& \left( \underline{\Ff^1_\Sp} \right)^{-1}(0) \cap 
\left(\underline{\Ff^2_\Sp} \right)^{-1}(0) \\
=& \left( (\Ff^1_\Sp)^{-1}(0) \cap (\Ff^2_\Sp)^{-1}(0) \right)/\Gg_U. 
\end{align*}
\end{definition}

\begin{remark} \label{rem:metric B0}
As \cite{KY}, 
there is a metric on $\Bb_0 = \Aa_0 /\Gg_U$ 
whose induced topology agrees with the $C^\infty$ (quotient) topology on $\Bb_0$. 
In particular, $\Bb_0$ (with the $C^\infty$ topology) is paracompact and Hausdorff. 
Similarly, $\Mm_{\Sp}$ (with the induced topology from $\Bb_0$)
is also metrizable, and hence, it is paracompact and Hausdorff. 
\end{remark}

By Proposition \ref{prop:F1toF2}, 
if we assume $* F^4_\n/24 \neq 1$, 
we only have to consider $\Ff^1_{{\rm Spin}(7)} (\n) = 0$, 
which makes the deformation theory of $\Sp$-dDT connections more tractable. 
In Section \ref{sec:defofSpin7dDT}, we focus on this case and study their deformations. 
The technical difficulty in the general case is explained in Remark \ref{rem:prob}.

%%%%%%%%%%%%%%%%%%%%%%%%%%%%%%%%%%%%%%%%%%%%%%%%%%%%%%%%%
%%%%%%%%%%%%%%%%%%%%%%%%%%%%%%%%%%%%%%%%%%%%%%%%%%%%%%%%%
%%%%%%%%%%%%%%%%%%%%%%%%%%%%%%%%%%%%%%%%%%%%%%%%%%%%%%%%%
%%%%%%%%%%%%%%%%%%%%%%%%%%%%%%%%%%%%%%%%%%%%%%%%%%%%%%%%%
%%%%%%%%%%%%%%%%%%%%%%%%%%%%%%%%%%%%%%%%%%%%%%%%%%%%%%%%%
%%%%%%%%%%%%%%%%%%%%%%%%%%%%%%%%%%%%%%%%%%%%%%%%%%%%%%%%%
%%%%%%%%%%%%%%%%%%%%%%%%%%%%%%%%%%%%%%%%%%%%%%%%%%%%%%%%%
%%%%%%%%%%%%%%%%%%%%%%%%%%%%%%%%%%%%%%%%%%%%%%%%%%%%%%%%%
%%%%%%%%%%%%%%%%%%%%%%%%%%%%%%%%%%%%%%%%%%%%%%%%%%%%%%%%%
%%%%%%%%%%%%%%%%%%%%%%%%%%%%%%%%%%%%%%%%%%%%%%%%%%%%%%%%%
%%%%%%%%%%%%%%%%%%%%%%%%%%%%%%%%%%%%%%%%%%%%%%%%%%%%%%%%%
\section{Deformations of ${\rm Spin}(7)$-dDT connections $\n$ with $* F_\n^4/24 \neq 1$  
}\label{sec:defofSpin7dDT}
Let $X^8$ be a compact connected 8-manifold with a ${\rm Spin}(7)$-structure $\Phi$ 
and $L \to X$ be a smooth complex line bundle with a Hermitian metric $h$.
In this section, 
we study deformations of $\Sp$-dDT connections $\n$ satisfying $* F_\n^4/24 \neq 1$. 
Set 
\begin{align} \label{eq:A'}
\mathcal{A}'_{0}=\{\,\nabla \in \Aa_0 \mid * F_\n^4/24 \neq 1 \}. 
\end{align}
Define a map $\Ff_{{\rm Spin}(7)}:\Aa'_0 \rightarrow \i \Om^2_7$ by 
\begin{align} \label{eq:deform map Spin7}
\Ff_{{\rm Spin}(7)} = \Ff^1_{{\rm Spin}(7)}|_{\Aa'_0}, \qquad
\nabla \mapsto \pi^2_7 \left( F_\nabla + \frac{1}{6} * F_\nabla^3 \right). 
\end{align}
Each element of $\Ff_{{\rm Spin}(7)}^{-1}(0)$ is a $\Sp$-dDT connection with  $* F_\n^4/24 \neq 1$ 
by Proposition \ref{prop:F1toF2}. 
The unitary gauge group $\Gg_U$ also acts on $\Aa_0'$, 
and 
$\Ff_{{\rm Spin}(7)}$ reduces to 
\[
\underline{\Ff_{{\rm Spin}(7)}}: \Aa'_{0}/\Gg_U \rightarrow \i \Om^2_7. 
\]
Set 
\[
\Mm'_{{\rm Spin}(7)} = \underline{\Ff_{{\rm Spin}(7)}}^{-1}(0) = \Ff_{{\rm Spin}(7)}^{-1}(0)/\Gg_U. 
\]
Then, by Proposition \ref{prop:F1toF2}, we have an inclusion 
$$
\Mm'_{{\rm Spin}(7)} = \Mm_{{\rm Spin}(7)} \cap (\Aa'_0/\Gg_U) \hookrightarrow \Mm_{{\rm Spin}(7)}. 
$$

\subsection{The differential complexes}
In this subsection, before considering deformations of $\Sp$-dDT connections, 
we introduce differential complexes that are important in this paper. 

Let $\Psi \in \Om^4$ be any $\Sp$-structure on $X$. 
The $\Sp$-structure $\Psi$ induces the decomposition of 
$\Om^\bullet$ as in \eqref{decom1-L-W8}: 
$\Om^2 = \Om^2_{7, \Psi} \oplus \Om^2_{21, \Psi}$ etc. 
Let $\pi^k_{\l, \Psi}: \Om^k \rightarrow \Om^k_{\l, \Psi}$ be the projection. 
Set 
\begin{equation}\label{cancpx}
\begin{aligned}
0 \rightarrow \i \Om^0 
\stackrel{d} \rightarrow \i \Om^1 
\stackrel{\pi^2_{7, \Psi} \circ d} \longrightarrow \i \Om^2_{7, \Psi}
\rightarrow 0, 
\end{aligned}
\tag{$\flat_{\Psi}$} 
\end{equation}
which is the canonical complex introduced by \cite{Reyes} multiplied by $\i$.  
It is known that \eqref{cancpx} is elliptic for any $\Sp$-structure $\Psi$ by \cite[Lemma 5]{Reyes}. 
Denote by 
\begin{align} \label{cancpx2}
D \eqref{cancpx} = (\pi^2_{7, \Psi} \circ d, d^*): \i \Om^1 \rightarrow \i \Om^2_{7, \Psi} \oplus \i \Om^0
\end{align}
the associated two term elliptic complex. 
Note that $d^*$ is the codifferential with respect to the metric 
induced from the initial $\Sp$-structure $\Phi$. 
The index ${\rm ind} D \eqref{cancpx}$ of $D \eqref{cancpx}$ 
is given by the minus of the index of \eqref{cancpx}, 
that is, 
\begin{align*}
{\rm ind} D \eqref{cancpx} 
=& - \dim H^0 \eqref{cancpx} + \dim H^1 \eqref{cancpx} - \dim H^2 \eqref{cancpx} \\
=& - 1 + \dim H^1 \eqref{cancpx} - \dim H^2 \eqref{cancpx}, 
\end{align*}
where $H^i \eqref{cancpx}$ is the $i$-th cohomology of the complex \eqref{cancpx}.

%%%%%%%%%%%%%%%%%%%%%%%%%%%%%%%%%%%%%%%%%%%%%%%%%%%%%%%%%
\subsection{The infinitesimal deformation} \label{sec:infi deform}
Let $X^8$ be a compact connected 8-manifold with a ${\rm Spin}(7)$-structure $\Phi$. 
For the rest of Section \ref{sec:defofSpin7dDT}, we suppose that $\Mm'_{\Sp} \neq \emptyset$. 
In this subsection, we study the infinitesimal deformation of $\Sp$-dDT connections. 

To study the infinitesimal deformation is highly nontrivial. 
The key is Theorem \ref{thm:1+F Spin7 neq}, which enables 
us to describe the linearization of $\Ff_\Sp$ ``nicely'' 
in terms of a new $\Sp$-structure defined by $\Phi$ and $\n$. 
Then, we can control the deformations of $\Sp$-dDT connections by 
the complex \eqref{cpxfordDT} in this subsection.

First, we show the following lemma. 
Set $\Om^k_\l = \Om^k_{\l, \Phi}$ and $\pi^k_\l =\pi^k_{\l, \Phi} :\Om^k \to \Om^k_\l$ for simplicity. 

\begin{lemma} \label{lem:decomp om27}
Set 
$$
\Hh^2_7= \{\, \beta \in \Om^2_7 \mid \Delta \beta =0 \,\}, 
$$
where $\Delta = dd^*+d^*d$ is the Laplacian. 
Then, there exists a subspace $\Vv \subset \Om^2_7$ and 
we have the orthogonal decomposition  
\begin{align} \label{eq:decomp om27}
\Om^2_7=\Hh^2_7 \oplus \Vv 
\end{align}
with respect to the $L^2$ inner product. 
\end{lemma}

\begin{proof}
By the ellipticity of the canonical complex, 
we have the following $L^2$ orthogonal decomposition: 
\begin{align} \label{eq:decomp om27 1}
\Om^2_7 = \ker (\pi^2_7 \circ d)^* \oplus \pi^2_7 (d \Om^1), 
\end{align}
where $(\pi^2_7 \circ d)^*$ is the formal adjoint of $\pi^2_7 \circ d$. 
For any $\beta \in \Hh^2_7$ and $\gamma \in \Om^1$, we have 
$$
\la \beta, (\pi^2_7 \circ d)(\gamma) \ra_{L^2}
=
\la \beta, d \gamma \ra_{L^2}
=0, 
$$
where $\la \,\cdot\,, \,\cdot\, \ra_{L^2}$ is the $L^2$ inner product. 
This implies that $\Hh^2_7 \subset \ker (\pi^2_7 \circ d)^*$. 
Since $\ker (\pi^2_7 \circ d)^*$ is finite dimensional, 
there exists a subspace $\Vv'$ so that 
we have an $L^2$ orthogonal decomposition
$$
\ker (\pi^2_7 \circ d)^* = \Hh^2_7 \oplus \Vv'. 
$$
Thus, setting $\Vv = \Vv' \oplus \pi^2_7 (d \Om^1)$, 
we obtain \eqref{eq:decomp om27}. 
\end{proof}

Recall the map $\Ff_\Sp$ given in \eqref{eq:deform map Spin7}. 
Note the following. 

\begin{lemma} \label{onimage}
The image of $\Ff_\Sp$ is contained in $\i \Vv$. 
\end{lemma}

\begin{proof}
Given two connections $\nabla, \nabla' \in \Aa'_0$, 
we know that $F_{\nabla'}-F_\nabla \in \i d \Om^1$. 
Then, it follows that 
\[
\Ff_\Sp (\nabla') - \Ff_\Sp (\nabla) \in \i \pi^2_7 (d \Om^1 \oplus d^* \Om^3). \]
Since $\pi^2_7 (d \Om^1 \oplus d^* \Om^3)$ is orthogonal to 
$\Hh^2_7$ with respect to the $L^2$ inner product, 
it follows that 
$$
\i \pi^2_7 (d \Om^1 \oplus d^* \Om^3) \subset \i \Vv 
$$ 
by \eqref{eq:decomp om27}. Thus, if we take $\nabla \in \Ff_\Sp^{-1}(0)$, the statement follows. 
\end{proof}

We now describe the linearization of $\Ff_{\Sp}$. 
Use the notation of Appendix \ref{sec:new Spin7 str}. 
Fix $\nabla \in \Ff_{\Sp}^{-1}(0)$. 
Define a new $\Sp$-structure $\Phi_\n$ by 
\begin{align} \label{eq:newSpin7str}
\Phi_\n = 
(\id_{TX} + (- \i F_\n)^\sharp)^* \Phi. 
\end{align}
Denote by $g_{\nabla}$ the Riemannian metric on $X$ induced from $\Phi_\n$. 
Denote by $*_\n$ and $d^{*_\n}$ 
the Hodge star and the formal adjoint of the exterior derivative $d$ with respect to $g_{\n}$, respectively. 
The $\Sp$-structure $\Phi_\n$ induces the decomposition of 
$\Om^\bullet$ as in \eqref{decom1-L-W8}: 
$\Om^2 = \Om^2_{7, \n} \oplus \Om^2_{21, \n}$ etc. 
Let $\pi^k_{\l, \n}: \Om^k \rightarrow \Om^k_{\l, \n}$ be the projection. 
Then, by \eqref{eq:proj2} we have 
$$
\pi^k_{\l, \n} = (\id_{TX} + (- \i F_\n)^\sharp)^* \circ \pi^k_{\l} \circ ((\id_{TX} + (- \i F_\n)^\sharp)^{-1})^*. 
$$

\begin{remark}
In \cite{KY}, we show that 
the operator 
$(\id_{TX} + (- \i F_\n)^\sharp)^*$ plays an important role 
for deformations of dDT connections on a $G_2$-manifold. 
It also turns out to be successful for the deformation theory of $\Sp$-dDT connections. 
\end{remark}

Then, the linearization of $\Ff_{\Sp}$ at $\nabla$ is given as follows. 

\begin{proposition}
Fix $\nabla \in \Ff_{\Sp}^{-1}(0)$. 
Then, there exists a bundle isomorphism $P_\n:\Om^2_7 \rightarrow \Om^2_7$ 
depending on  $\n$ so that  
the linearization 
$\delta_\nabla \Ff_{\Sp}: \i \Om^1 \rightarrow \i \Vv$ 
of $\Ff_{\Sp}$ at $\nabla$ is given by  
\begin{equation}\label{eq:1stder Spin7}
\delta_\nabla \Ff_{\Sp}
= P_\n \circ  ((\id_{TX} + (- \i F_\n)^\sharp)^{-1})^* \circ \pi^2_{7, \n} \circ d. 
\end{equation}
\end{proposition}

\begin{proof}
By the definition of $\Ff_{\Sp}$ in \eqref{eq:deform map Spin7}, 
$\delta_\n \Ff_{\Sp}$ is given by 
\[
\begin{aligned}
\delta_\n \Ff_{\Sp} (\i b) 
= & 
\i \pi^2_7 \left(  d b + * \left(\frac{1}{2} F_\nabla^2 \wedge d b \right) \right). 
\end{aligned}
\]
for $b \in \Om^1$. 
Since $F_\n$ is $\i \R$-valued, 
$\nabla \in \Ff_{\Sp}^{-1}(0)$ implies that 
$- \i F_\n \in \Om^2$ satisfies 
$
\pi^2_7 \left( (- \i F_\n) - * (- \i F_\n)^3/6 \right) = 0. 
$
Then, Theorem \ref{thm:1+F Spin7 neq} (1) implies \eqref{eq:1stder Spin7}. 
\end{proof}

For any fixed $\nabla \in \Ff_{\Sp}^{-1}(0)$, 
consider the following complex 
\begin{equation}\label{cpxfordDT}
\begin{aligned}
0 \rightarrow \i \Om^0 
\stackrel{d} \rightarrow \i \Om^1 
\stackrel{\delta_\n \Ff_{\Sp}} \longrightarrow \i \Vv 
\rightarrow 0. 
\end{aligned}
\tag{$\#_\n$} 
\end{equation}
By \eqref{eq:Gu orbit Spin7}, the tangent space of $\Gg_U$-orbit through $\n$ is 
identified with $\i d \Om^1$. 
Hence, 
the first cohomology $H^1(\#_\n)$ of \eqref{cpxfordDT} is considered to be the tangent space of 
$\Mm'_{\Sp}$. 
We show that 
the second cohomology $H^2(\#_\n)$ is the obstruction space in Theorem \ref{thm:moduli MSpin7}. 
Denote by 
\begin{align} \label{cancpxfordDT2}
D \eqref{cpxfordDT} = (\delta_\n \Ff_{\Sp}, d^*): \i \Om^1 \rightarrow \i \Vv \oplus \i \Om^0
\end{align}
the associated two term elliptic complex. 
Note that if $F_\n=0$, $D \eqref{cpxfordDT}$ becomes 
\begin{align} \label{cancpxfordDT3}
D = (\pi^2_7 \circ d, d^*): \i \Om^1 \rightarrow \i \Vv \oplus \i \Om^0. 
\end{align}

\subsection{The expected dimension of $\Mm'_{\Sp}$} \label{sec:expdim}
In this subsection, we compute the expected dimension of $\Mm'_{\Sp}$. 
 
Via the inclusion $\Vv \hookrightarrow \Om^2_7$, 
we also regard $\delta_\n \Ff_{\Sp}$ as a map $\delta_\n \Ff_{\Sp}: \i \Om^1 \rightarrow \i \Om^2_7$ 
for $\n \in \Ff_\Sp^{-1}(0)$. 
Then, 
by \eqref{eq:1stder Spin7}, we have the following commutative diagram. 
\begin{align} \label{eq:commdiag}
\xymatrix{
0\ar[r] & \i \Om^0 \ar@{=}[d] \ar[r]^-{d} & \i \Om^1 \ar@{=}[d] \ar[r]^-{\delta_\n \Ff_{\Sp}}
& \i \Om^2_7 \ar[r] & 0\\
0\ar[r] & \i \Om^0 \ar[r]^-{d} & \i \Om^1 \ar[r]^-{\pi^2_{7, \n} \circ d} & \i \Om^2_{7, \n} 
\ar[u]_{P_\n \circ ((\id_{TX} + (- \i F_\n)^\sharp)^{-1})^*}^\cong 
\ar[r]
& 0 
}
\end{align}
The complex \eqref{cpxfordDT} is a subcomplex of the top row 
and the bottom row is the canonical complex $(\flat_{\Phi_\n})$ for the $\Sp$-structure $\Phi_\n$. 
In particular, \eqref{cpxfordDT} is considered to be a subcomplex of 
the canonical complex \eqref{cancpx} for $\n \in \Ff_\Sp^{-1} (0)$.

\begin{remark} \label{rem:prob}
In fact, we can obtain the commutative diagram similar to \eqref{eq:commdiag} 
if we do not assume $*F_\n^4/24 \neq 1$. 
Set 
$$
\Ff = \left( \Ff^{1}_{{\rm Spin}(7)}, \Ff^{2}_{{\rm Spin}(7)} \right): 
\Aa_0 \to \i \Om^2_7 \oplus \Om^4_7. 
$$
For $\n \in \Aa_0$, define 
$
(T_\n, S_\n): \i \Om^2 \to \i \Om^2_7 \oplus \Om^4_7
$
by 
$$
T_\n(\beta)= 
\pi^2_7 \left(  \beta + * \left(\frac{1}{2} F_\nabla^2 \wedge \beta \right) \right), 
\qquad 
S_\n(\beta)= 2 \pi^4_7 \left( F_\n \wedge \beta \right). 
$$
Then, we have 
$\delta_\n \Ff (\i b) = (T_\n, S_\n) (\i d b)$ for $b \in \Om^1$. 
By Theorem \ref{thm:1+F Spin7 neq} (2), 
for $\n \in \Aa_0$ satisfying $\Ff (\nabla)=0$, 
there exists an isomorphism 
$Q_\n: \i \Om^2_7 \to \Im (T_\n, S_\n)$ 
such that the following diagram commutes. 
\begin{align*}
\xymatrix{
0\ar[r] & \i \Om^0 \ar@{=}[d] \ar[r]^-{d} & \i \Om^1 \ar@{=}[d] 
\ar[r]^-{\delta_\n \Ff}
& \Im (T_\n, S_\n) \ar[r] & 0\\
0\ar[r] & \i \Om^0 \ar[r]^-{d} & \i \Om^1 \ar[r]^-{\pi^2_{7, \n} \circ d} & \i \Om^2_{7, \n} 
\ar[u]_{Q_\n \circ ((\id_{TX} + (- \i F_\n)^\sharp)^{-1})^*}^\cong 
\ar[r]
& 0 
}
\end{align*}
Thus, we can expect that 
the similar deformation theory would work if we do not assume $*F_\n^4/24 \neq 1$. 
The problem is that 
we do not know whether $\Ff$ is a section 
of an infinite dimensional vector bundle 
$$
\bigcup_{\n \in \Aa_0} \Im (T_\n, S_\n). 
$$
If this is the case, it is very likely that 
the similar deformation theory works without assuming $*F_\n^4/24 \neq 1$. 
\end{remark}

To compute the expected dimension of $\Mm'_{\Sp}$, 
we make use of the index of the canonical complex. 
First, note the following. 
\begin{lemma}\label{lem:2cohom}
For $\n \in \Ff_\Sp^{-1}(0)$, we have 
$$
H^1(\flat_{\Phi_\n}) \cong H^1 \eqref{cpxfordDT} \quad \mbox{and} \quad  
H^2(\flat_{\Phi_\n}) \cong \i \Hh^2_7 \oplus H^2 \eqref{cpxfordDT}.
$$ 
In other words, we have 
$$
\ker D (\flat_{\Phi_\n}) \cong \ker D \eqref{cpxfordDT} \quad \mbox{and} \quad 
{\rm Coker} D (\flat_{\Phi_\n}) \cong \i \Hh^2_7 \oplus {\rm Coker} D \eqref{cpxfordDT}.
$$
\end{lemma}

\begin{proof}
The first statement follows from \eqref{eq:commdiag}. 
By \eqref{eq:commdiag} and \eqref{eq:decomp om27}, we see that 
$$
H^2(\flat_{\Phi_\n}) 
=
\frac{\i \Om^2_{7, \n}}{\i \pi^2_{7, \n}(d \Om^1)}
\cong 
\frac{\i \Om^2_7}{\delta_\n \Ff_{\Sp} (\i \Om^1)}
= \frac{\i (\Hh^2_7 \oplus \Vv )}{\delta_\n \Ff_{\Sp} (\i \Om^1)}. 
$$
Since $\delta_\n \Ff_{\Sp} (\i \Om^1) \subset \i \Vv$ 
by Lemma \ref{onimage}, 
it follows that 
$$
H^2(\flat_{\Phi_\n}) 
\cong \i \Hh^2_7 \oplus \frac{\i \Vv}{\delta_\n \Ff_{\Sp} (\i \Om^1)}
= \i \Hh^2_7 \oplus H^2 \eqref{cpxfordDT}. 
$$
\end{proof}

\begin{lemma}\label{lem:ind Phin}
For $\n \in \Ff_\Sp^{-1}(0)$, we have 
${\rm ind} D (\flat_{\Phi_\n}) = {\rm ind} D (\flat_\Phi)$. 
\end{lemma}

\begin{proof}
For $s \in [0,1]$, define a $\Sp$-structure $\Phi (s)$ by 
$$
\Phi (s) = (\id_{TX} + s (- \i F_\n)^\sharp)^* \Phi. 
$$
Then, $\Phi (0)=\Phi$ and $\Phi (1) = \Phi_\n$. 
By \eqref{cancpx2}, we can define an elliptic operator 
$D(\flat_{\Phi (s)})$ for any $s \in [0,1]$ by 
$$
D(\flat_{\Phi (s)}) = (\pi^2_{7, \Phi (s)} \circ d, d^*): \i \Om^1 \rightarrow \i \Om^2_{7, \Phi (s)} 
\oplus \i \Om^0. 
$$ 
Since 
$((\id_{TX} + s (-\i F_\n)^\sharp)^{-1})^* :\Om^2_{7, \Phi (s)} \rightarrow \Om^2_7$
is an isomorphism by construction, 
the index of $D (\flat_{\Phi (s)})$ agrees with 
that of 
$$
((\id_{TX} + s (-\i F_\n)^\sharp)^{-1})^* \circ \pi^2_{7, \Phi (s)} \circ d, d^*): 
\i \Om^1 \rightarrow \i \Om^2_7 \oplus \i \Om^0. 
$$ 
Since $\{ ((\id_{TX} + s (-\i F_\n)^\sharp)^{-1})^* \circ \pi^2_{7, \Phi (s)} \circ d, d^*) \}_{s \in [0,1]}$ 
defines a family of elliptic operators between 
$\i \Om^1$ and $\i \Om^2_7 \oplus \i \Om^0$ and 
the index is invariant by perturbation, 
we obtain 
${\rm ind} D (\flat_{\Phi_\n}) = {\rm ind} D  (\flat_{\Phi (1)}) 
= {\rm ind} D (\flat_{\Phi (0)}) = {\rm ind} D(\flat_\Phi)$. 
\end{proof}

\begin{proposition} \label{prop:expected dim Spin7}
The expected dimension of $\Mm'_{\Sp}$ is given by 
\[
\dim H^1 \eqref{cpxfordDT} - \dim H^2 \eqref{cpxfordDT} 
=1 + \dim \Hh^2_7 - \hat A_2. 
\]
Here, 
$\hat A_2 = (7 p_1^2 -4 p_2)/5760$ 
is the second term of the $\hat A$-genus of $X$, 
where $p_i$ is the $i$-th Pontryagin class of $X$. 
\end{proposition}

\begin{proof}
By Lemma \ref{lem:2cohom}, we have 
\begin{align} \label{eq:expected dim Spin7}
\begin{split}
{\rm ind} D (\flat_{\Phi_\n}) 
=& -1 + \dim H^1 (\flat_{\Phi_\n}) - \dim H^2 (\flat_{\Phi_\n}) \\
=& -1 + \dim H^1 \eqref{cpxfordDT} - (\dim H^2 \eqref{cpxfordDT} + \dim \Hh^2_7). 
\end{split}
\end{align}
By \cite[Proposition 5.14]{MS2}, we have 
${\rm ind} D (\flat_\Phi) = - \hat A_2$. 
Then, by Lemma \ref{lem:ind Phin}, we obtain 
$\dim H^1 \eqref{cpxfordDT} - \dim H^2 \eqref{cpxfordDT} = 1 + \dim \Hh^2_7 - \hat A_2$. 
\end{proof}

If the $\Sp$-structure $\Phi$ is torsion-free, we know that 
\begin{align} \label{eq:Ahat}
\hat A_2 = 1 - b^1 + \dim \Hh^2_7
\end{align}
by \cite[p.43]{MS2} or \cite[(10.25)]{Joyce0}, 
where $b^1$ is the first Betti number. 
Then, the expected dimension of $\Mm'_{\Sp}$ is given more simply as follows. 

\begin{lemma}\label{lem:ind Phi}
If the $\Sp$-structure $\Phi$ is torsion-free, 
we have  
\[
\dim H^1 \eqref{cpxfordDT} - \dim H^2 \eqref{cpxfordDT} =b^1. 
\]
\end{lemma}
We also obtain the following from \eqref{eq:Ahat}. 
\begin{lemma} \label{lem:vanish 27}
Suppose that the $\Sp$-structure $\Phi$ is torsion-free. 
If $\beta \in \Om^2_7$ satisfies $d^* \beta =0$, then $\beta \in \Hh^2_7$. 
\end{lemma}

\begin{proof}
By \eqref{eq:Ahat}, 
${\rm ind} D (\flat_\Phi) = - \hat A_2$ is equivalent to 
$$
\dim H^1(\flat_\Phi) - \dim H^2(\flat_\Phi) = b^1-\dim \Hh^2_7. 
$$
We first show that $H^1(\flat_\Phi)= H^1_{dR}$, 
which implies $\dim H^1(\flat_\Phi)= b^1$. 
Suppose that $\alpha \in \ker (\pi^2_7 \circ d)$. 
Then, $d \alpha \in \Om^2_{21}$, and hence, 
$d(\Phi \wedge \alpha) = \Phi \wedge d \alpha = - * d \alpha$. 
The Hodge decomposition implies that $d \alpha =0$, 
and hence, $H^1 (\flat_\Phi) = H^1_{dR}$.

Then, it follows that $\dim H^2(\flat_\Phi)= \dim \Hh^2_7$. 
By \eqref{eq:decomp om27 1}, we have 
$H^2 (\flat_\Phi) \cong \ker (\pi^2_7 \circ d)^*$. 
Since $\ker (\pi^2_7 \circ d)^* \supset \Hh^2_7$ 
by the proof of Lemma \ref{lem:decomp om27}, we obtain 
$$
\ker (\pi^2_7 \circ d)^* = \Hh^2_7. 
$$
Since $\beta \in \ker (\pi^2_7 \circ d)^*$ is equivalent to 
$0 
= \la \beta,  \pi^2_7 (d \alpha) \ra_{L^2}
= \la \beta,  d \alpha \ra_{L^2}
= \la d^* \beta,  \alpha \ra_{L^2} 
$
for any $\alpha \in \Om^1$, 
$\ker (\pi^2_7 \circ d)^*$ agrees with the space of 
coclosed forms in $\Om^2_7$. 
Hence, the proof is completed. 
\end{proof}

%%%%%%%%%%%%%%%%%%%%%%%%%%%%%%%%%%%%%%%%%%%%%%%%%%%%%%%%%
\subsection{The local structure of $\Mm'_{\Sp}$}\label{sec:local str MSpin7}
In this subsection, we consider the smoothness of $\Mm'_{\Sp}$. 
Fix any $\nabla \in \Ff_{\Sp}^{-1}(0)$. 
Set 
\[
\Om^1_{d^{*_\n}} 
= \{\, d^{*_\n} \mbox{-closed 1-forms on } X \,\}.
\]
By \eqref{eq:Gu orbit Spin7}, the tangent space of $\Gg_U$-orbit through $\n$ is 
identified with $\i d \Om^1$. 
Then, by the Hodge decomposition with respect to $g_\n$ and the same argument 
as in \cite{KY}, 
we can show that 
\[
\i \Om^1_{d^{*_\n}} \rightarrow \Aa'_0/\Gg_U, \quad
A \mapsto [\nabla+A \cdot \id_L]
\]
gives a homeomorphism from a neighborhood of $0 \in \i \Om^1_{d^{*_\n}}$ 
to that of $[\nabla] \in \Aa'_0/\Gg_U$. 
Hence, 
a neighborhood of $[\nabla]$ in $\Mm'_{\Sp}$ is homeomorphic to 
that of $0$ in 
\[
\Ss_\n = \{\, a \in \Om^1_{d^{*_\n}} \mid \Ff_{\Sp} (\nabla + \i a \cdot \id_L) = 0 \,\}. 
\]

\begin{theorem} \label{thm:moduli MSpin7}
If $H^2 \eqref{cpxfordDT} = \{\, 0 \,\}$ for $\nabla \in \Ff^{-1}_{\Sp}(0)$, 
the moduli space $\Mm'_{\Sp}$ is a finite dimensional smooth manifold near $[\nabla]$. 
If $\Phi$ is torsion-free, its dimension is $b^1$, 
where $b^1$ is the first Betti number.
\end{theorem}

\begin{proof}
We only have to show that $\Ss_\n$ is a smooth manifold near $0$. 
If  $H^2 \eqref{cpxfordDT} = \{\, 0 \,\}$, we have 
$\mathop{\mathrm{Im}} (\delta_\n \Ff_{\Sp}) = \i \Vv$. 
Then, we can apply the implicit function theorem 
(after the Banach completion) to 
\[
\Om^1_{d^{*_\n}} \rightarrow \i \Vv, \quad 
a \mapsto \Ff_{\Sp} (\nabla + \i a \cdot \id_L). 
\]
Then, we see that 
$\Ss_\n$ is a smooth manifold near $0$. 
Its dimension is given by 
$\dim H^1 \eqref{cpxfordDT} = 1 + \dim \Hh^2_7 - \hat A_2$ 
by Proposition \ref{prop:expected dim Spin7}.
If $\Phi$ is torsion-free, it is $b^1$ by Lemma \ref{lem:ind Phi}.

As for the regularity of elements in $\Ff_{\Sp}^{-1}(0)$ in the Banach completion, 
note that 
\[
\Ss_\n = 
\{\, a \in \Om^1 \mid \Ff_{\Sp} (\nabla + \i a \cdot \id_L) = 0, \ d^{*_\n} a =0 \,\}. 
\]
Since 
\[
(\delta_\nabla \Ff_{\Sp}, d^{*_\n}): \i \Om^1
\rightarrow \i \Om^2_7 \oplus \Om^0  
\]
is overdetermined elliptic by 
the ellipticity of \eqref{cpxfordDT} at $\i \Om^1$, 
we see that $\Ss_\n$ is the solution space of 
an overdetermined elliptic equation around $0$. 
Thus, all solutions around $0$ are smooth. 
\end{proof}

\begin{remark}

If $\Phi$ and $\Phi_\n$ are torsion-free, 
$\Mm'_\Sp$ is smooth near $[\n]$ of dimension $b^1$. 
Indeed, if $\Phi_\n$ is torsion-free, 
by the proof of Lemma \ref{lem:vanish 27}, 
we see that 
$\dim H^2(\flat_{\Phi_\n}) = \dim \Hh^2_{7, \n}$, 
where $\Hh^2_{7, \n}$ is the space of harmonic forms with respect to $g_\n$ in $\Om^2_{7,\n}$. 
On a $\Sp$-manifold, $\dim \Hh^2_{7, \n}$  
is a topological invariant by \eqref{eq:Ahat}. 
Hence, we have $\dim \Hh^2_{7, \n} =\dim \Hh^2_7$. 
Then, by the second equation of Lemma \ref{lem:2cohom}, we see that 
$H^2 \eqref{cpxfordDT} = \{ 0 \}$.

In particular, if $\Phi$ is torsion-free and 
a connection $\n$ is flat, which is obviously a $\Sp$-dDT connection 
satisfying $* F_\n^4/24 \neq 1$, 
we have $\Phi_\n = \Phi$, and hence, $\Mm'_\Sp$ is a smooth manifold near $[\n]$ of dimension $b^1$. 
On the other hand, 
the space of flat Hermitian connections is $\n + \i Z^1$, 
where $Z^1$ is the space of closed 1-forms,  
and 
the tangent space of $\Gg_U$-orbit through $\n$ is 
identified with $d \Om^1$ by \eqref{eq:Gu orbit Spin7}. 
Thus, we have a $b^1$-dimensional family of flat Hermitian connections in $\Mm'_\Sp$. 
In particular, a neighborhood of a flat connection $\n$ in $\Mm'_\Sp$ consists of flat connections. 
\end{remark}

%%%%%%%%%%%%%%%%%%%%%%%%%%%%%%%%%%%%%%%%%%%%%%%%%%%%%%%%%
%%%%%%%%%%%%%%%%%%%%%%%%%%%%%%%%%%%%%%%%%%%%%%%%%%%%%%%%%
%%%%%%%%%%%%%%%%%%%%%%%%%%%%%%%%%%%%%%%%%%%%%%%%%%%%%%%%%
%%%%%%%%%%%%%%%%%%%%%%%%%%%%%%%%%%%%%%%%%%%%%%%%%%%%%%%%%
%%%%%%%%%%%%%%%%%%%%%%%%%%%%%%%%%%%%%%%%%%%%%%%%%%%%%%%%%
%%%%%%%%%%%%%%%%%%%%%%%%%%%%%%%%%%%%%%%%%%%%%%%%%%%%%%%%%
\subsection{Varying the $\Sp$-structure}\label{sec:moduli generic}
Let $X^8$ be a compact connected 8-manifold with a $\Sp$-structure $\Phi$ 
and $L \to X$ be a smooth complex line bundle with a Hermitian metric $h$. 
In Theorem \ref{thm:moduli MSpin7}, 
we gave the condition for the moduli space of $\Sp$-dDT connections $\n$ 
with $*F_\n^4/24 \neq 1$ to be smooth if it is not empty. 
In this subsection, 
we show that it is a smooth manifold (or empty) if we perturb $\Phi$ generically 
in some cases. 

Let $\Uu \subset \Om^4$ be an open set of the space of $\Sp$-structures 
containing the initial $\Sp$-structure $\Phi$. 
Define infinite dimensional vector bundles $\Ee \to \Uu$ 
and $\varpi: \Aa'_{0, \Uu} \to \Uu$ by 
$$
\Ee = \bigcup_{\Psi \in \Uu} \{\Psi\} \times \i \Vv_\Psi,  \qquad
\Aa'_{0, \Uu} = \bigcup_{\Psi \in \Uu} \{\Psi\} \times \{\,\nabla \in \Aa_0 \mid *_\Psi F_\n^4/24 \neq 1 \}, 
$$
where $\Vv_\Psi \subset \Om^2_{7, \Psi}$ is the space constructed in Lemma \ref{lem:decomp om27} 
for $\Psi \in \Uu$, 
$\Aa_0$ is the space of Hermitian connections of $(L,h)$ 
and $*_\Psi$ is the Hodge star induced from $\Psi \in \Uu$. 
Note that $\Aa'_{0, \Uu}$ is an open subset of $\Uu \times \Aa_0$ 
and $\Gg_U$ acts on $\Aa'_{0, \Uu}$ fiberwise.

For the rest of this subsection, we suppose that 
there exist $\Sp$-dDT connections $\n_\Psi$ with $*_\Psi F_{\n_\Psi}^4/24 \neq 1$ 
for any $\Psi \in \Uu$.  
By Lemma \ref{onimage}, 
we can define a section  
$\wFf : \Aa'_{0, \Uu} \to \varpi^* \Ee$, where $\varpi^* \Ee \to \Aa'_{0, \Uu}$ is the pullback of $\Ee \to \Uu$ by $\varpi$, 
by 
\begin{align*}
\wFf (\Psi, \n') 
&= 
\pi^2_{7, \Psi} \left( F_{\n'} + \frac{1}{6} *_\Psi F_{\n'}^3 \right) \\
&= 
\frac{1}{4}
\left(
F_{\n'} + \frac{1}{6} *_\Psi F_{\n'}^3 + 
*_\Psi \left( \left( F_{\n'} + \frac{1}{6} *_\Psi F_{\n'}^3  \right) \wedge \Psi \right) 
\right)
\end{align*}
and set 
$$
\widehat \Mm'_\Sp = \left \{\,(\Psi, [\n']) \in  \Aa'_{0, \Uu}/\Gg_U 
\mid \wFf (\Psi, \n') = 0\,\right\}. 
$$
For each $\Psi \in \Uu$, set 
$\mathcal{M}'_{\Sp, \Psi} = \left( \wFf^{-1}(0) \cap \varpi^{-1}(\Psi) \right)/\Gg_U$, 
which is the moduli space of $\Sp$-dDT connections 
with $*_\Psi F^4_\n/24 \neq 1$ for the $\Sp$-structure $\Psi$.

We first compute the linearization of $\wFf$. 
Recall the notation in Appendix \ref{app:notation}. 
Define $\ja:\Om^1 \otimes \Om^1 \to \Om^4, 
\jb (h):\Om^p \to \Om^p$ for $h \in \Ss^2$, 
$\jc: \Om^p \times \Om^p \to \Ss^2$ 
by \eqref{eq:j1}, \eqref{eq:j2}, \eqref{eq:j3}
for the initial $\Sp$-structure $\Phi$, respectively. 
For simplicity, set 
$g=g_\Phi, *=*_\Phi, \Vv=\Vv_\Phi, \Om^k_\l=\Om^k_{\l, \Phi}$ 
and $\pi^k_\l = \pi^k_{\l, \Phi}: \Om^k \to \Om^k_\l$ 
for the initial $\Sp$-structure $\Phi$. 

\begin{lemma} \label{lem:lin wFf}
Set $A =h + \beta  \in \Ss^2 \oplus \Om^2_7$. 
Then, the linearization of $\wFf$ at $(\Phi, \n) \in \wFf^{-1}(0)$ 
in the direction of $\ja (A)$ is given by 
\begin{align*}
&4 \delta_{(\Phi, \n)} \wFf (\ja (A), 0) \\
=&
*\left( \left( F_\n + \frac{1}{6} * F_\n^3 \right) \wedge \ja (A) \right) 
+ {\rm tr} (h^\sharp)  \left(  -F_\nabla + \frac{1}{6} * \left( (* F_\nabla^3) \wedge \Phi \right) \right) \\
&+2 * \left( \jb (h) (* F_\n) - \frac{1}{6} * \left( \jb (h) (F^3_\n) \right) \wedge \Phi \right).  
\end{align*}
\end{lemma}

\begin{proof}
Let $\{ \Phi_t \}_{t \in (- \epsilon, \epsilon)}$ be 
a path in $\Uu$ such that 
$
\Phi_0 = \Phi
$ and
$d \Phi_t/dt |_{t=0} = \ja (A)$. 
Recall that 
$$
\left. \frac{d}{dt} *_{\Phi_t} \beta \right|_{t=0} = * \left( -2 \jb (h) (\beta) + {\rm tr}(h^\sharp) \beta \right)  
$$
for $\beta \in \Om^p$ by Lemma \ref{lem:diff Hodge}. 
Thus, we have 
\begin{align*}
&4 \delta_{(\Phi, \n)} \wFf (\ja (A), 0) \\
=&
4 \left. \frac{d}{dt} \wFf (\Phi_t, \n) \right|_{t=0} \\
=&\frac{1}{6} * \left( -2 \jb (h) (F^3_\n) + {\rm tr}(h^\sharp) F^3_\n \right) \\
&+
* \left( -2 \jb (h) \left( \left( F_\nabla + \frac{1}{6} * F_\nabla^3  \right) \wedge \Phi \right) 
+ {\rm tr}(h^\sharp) \left( F_\nabla + \frac{1}{6} * F_\nabla^3  \right) \wedge \Phi \right)\\
&+
*\left( \left( F_\n + \frac{1}{6} * F_\n^3 \right) \wedge \ja (A) \right)
+
\frac{1}{6} * \left( *\left( -2 \jb (h) (F^3_\n) + {\rm tr}(h^\sharp) F^3_\n \right) \wedge \Phi \right)\\
=&*\left( \left( F_\n + \frac{1}{6} * F_\n^3 \right) \wedge \ja (A) \right) \\
&+ 
{\rm tr}(h^\sharp) 
\left( \frac{1}{6} * F^3_\n 
+ * \left( \left( F_\nabla + \frac{1}{6} * F_\nabla^3  \right) \wedge \Phi \right) 
+ \frac{1}{6} * \left( (* F_\nabla^3) \wedge \Phi \right) \right) \\
&-
2 * \left( \jb (h) \left( \frac{1}{6} F^3_\n 
+ \left( F_\nabla + \frac{1}{6} * F_\nabla^3  \right) \wedge \Phi \right) 
+ \frac{1}{6} * \left( \jb (h) (F^3_\n) \right) \wedge \Phi
\right). 
\end{align*}
Since $\n$ is a $\Sp$-dDT connection for $\Phi$, we have 
$$
\frac{1}{6} * F^3_\n 
+ * \left( \left( F_\nabla + \frac{1}{6} * F_\nabla^3  \right) \wedge \Phi \right) 
= -F_\n. 
$$
Then, the proof is completed. 
\end{proof}

\begin{proposition} \label{prop:L2 oncomp}
Let $\n$ be a $\Sp$-dDT connection with respect to the initial $\Sp$-structure $\Phi$. 
Let 
$$
D_{(\Phi, \n)} \wFf: 
T_{(\Phi, \n)} \Aa'_{0, \Uu} = 
\Om^4_1 \oplus \Om^4_7 \oplus \Om^4_{35} \oplus \i \Om^1 
\to \i \Vv
$$ 
be the composition of the linearization of 
$\delta_{(\Phi, \n)} \wFf: T_{(\Phi, \n)} \Aa'_{0, \Uu} \to T_{(\Phi,0)} \Ee$ 
with the projection
$$
T_{(\Phi,0)} \Ee = T_\Phi \Uu \oplus \Ee_\Phi \to \Ee_\Phi = \i \Vv_\Phi = \i \Vv. 
$$ 
Denote by  
$\left( \Im D_{(\Phi, \n)} \wFf  \right)^\perp$ 
the space of elements in $\i \Vv$ orthogonal to 
the image of $D_{(\Phi, \n)} \wFf$  
with respect to the $L^2$ inner product induced from $g=g_\Phi$.  
\begin{enumerate}
\item 
The map $D_{(\Phi, \n)} \wFf$ is surjective 
if and only if 
$$\left( \Im D_{(\Phi, \n)} \wFf  \right)^\perp = \{\, 0 \,\}.$$ 

\item
For $\gamma \in \Vv$, 
$\i \gamma \in \left( \Im D_{(\Phi, \n)} \wFf  \right)^\perp$ 
if and only if 
\begin{align}
d^* \left( \gamma + * \left( \frac{1}{2} F^2_\n \wedge \gamma \right) \right) &=0, 
\label{eq:L2 on1} \\
\pi^4_1 (F_\n \wedge \gamma) &=0, \label{eq:L2 on2}\\
\pi^4_7 \left( \left( 6F_\n +*F_\n^3 \right) \wedge \gamma \right) &=0, \label{eq:L2 on3}\\
\pi^4_{35} \left( \left(6 F_\n -* F_\n^3 \right) \wedge \gamma \right) &=0. \label{eq:L2 on4}
\end{align}
\end{enumerate}
\end{proposition}

\begin{proof}
We first prove (1). The proof is an analogue of \cite{KY}. 
Define 
$D_1: \i \Om^1 \to \i \Vv$ and 
$D_2: \Om^4_1 \oplus \Om^4_7 \oplus \Om^4_{35}  
\to \i \Vv$ by 
\begin{align} \label{eq:D1D2}
D_1 = D_{(\Phi, \n)} \wFf|_{\{ 0 \} \times \i \Om^1}, \quad 
D_2 = D_{(\Phi, \n)} \wFf|_{\Om^4_1 \oplus \Om^4_7 \oplus \Om^4_{35} \times \{ 0 \}}. 
\end{align}
Then, $\Im D_{(\Phi, \n)} \wFf = \Im D_1 + \Im D_2$. 
By \eqref{eq:commdiag} and 
the ellipticity of the canonical complex,  
there exists a finite dimensional subspace $U_1 \subset \i \Om^2_7$ 
such that 
\[
 \i \Om^2_7 = U_1 \oplus \Im D_1, 
\]
which is an orthogonal decomposition 
with respect to the $L^2$ inner product induced from $g$. 
By Lemma \ref{onimage}, we have $\Im D_1 \subset \i \Vv$. 
Set $U_2=U_1 \cap \i \Vv$.  
Then, 
\begin{align} \label{eq:L2 oncomp 1}
\i \Vv = U_2 \oplus \Im D_1. 
\end{align}
Since $U_2$ is finite dimensional, there is 
the orthogonal complement of $U_2 \cap (\Im D_1 + \Im D_2)$ in $U_2$ 
with respect to the $L^2$ inner product induced from $g$. 
Denote it by $U_3$. 
Then, 
since $\Im D_1 + \Im D_2 = (U_2 \cap (\Im D_1 + \Im D_2)) \oplus \Im D_1$, 
we obtain the following orthogonal decomposition
\[
\i \Vv = U_3 \oplus (\Im D_1 + \Im D_2) 
\]
with respect to the $L^2$ inner product induced from $g$. 
By construction, 
$U_3 = \left( \Im D_{(\Phi, \n)} \wFf  \right)^\perp$, 
and hence, we obtain (1).

Next, we prove (2). 
By the decomposition $S^2 T^*X = \R g \oplus S^2_0 T^*X$, 
$\i \gamma \in \left( \Im D_{(\Phi, \n)} \wFf  \right)^\perp$ 
if and only if 
\begin{align}
\la \delta_{(\Phi, \n)} \wFf (0, \i b), \i \gamma \ra_{L^2}&=0, 
\label{eq:L20 on1} \\
\la \delta_{(\Phi, \n)} \wFf (\ja (f g), 0), \i \gamma \ra_{L^2}&=0, \label{eq:L20 on2}\\
\la \delta_{(\Phi, \n)} \wFf (\ja (\beta), 0), \i \gamma \ra_{L^2}&=0, \label{eq:L20 on3}\\
\la \delta_{(\Phi, \n)} \wFf (\ja (h), 0), \i \gamma \ra_{L^2}&=0 \label{eq:L20 on4}
\end{align}
for any $b \in \Om^1, f \in \Om^0, \beta \in \Om^2_7$ and $h \in \Ss^2_0$, 
where $\la \,\cdot\,, \,\cdot\, \ra_{L^2}$ is the $L^2$ inner product induced from $g$.  
We rewrite these equations. 

First, we rewrite \eqref{eq:L20 on1}. 
Since 
$$
\delta_{(\Phi, \n)} \wFf (0, \i b) = \i \pi^2_7 \left( db + * \left( \frac{1}{2} F^2_\n \wedge db \right) \right), 
$$
we have 
\begin{align*}
\la \delta_{(\Phi, \n)} \wFf (0, \i b), \i \gamma \ra_{L^2}
=&
\la db, \gamma \ra_{L^2} 
+ \left \la db, * \left( \frac{1}{2} F^2_\n \wedge \gamma \right) \right \ra_{L^2} \\
=&
\left \la b, d^* \left( \gamma + * \left( \frac{1}{2} F^2_\n \wedge \gamma \right) \right) \right \ra_{L^2}, 
\end{align*}
which implies that \eqref{eq:L20 on1} is equivalent to \eqref{eq:L2 on1}. 

Next, we rewrite \eqref{eq:L20 on2}. 
Since 
$$
{\rm tr}(f g^\sharp) = 8f, \qquad \jb(fg) |_{\Om^p} = pf \id_{\Om^p}, \qquad \ja (f g) = 4 f \Phi, 
$$
Lemma \ref{lem:lin wFf} implies that 
\begin{align*}
&4 \delta_{(\Phi, \n)} \wFf (\ja (fg), 0) \\
=&
4f *\left( \left( F_\n + \frac{1}{6} * F_\n^3 \right) \wedge \Phi \right) 
+ 8f \left(  -F_\nabla + \frac{1}{6} * \left( (* F_\nabla^3) \wedge \Phi \right) \right) \\
&+2f * \left( 6 (* F_\n) - (* F^3_\n) \wedge \Phi \right) \\
=&
4f *(F_\n \wedge \Phi) +4f F_\n \\
=&
16 f \pi^2_7 (F_\n). 
\end{align*}
Thus, since $\gamma \in \Vv \subset \Om^2_7$, we have 
\begin{align*}
&4 \la \delta_{(\Phi, \n)} \wFf (\ja (f g), 0), \i \gamma \ra_{L^2} \\
=&
16 \la f (-\i F_\n), \gamma \ra_{L^2} \\
=&
\frac{16}{3} \int_X f (-\i F_\n) \wedge \gamma \wedge \Phi 
=
\frac{16}{3} \la -\i F_\n \wedge \gamma, f \Phi \ra_{L^2}, 
\end{align*}
which implies that \eqref{eq:L20 on2} is equivalent to \eqref{eq:L2 on2}.

Next, we rewrite \eqref{eq:L20 on3}. 
By Lemma \ref{lem:lin wFf}, we have 
\begin{align*}
\la 4 \delta_{(\Phi, \n)} \wFf (\ja (\beta), 0), \i \gamma \ra_{L^2} 
=&
\left \la * \left( \left( F_\nabla + \frac{1}{6} * F_\nabla^3  \right) \wedge \ja (\beta) \right), \i \gamma \right \ra_{L^2} \\
=&
\left \la -\i  \left( F_\nabla + \frac{1}{6} * F_\nabla^3  \right) \wedge \gamma, \ja (\beta) \right \ra_{L^2}, 
\end{align*}
which implies that \eqref{eq:L20 on3} is equivalent to \eqref{eq:L2 on3}.

Finally, we rewrite \eqref{eq:L20 on4}. 
By Lemma \ref{lem:lin wFf}, we have 
\begin{align*}
&\la 4 \delta_{(\Phi, \n)} \wFf (\ja (h), 0), \i \gamma \ra_{L^2} \\
=&
\left \la 
* \left( \left( F_\nabla + \frac{1}{6} * F_\nabla^3  \right) \wedge \ja (h) \right) \right. \\
&\left. +2 * \left( \jb (h) (* F_\n) - \frac{1}{6} * \left( \jb (h) (F^3_\n) \right) \wedge \Phi \right), 
\i \gamma \right \ra_{L^2}. 
\end{align*}
We compute 
\begin{align*}
&\left \la * \left( \left( F_\nabla + \frac{1}{6} * F_\nabla^3  \right) \wedge \ja (h) \right), \i \gamma \right \ra_{L^2} \\
=&
\left \la -\i * \left( \left( F_\nabla + \frac{1}{6} * F_\nabla^3  \right) \wedge \ja (h) \right), \gamma \right \ra_{L^2} \\
=&
\left \la \ja (h), \i \left( F_\nabla + \frac{1}{6} * F_\nabla^3  \right) \wedge \gamma \right \ra_{L^2}
\end{align*}
since $\ja (h) \in \Om^4_{35}$ is anti self dual. 
By Lemmas \ref{lem:j2ad} and \ref{lem:pi35} and \eqref{eq:F2decomp}, we have 
\begin{align*}
\left \la 2 * \left( \jb (h) (* F_\n) \right), \i \gamma \right \ra_{L^2} 
=&
\left \la -2 \i \jb (h) (* F_\n), * \gamma \right \ra_{L^2} \\
=&
\frac{1}{4} \left \la \ja (h) , (\ja \circ \jc) ( -\i *F_\n \otimes * \gamma) \right \ra_{L^2} \\
=&
\frac{1}{4} \left \la \ja (h) , (\ja \circ \jc) ( -\i * \pi^2_{21}(F_\n) \otimes * \gamma) \right \ra_{L^2} \\
=&
\left \la \ja (h) , \pi^4_{35} ( \pi^2_{21}(\i F_\n) \wedge \gamma) \right \ra_{L^2} \\
=&
\left \la \ja (h) , \pi^4_{35} ( \i F_\n \wedge \gamma) \right \ra_{L^2}. 
\end{align*}

Similarly, we compute 
\begin{align*}
\left \la * \left( - \frac{1}{3} * \left( \jb (h) (F^3_\n) \right) \wedge \Phi \right), 
\i \gamma \right \ra_{L^2} 
=&
\left \la \frac{\i}{3} * \left( \jb (h) (F^3_\n) \right) \wedge \Phi, 
* \gamma \right \ra_{L^2}\\
=&
\int_X 
\frac{\i}{3} * \left( \jb (h) (F^3_\n) \right) \wedge \Phi \wedge \gamma. 
\end{align*}
Since $\Phi \wedge \gamma = 3 * \gamma$, it follows that 
\begin{align*}
\int_X 
\i * \left( \jb (h) (F^3_\n) \right) \wedge * \gamma 
=&
\la \i \jb (h) (F^3_\n), * \gamma \ra_{L^2} \\
=&
\frac{1}{8} \left \la \ja (h), (\ja \circ \jc) (\i F^3_\n \otimes * \gamma) \right \ra_{L^2} \\
=&
\frac{1}{8} \left \la \ja (h), (\ja \circ \jc) (\i \pi^6_{21}(F^3_\n) \otimes * \gamma) \right \ra_{L^2} \\
=&
- \frac{1}{2} \left \la \ja (h), 
\pi^4_{35} \left( \i \pi^2_{21}(* F^3_\n) \wedge \gamma \right) \right \ra_{L^2} \\
=&
- \frac{1}{2} \left \la \ja (h), 
\pi^4_{35} \left( \i * F^3_\n \wedge \gamma \right) \right \ra_{L^2}. 
\end{align*}
Hence, we obtain 
\begin{align*}
&\la 4 \delta_{(\Phi, \n)} \wFf (\ja (h), 0), \i \gamma \ra_{L^2} \\
=&
\left \la 
\ja (h),  
\pi^4_{35} \left( \left( 2 \i F_\n -\frac{\i}{3} * F^3_\n \right) \wedge \gamma \right) 
\right \ra_{L^2},  
\end{align*}
which implies that \eqref{eq:L20 on4} is equivalent to \eqref{eq:L2 on4}.
\end{proof}

By Proposition \ref{prop:L2 oncomp}, we obtain the following.

\begin{theorem}\label{thm:moduli generic}
Let $X^8$ be a compact connected 8-manifold with a $\Sp$-structure $\Phi$ 
and $L \to X$ be a smooth complex line bundle with a Hermitian metric $h$. 
Let $\n$ be a $\Sp$-dDT connection for $\Phi$ with $* F_\n^4/24 \neq 1$. 
Suppose that there exist $\Sp$-dDT connections $\n_\Psi$
with $*_\Psi F_{\n_\Psi}^4/24 \neq 1$
for 
every $\Sp$-structure $\Psi$ sufficiently close to $\Phi$. 
If 
\begin{enumerate}
\item
$F_\n \neq 0$ on a dense set of $X$, or 
\item
$\n$ is flat and the $\Sp$-structure $\Phi$ is torsion-free, 
\end{enumerate}
then for every generic $\Sp$-structure $\Psi$ close to $\Phi$, 
the subset of elements of the moduli space $\mathcal{M}'_{\Sp, \Psi}$ 
of $\Sp$-dDT connections $\n'$ for $\Psi$ with $*_\Psi F^4_{\n'}/24 \neq 1$ 
close to $[\n]$ is a finite dimensional smooth manifold (or empty).  
\end{theorem}

\begin{proof}
We first show that 
$D_{(\Phi, \n)} \wFf: \Om^4_1 \oplus \Om^4_7 \oplus \Om^4_{35} \oplus \i \Om^1 \to \i \Vv$ is surjective if we assume (1) or (2). 
By Proposition \ref{prop:L2 oncomp}, 
we only have to show that 
$\gamma \in \Vv$ satisfying \eqref{eq:L2 on1}--\eqref{eq:L2 on4} is zero. 

Suppose that 
$\gamma \in \Vv$ satisfies \eqref{eq:L2 on1}--\eqref{eq:L2 on4}. 
If (1) holds, we see that $\gamma$ vanishes on a dense set by Lemma \ref{lem:generic pt}. 
Since $\gamma$ is continuous, we see that $\gamma=0$. 
If (2) holds, \eqref{eq:L2 on1} is equivalent to $d^{*} \gamma =0$. 
By Lemma \ref{lem:vanish 27}, we see that $\gamma \in \Hh^2_7$. 
Since $\Hh^2_7 \cap \Vv = \{ 0 \}$, we obtain $\gamma=0$.

Next, we show that $\widehat \Mm'_\Sp$ is smooth near $(\Phi, [\n])$ by the implicit function theorem. 
As in Subsection \ref{sec:local str MSpin7}, the map 
\[
p: \Uu \times \Om^1_{d^{*_\n}} \rightarrow \Aa'_{0, \Uu}/\Gg_U, \quad
(\Psi, a) \mapsto (\Psi, [\nabla+ \i a \cdot \id_L])
\]
gives a homeomorphism from a neighborhood of $(\Phi, 0) \in \Uu \times \Om^1_{d^{*_{\n}}}$ 
to that of $(\Phi, [\nabla]) \in \Aa'_{0, \Uu}/\Gg_U$. 
Hence, a neighborhood of $(\Phi, [\nabla])$ in $\widehat \Mm'_\Sp$ is homeomorphic to 
that of $(\Phi, 0)$ in 
\[
\widehat \Ss_{(\Phi, \n)} = \left\{\,(\Psi,a) \in \Uu \times \Om^1_{d^{*_{\n}}}
\mid 
\wFf (\Psi, \nabla + \i a \cdot \id_L) = 0\,
\right \}. 
\]
Thus, we only have to show that $\widehat \Ss_{(\Phi, \n)}$ is smooth near $(\Phi, 0)$. 
Since 
$\{ 0 \} \times \i Z^1 \subset \ker \delta_{(\Phi, \n)} \wFf$, 
where $Z^1$ is the space of closed 1-forms, 
\begin{align*}
D':=D_{(\Phi, \n)} \wFf|_{T_\Phi \Uu \oplus \i \Om^1_{d^{*_\n}}}: 
T_\Phi \Uu \oplus \i \Om^1_{d^{*_\n}} \to \i \Vv
\end{align*}
is surjective if we assume (1) or (2).  

To apply the implicit function theorem, 
we show that there exist vector spaces 
$\Ww, \Ww' \subset T_\Phi \Uu \oplus \i \Om^1_{d^{*_\n}}$ 
such that 
$T_\Phi \Uu \oplus \i \Om^1_{d^{*_\n}} = \Ww \oplus \Ww'$ 
and 
$D'|_\Ww: \Ww \to \i \Vv$ is an isomorphism. 
Set 
$$
D'_1=D'|_{\i \Om^1_{d^{*_\n}}}, \qquad D'_2=D'|_{T_\Phi \Uu}.  
$$
Then, we have 
${\rm Im}D_1'={\rm Im}D_1, D_2'=D_2$ and 
${\rm Im}D'={\rm Im}D_1'+{\rm Im}D_2'$, 
where $D_1$ and $D_2$ are defined by \eqref{eq:D1D2}. 
By \eqref{eq:L2 oncomp 1}, we have 
$\i \Vv = U_2 \oplus \Im D_1'$. 
Since $D'$ is surjective, there exists a finite dimensional subspace 
$U_2' \subset T_\Phi \Uu$ such that 
$D'_2|_{U_2'}: U_2' \to U_2$ is an isomorphism. 
Since 
$U_2'$ is finite dimensional, there exists a subspace 
$(U_2')^\perp \subset T_\Phi \Uu$ such that 
$$
T_\Phi \Uu = U_2' \oplus (U_2')^\perp. 
$$ 
By \eqref{eq:commdiag} and the ellipticity of the canonical complex, 
$\ker D_1'$ is finite dimensional. 
Then, there exists a subspace 
$(\ker D_1')^\perp \subset T_\Phi \Uu$ such that 
$$
\i \Om^1_{d^{*_\n}} = \ker D_1' \oplus (\ker D_1')^\perp
$$
and $D_1'|_{(\ker D_1')^\perp} : (\ker D_1')^\perp \to {\rm Im}D_1'$ 
is an isomorphism. 
Thus, setting 
$$
\Ww=U_2' \oplus (\ker D_1')^\perp \qquad  
\Ww'= (U_2')^\perp \oplus \ker D_1', 
$$
we see that $T_\Phi \Uu \oplus \i \Om^1_{d^{*_\n}} = \Ww \oplus \Ww'$ 
and $D'|_\Ww: \Ww \to \i \Vv$ is an isomorphism.

Then, we can apply the implicit function theorem 
(after the Banach completion) to 
\[
\Uu \times \Om^1_{d^{*_{\n}}} \rightarrow (\varpi \circ p)^* \Ee, \quad 
(\Psi, a) \mapsto \wFf (\Psi, \n + \i a \cdot \id_L)
\]
and we see that $\widehat \Ss_{(\Phi, \n)}$ is smooth near $(\Phi, 0)$. 
Then, by the Sard--Smale theorem applied to the projection 
$\widehat \Mm'_\Sp \rightarrow \Uu$, 
for every generic $\Psi \in \Uu$ close to $\Phi$, 
$\mathcal{M}'_{\Sp, \Psi}$ is a smooth manifold or an empty set near $[\n]$.

Finally, we explain how to recover the regularity of elements in $\wFf^{-1}(0)$ after the Banach completion. 
By the proof of Theorem \ref{thm:moduli MSpin7}, 
the space $\wFf (\Phi, \,\cdot\,)^{-1}(0)$ 
can be seen as a space of solutions of an overdetermined elliptic equation around $0$. 
Since to be overdetermined elliptic is an open condition, 
if the $\Sp$-structure $\Psi$ is sufficiently close to $\Phi$, 
the space $\wFf (\Psi, \,\cdot\,)^{-1}(0)$ can also be seen  as a space of solutions 
of an overdetermined elliptic equation. 
Hence, elements in $\wFf^{-1}(0)$ around $(\Phi, \n)$ are smooth. 
\end{proof}

\subsection{The orientation of $\Mm'_{\Sp}$} \label{sec:orient}
Let $X^8$ be a compact connected manifold with 
a ${\rm Spin}(7)$-structure $\Phi$ and $L \to X$ be a smooth complex line bundle with a Hermitian metric $h$.
In this section, we show that $\Mm'_{\Sp}$ has a canonical orientation 
if $H^2(\#_\nabla) = \{\, 0 \,\}$ for any $[\n] \in \Mm'_{\Sp}$.

As in Section \ref{sec:suggestSpin7dDT}, 
let $\Aa_0$ be the space of a Hermitian connections of $(L, h)$
and $\Gg_U$ be the group of unitary gauge transformations of $(L,h)$ acting on $\Aa_0$. 
Set $\Bb_0 = \Aa_0/\Gg_U$. 
For any $\n \in \Aa_0$, we define a $\Sp$-structure $\Phi_\n$ 
by \eqref{eq:newSpin7str}. 
Then, we can define an elliptic operator 
$D (\flat_{\Phi_\n}) 
= (\pi^2_{7, \Phi_\n} \circ d, d^*): \i \Om^1 \rightarrow \i \Om^2_{7, \Phi_\n} \oplus \i \Om^0
$ as \eqref{cancpx2}. 
Set 
$$
\det D (\flat_{\Phi_\n}) = \Lambda^{{\rm top}} \ker D (\flat_{\Phi_\n}) 
\otimes (\Lambda^{{\rm top}} {\rm Coker} D (\flat_{\Phi_\n}))^*. 
$$
Define the determinant line bundle $\hat \Ll \rightarrow \Aa_0$ by 
$$
\hat \Ll = \bigsqcup_{\n \in \Aa_0} \det D (\flat_{\Phi_\n}). 
$$
Since the curvature 2-form $F_\n$ is invariant under the action of $\Gg_U$, 
$\hat \Ll$ induces the line bundle $\Ll \rightarrow \Bb_0$. 
We first show that $\Ll$ is a trivial line bundle.

\begin{proposition} \label{prop:trivlb}
We have an isomorphism 
$$
\Ll \cong \Bb_0 \times \det D (\flat_\Phi), 
$$ 
and hence, $\Ll$ is trivial. 
In particular, if we choose an orientation of $\det D (\flat_\Phi)$, 
we obtain a canonical orientation of $\Ll$. 
\end{proposition}

\begin{proof}
As in the proof of Lemma \ref{lem:ind Phin}, 
$((\id_{TX} + (-\i F_\n)^\sharp)^{-1})^* :\Om^2_{7, \Phi_\n} \rightarrow \Om^2_7$
is an isomorphism for any $[\n] \in \Bb_0$ by construction, 
$\det D (\flat_{\Phi_\n})$ is isomorphic to
$\det D' (\flat_{\Phi_\n})$, where 
\begin{align*}
D' (\flat_{\Phi_\n}) =& (((\id_{TX} + (-\i F_\n)^\sharp)^{-1})^* \circ \pi^2_{7, \Phi_\n} \circ d, d^*) \\ 
&: \i \Om^1 \rightarrow \i \Om^2_7 \oplus \i \Om^0. 
\end{align*}
Then, 
$\{ D' (\flat_{\Phi_\n}): \i \Om^1 \rightarrow \i \Om^2_7 \oplus \i \Om^0 \}_{\n \in \Bb_0}$ 
defines a $\Bb_0$-family of 
Fredholm operators in the sense of \cite[Definition 2.4]{JU}. 
Recall that $\Bb_0$ is paracompact and Hausdorff by Remark \ref{rem:metric B0}. 
For $[\n] \in \Bb_0$ and $s \in [0,1]$, define a $\Sp$-structure $\Phi_\n (s)$ by 
$$
\Phi_\n (s) = (\id_{TX} + s (- \i F_\n)^\sharp)^* \Phi. 
$$
Then, $\Phi_\n (0)=\Phi, \Phi_\n (1) = \Phi_\n$ and 
$D' (\flat_{\Phi_\n (s)})$ is elliptic for any $[\n] \in \Bb_0$ and $s \in [0,1]$. 
Hence, 
$\{ D' (\flat_{\Phi_\n (s)}): \i \Om^1 \rightarrow \i \Om^2_7 \oplus \i \Om^0 \}_{\n \in \Bb_0, s \in [0,1]}$ 
gives a homotopy between 
$\{ D' (\flat_{\Phi_\n}): \i \Om^1 \rightarrow \i \Om^2_7 \oplus \i \Om^0 \}_{\n \in \Bb_0}$ 
and 
$\{ D' (\flat_\Phi): \i \Om^1 \rightarrow \i \Om^2_7 \oplus \i \Om^0 \}_{\n \in \Bb_0}$.  
Then, by \cite[Lemma 2.6]{JU}, we obtain 
\[
\bigsqcup_{[\n] \in \Bb_0} \det D' (\flat_{\Phi_\n})
\cong 
\bigsqcup_{[\n] \in \Bb_0} \det D' (\flat_\Phi)
\cong 
\Bb_0 \times \det D' (\flat_\Phi)
= 
\Bb_0 \times \det D (\flat_\Phi). 
\]
Then, the proof is completed. 
\end{proof}

Then, using this proposition, we show that $\Mm'_{\Sp}$ is orientable.

\begin{corollary} \label{cor:oriSpin7}
Suppose that $H^2(\#_\nabla) = \{\, 0 \,\}$ for any $[\n] \in \Mm'_{\Sp}$. 
Then, $\Mm'_{\Sp}$ is a manifold and it admits a canonical orientation 
if we choose an orientation of $\det D$, 
where $D$ is defined by \eqref{cancpxfordDT3}. 
\end{corollary}

\begin{proof}
By the results of previous sections, 
for $\n \in \Ff^{-1}_\Sp (0)$, 
the deformation is controlled by the complex \eqref{cpxfordDT}. 
This is considered to be a subcomplex of the canonical complex $(\flat_{\Phi_\n})$ 
by \eqref{eq:commdiag}. 
By Lemma \ref{lem:2cohom}, it follows that 
\begin{align*}
\det D (\flat_{\Phi_\n}) 
\cong &
\Lambda^{{\rm top}} \ker D \eqref{cpxfordDT} 
\otimes (\Lambda^{{\rm top}} {\rm Coker} D \eqref{cpxfordDT})^*
\otimes \Lambda^{{\rm top}} (\Hh^2_7)^* \\
=& \det D \eqref{cpxfordDT} \otimes \Lambda^{{\rm top}} (\Hh^2_7)^*, 
\end{align*}
and hence,
\begin{align} \label{eq:detn}
\det D \eqref{cpxfordDT} \cong \det D (\flat_{\Phi_\n}) \otimes \Lambda^{{\rm top}} \Hh^2_7. 
\end{align}
If $H^2(\#_\nabla) = \{\, 0 \,\}$, which is equivalent to ${\rm Coker} D (\#_\nabla) = \R$, 
for any $[\n] \in \Mm'_{\Sp}$, 
$\Mm'_{\Sp}$ is a smooth manifold by Theorem \ref{thm:moduli MSpin7}
and the tangent space $T_{[\n]} \Mm'_{\Sp}$ at $[\n] \in \Mm'_\Sp$ is 
identified with $\ker D \eqref{cpxfordDT}$. 
Thus, denoting by $\iota: \Mm'_{\Sp} \hookrightarrow \Bb_0$ the inclusion, 
we see that 
$$
\Lambda^{{\rm top}} T \Mm'_{\Sp} 
\cong \iota^* \Ll \otimes \Lambda^{{\rm top}} \Hh^2_7
\cong \Mm'_{\Sp} \times 
\left(\det D (\flat_\Phi) \otimes \Lambda^{{\rm top}} \Hh^2_7 \right)
$$
by \eqref{eq:detn} and Proposition \ref{prop:trivlb}. 
By \eqref{eq:detn} again, 
$\det D (\flat_\Phi) \otimes \Lambda^{{\rm top}} \Hh^2_7 \cong \det D$, 
and the proof is completed. 
\end{proof}

%%%%%%%%%%%%%%%%%%%%%%%%%%%%%%%%%%%%%%%%%%%%%%%%%%%%%%%%%
%%%%%%%%%%%%%%%%%%%%%%%%%%%%%%%%%%%%%%%%%%%%%%%%%%%%%%%%%
%%%%%%%%%%%%%%%%%%%%%%%%%%%%%%%%%%%%%%%%%%%%%%%%%%%%%%%%%
%%%%%%%%%%%%%%%%%%%%%%%%%%%%%%%%%%%%%%%%%%%%%%%%%%%%%%%%%
%%%%%%%%%%%%%%%%%%%%%%%%%%%%%%%%%%%%%%%%%%%%%%%%%%%%%%%%%
\appendix

%%%%%%%%%%%%%%%%%%%%%%%%%%%%%%%%%%%%%%%%%%%%

\section{
The induced $\Sp$-structure from a $\Sp$-dDT connection}
\label{sec:new Spin7 str} 
In this appendix, we give proofs of statements in Section \ref{sec:suggestSpin7dDT}. 
We often need straightforward computations and some of them will be technically hard to follow. 
We recommend readers to use some softwares to check these computations. 
(We also use the software Maple for confirmation. )

Use the notation of Section \ref{sec:basic}. 
Set $W =\R^8$ and let $g$ be the standard scalar product on $W$. 
For a 2-form $F \in \Lambda^2 W^*$, define $F^\sharp \in {\rm End} (W)$ by 
\begin{align} \label{eq:Fsharp}
g(F^\sharp (u), v) = F(u, v)
\end{align}
for $u,v \in W$. 
Then, 
$F^\sharp$ is skew-symmetric, and hence, 
$\det(I + F^\sharp) >0$, where $I$ is the identity matrix. 
Define a $\Sp$-structure $\Phi_F$ by 
$$
\Phi_F := (I + F^\sharp)^* \Phi, 
$$
where $\Phi$ is the standard $\Sp$-structure given by (\ref{Phi4}).

We first prove the following. 
This can be thought as a generalization of the last statement of 
\cite[Chapter I\hspace{-.1em}V,Theorem 2.20]{HL}.

\begin{proposition} \label{prop:1+F 00 Spin7}
Suppose that a 2-form $F \in \Lambda^2 W^*$ satisfies $\pi^2_7 \left(F - * F^3/6 \right) = 0$. 
Then, 
$$
1 - \frac{* F^4}{24}=0 \quad \mbox{or} \quad \pi^4_7 (F^2) =0. 
$$ 
\end{proposition}
This implies that if a 2-form $F$ satisfies $\pi^2_7 \left(F - * F^3/6 \right) = 0$ and 
$1-*F^4/24 \neq 0$, we have $\pi^4_7 (F^2) =0$.

To prove this, we first show the following.

\begin{lemma} \label{lem:computation}
Suppose that $F$ is of the form of the right hand side of \eqref{eq:2form cong}. 
That is, $F = F_7 + F_{21}$, where 
\begin{align*}
F_7= 2 \lambda^2(\alpha) = e^0 \wedge \alpha + i(\alpha^\sharp) \varphi, \quad
F_{21} = 
\mu_1 e^{01} + \mu_2 e^{23} + \mu_3 e^{45} + \mu_4 e^{67},  
\end{align*}
$\alpha = \alpha_1 e^1 + \alpha_3 e^3 + \alpha_5 e^5 + \alpha_7 e^7$, 
$\sum_{j=1}^4 \mu_j =0$ for 
$\alpha_i, \mu_j \in \R$ 
and $\mu_1 \leq \mu_2 \leq \mu_3 \leq \mu_4$.  
\begin{enumerate}
\item
We have $\pi^4_7 (F^2) = 0$ if and only if
\begin{align*}
\alpha_3 (\mu_3+\mu_4)= \alpha_5 (\mu_2+\mu_4)= \alpha_7 (\mu_2+\mu_3)=0. 
\end{align*}

\item We have 
$\pi^2_7 \left(F - * F^3/6 \right) = 0$ if and only if 
\begin{align}
\label{eq:sol F11}
\left( 
- \sum_{j=1}^4 \mu_j^2 +4|\alpha|^2-4
\right) \alpha_1 
+ \sum_{1 \leq i < j <k \leq 4} \mu_i \mu_j \mu_k &=0, 
\\
\left( \mu_1 \mu_2 + \mu_3 \mu_4 + 2|\alpha|^2-2 \right) \alpha_3 &=0, 
\label{eq:sol F12} \\
\left( \mu_1 \mu_3 + \mu_2 \mu_4 + 2|\alpha|^2-2 \right) \alpha_5 &=0, 
\label{eq:sol F13} \\
\left( \mu_1 \mu_4 + \mu_2 \mu_3 + 2|\alpha|^2-2 \right) \alpha_7 &=0. 
\label{eq:sol F14}
\end{align}

\item We have 
\begin{align*}
\begin{split}
1 - \frac{* F^4}{24}
=&
1-|\alpha|^4
- (\alpha_1^2+\alpha_3^2) (\mu_1 \mu_2 + \mu_3 \mu_4) \\
&- (\alpha_1^2+\alpha_5^2) (\mu_1 \mu_3 + \mu_2 \mu_4) 
 - (\alpha_1^2+\alpha_7^2) (\mu_1 \mu_4 + \mu_2 \mu_3) \\
&- \alpha_1 \left(\mu_1 \mu_2 \mu_3 +\mu_1 \mu_2 \mu_4
+ \mu_1 \mu_3 \mu_4 + \mu_2 \mu_3 \mu_4 \right) 
- \mu_1 \mu_2 \mu_3 \mu_4. 
\end{split}
\end{align*}

\end{enumerate}
\end{lemma}

\begin{proof}
First we show (1). 
By \eqref{eq:F2decomp}, we see that 
$\pi^4_7 (F^2) = 2 \pi^4_7 (F_7 \wedge F_{21})$. 
Then, 
$\pi^4_7 (F^2) = 0$ if and only if
$\left \la F_7 \wedge F_{21}, \lambda^4 (e^i) \right \ra = 0$
for any $i=1, \cdots, 7$. 
Since 
\begin{align*}
F_7 \wedge F_{21} 
=&
e^0 \wedge 
\left \{ \alpha \wedge \left( \mu_2 e^{23} + \mu_3 e^{45} + \mu_4 e^{67} \right) 
+ \mu_1 e^1 \wedge i(\alpha^\sharp) \varphi \right \} \\
&+
i(\alpha^\sharp) \varphi \wedge \left( \mu_2 e^{23} + \mu_3 e^{45} + \mu_4 e^{67} \right) 
\end{align*}
and $\sqrt{8} \lambda^4 (e^i) = e^0 \wedge i(e_i) *_7 \varphi - e^i \wedge \varphi$, 
we have 
$$
\left \la F_7 \wedge F_{21}, \sqrt{8} \lambda^4 (e^i) \right \ra
=I_1 + I_2, 
$$
where 
\begin{align*}
I_1=&\left \la \alpha \wedge \left( \mu_2 e^{23} + \mu_3 e^{45} + \mu_4 e^{67} \right)
+ \mu_1 e^1 \wedge i(\alpha^\sharp) \varphi, 
i(e_i) * \varphi \right \ra \\
I_2=&-
\left \la i(\alpha^\sharp) \varphi \wedge \left( \mu_2 e^{23} + \mu_3 e^{45} 
+ \mu_4 e^{67} \right), 
e^i \wedge \varphi \right \ra. 
\end{align*}
Then, since 
\begin{align*}
I_1=&
\left \la \alpha, \mu_2 i(e_i) (e^{45}+e^{67}) + \mu_3 i(e_i) (e^{23}+e^{67}) 
+ \mu_4 i(e_i) (e^{23}+e^{45}) \right \ra \\
&+
\mu_1 \left \la i(\alpha^\sharp) \varphi, 
i(e_i) (-e^{357}+e^{346}+e^{256}+e^{247}) \right \ra, \\ 
I_2=&
-
\left \la i(\alpha^\sharp) \varphi, \right. 
\left. \mu_2 i(e_3) i(e_2) (e^i \wedge \varphi) 
+\mu_3 i(e_5) i(e_4) (e^i \wedge \varphi) 
+\mu_4 i(e_7) i(e_6) (e^i \wedge \varphi) 
\right \ra, 
\end{align*}
it is straightforward to obtain 
$\left \la F_7 \wedge F_{21}, \sqrt{8} \lambda^4 (e^i) \right \ra =0$
for $i=1,3,5,7$ and 
$$
\left \la F_7 \wedge F_{21}, \sqrt{8} \lambda^4 (e^i) \right \ra 
=
\left\{ \begin{array}{ll}
3 \alpha_3 (\mu_3+\mu_4) & \mbox{for } i=2, \\
3 \alpha_5 (\mu_2+\mu_4) & \mbox{for } i=4, \\
3 \alpha_7 (\mu_2+\mu_3) & \mbox{for } i=6. \\
\end{array} \right.
$$
Then, the proof of (1) is completed. 

Next, we prove (2). 
Note that 
$\pi^2_7 \left(F - * F^3/6 \right) = 0$ if and only if
$\left \la F - * F^3/6, 2 \lambda^2 (e^i) \right \ra = 0$
for any $i=1, \cdots, 7$. 
Since $F_7 =2 \lambda^2 (\alpha)$, we have 
\begin{align} \label{eq:comp1}
\left \la F, 2 \lambda^2 (e^i) \right \ra = 
\left\{ \begin{array}{ll}
4 \alpha_i & \mbox{for } i=1,3,5,7, \\
0             & \mbox{for } i=2,4,6. 
\end{array} \right. 
\end{align}
Next, we compute 
$\left \la * F^3, 2 \lambda^2 (e^i) \right \ra$. 
Recall that 
$F^3 = F_7^3 + 3 F_7^2 \wedge F_{21} + 3 F_7 \wedge F_{21}^2 + F_{21}^3$. 
By Proposition \ref{prop:2form norm}, we have 
$
F_7^3=(3/2) |F_7|^2 *F_7 = 12 |\alpha|^2 * \lambda^2 (\alpha). 
$
Hence, 
\begin{align} \label{eq:comp2}
\left \la * F_7^3, 2 \lambda^2 (e^i) \right \ra 
= 
\left\{ \begin{array}{ll}
24 |\alpha|^2 \alpha_i & \mbox{for } i=1,3,5,7, \\
0             & \mbox{for } i=2,4,6. 
\end{array} \right. 
\end{align}
By \eqref{eq:F2decomp}, 
we see that $F_7^2$ is self dual and 
$\left \la \Lambda^2_7 W^* \wedge \Lambda^2_{21} W^*, 
\Lambda^2_7 W^* \wedge \Lambda^2_7 W^* \right \ra =0. $
Then, we obtain 
\begin{align} \label{eq:comp3}
\left \la *( F_7^2 \wedge F_{21}), \lambda^2 (e^i) \right \ra
=
* \left( F_{21} \wedge \lambda^2 (e^i) \wedge F_7^2 \right)
= \left \la F_{21} \wedge \lambda^2 (e^i), F_7^2 \right \ra =0.
\end{align}

We also compute 
\begin{align*}
\left \la *( F_7 \wedge F_{21}^2), 2 \lambda^2 (e^i) \right \ra
=
* \left( F_{21}^2 \wedge F_7 \wedge 2 \lambda^2 (e^i) \right)
= \left \la F_{21}^2, F_7 \wedge 2 \lambda^2 (e^i) \right \ra. 
\end{align*}
Since 
\begin{align*}
F_{21}^2
=&2 \left(\mu_1 \mu_2 e^{0123} + \mu_1 \mu_3 e^{0145} + \mu_1 \mu_4 e^{0167} \right. \\
&\left. + \mu_2 \mu_3 e^{2345} 
+ \mu_2 \mu_4 e^{2367} + \mu_3 \mu_4 e^{4567} \right), \\
F_7 \wedge 2 \lambda^2 (e^i) 
=& 
\left( e^0 \wedge \alpha + i(\alpha^\sharp) \varphi \right) \wedge 
\left( e^{0 i} + i(e_i) \varphi \right), 
\end{align*}
and 
\begin{align*}
&\left \la e^{0123}, F_7 \wedge 2 \lambda^2 (e^i) \right \ra \\
=&
F_7 (e_0, e_1) \cdot (2 \lambda^2 (e^i)) (e_2,e_3)
-F_7 (e_0, e_2) \cdot (2 \lambda^2 (e^i)) (e_1,e_3) \\
&+F_7 (e_0, e_3) \cdot (2 \lambda^2 (e^i)) (e_1,e_2)
+F_7 (e_1, e_2) \cdot (2 \lambda^2 (e^i)) (e_0,e_3) \\
&-F_7 (e_1, e_3) \cdot (2 \lambda^2 (e^i)) (e_0,e_2)
+F_7 (e_2, e_3) \cdot (2 \lambda^2 (e^i)) (e_0,e_1) \\
=&
2 \delta_{i1} \alpha_1 + 2 \delta_{i3} \alpha_3, 
\end{align*}
it is straightforward to obtain 
\begin{align} \label{eq:comp4}
\begin{split}
&\left \la *( F_7 \wedge F_{21}^2), 2 \lambda^2 (e^i) \right \ra \\
=&
4 \left \{ 
\mu_1 \mu_2 (\delta_{i1} \alpha_1 + \delta_{i3} \alpha_3) 
+\mu_1 \mu_3 (\delta_{i1} \alpha_1 + \delta_{i5} \alpha_5) 
+\mu_1 \mu_4 (\delta_{i1} \alpha_1 + \delta_{i7} \alpha_7)
\right. \\
& 
\left. +\mu_2 \mu_3 (\delta_{i1} \alpha_1 + \delta_{i7} \alpha_7) 
+\mu_2 \mu_4 (\delta_{i1} \alpha_1 + \delta_{i5} \alpha_5) 
+\mu_3 \mu_4 (\delta_{i1} \alpha_1 + \delta_{i3} \alpha_3) 
\right \}. 
\end{split}
\end{align}
Since 
\begin{align*}
* F_{21}^3
=&6 \left(\mu_1 \mu_2 \mu_3 e^{67} +\mu_1 \mu_2 \mu_4 e^{45}
+ \mu_1 \mu_3 \mu_4 e^{23} + \mu_2 \mu_3 \mu_4 e^{01} \right), 
\end{align*}
we also have
\begin{align} \label{eq:comp5}
\left \la * F_{21}^3, 2 \lambda^2 (e^i) \right \ra 
= 6 \delta_{i1} \left( \mu_1 \mu_2 \mu_3 + \mu_1 \mu_2 \mu_4 
+ \mu_1 \mu_3 \mu_4 + \mu_2 \mu_3 \mu_4 \right). 
\end{align}
Hence, by \eqref{eq:comp2}-\eqref{eq:comp5}, we obtain 
$\left \la - * (F^3/6), 2 \lambda^2 (e^i) \right \ra = 0$ for $i=2,4,6$ and 
\begin{align*}
&\left \la - * (F^3/6), 2 \lambda^2 (e^i) \right \ra \\
=&
-4 |\alpha|^2 \alpha_i 
- \delta_{i1} 
\left(\mu_1 \mu_2 \mu_3 +\mu_1 \mu_2 \mu_4
+ \mu_1 \mu_3 \mu_4 + \mu_2 \mu_3 \mu_4 \right)
\\
&-2 \left \{ 
\mu_1 \mu_2 (\delta_{i1} \alpha_1 + \delta_{i3} \alpha_3) 
+\mu_1 \mu_3 (\delta_{i1} \alpha_1 + \delta_{i5} \alpha_5) 
+\mu_1 \mu_4 (\delta_{i1} \alpha_1 + \delta_{i7} \alpha_7)
\right. \\
& 
\left. +\mu_2 \mu_3 (\delta_{i1} \alpha_1 + \delta_{i7} \alpha_7) 
+\mu_2 \mu_4 (\delta_{i1} \alpha_1 + \delta_{i5} \alpha_5) 
+\mu_3 \mu_4 (\delta_{i1} \alpha_1 + \delta_{i3} \alpha_3) 
\right \}
\end{align*}
for $i=1,3,5,7$. 
This together with \eqref{eq:comp1} implies (2). Note that 
$$
2 (\mu_1 \mu_2 + \mu_1 \mu_3 + \mu_1 \mu_4 
+ \mu_2 \mu_3 + \mu_2 \mu_4 + \mu_3 \mu_4) = - \sum_{j=1}^4 \mu_j^2
$$
since 
$0 = (\mu_1 + \mu_2 + \mu_3 + \mu_4)^2 = \sum_{j=1}^4 \mu_j^2 
+ 2 \sum_{1 \leq i < j \leq 4} \mu_i \mu_j$.

Finally, we prove (3). 
By Proposition \ref{prop:2form norm}, we have 
$$
* F_7^4 =\frac{3}{2} |F_7|^4 = 24|\alpha|^4. 
$$
By Proposition \ref{prop:2form norm}, we have 
$* F_7^3 \in \Lambda^2_7 W^*$. Hence, we have 
$$
*(F_7^3 \wedge F_{21}) = 0. 
$$
Since $F_7=2 \lambda^2 (\alpha) = \sum_{i=1,3,5,7} \alpha_i \cdot 2 \lambda^2 (e^i)$, 
we have by \eqref{eq:comp4}
\begin{align*}
* (F_7^2 \wedge F_{21}^2) 
=& \left \la *(F_7 \wedge F_{21}^2), F_7 \right \ra \\
=&
4 (\alpha_1^2+\alpha_3^2) (\mu_1 \mu_2 + \mu_3 \mu_4)
+4 (\alpha_1^2+\alpha_5^2) (\mu_1 \mu_3 + \mu_2 \mu_4) \\
&+4 (\alpha_1^2+\alpha_7^2) (\mu_1 \mu_4 + \mu_2 \mu_3). 
\end{align*}
By \eqref{eq:comp5}, we have  
$$
* (F_7 \wedge F_{21}^3)= \left \la *F_{21}^3, F_7 \right \ra
= 6 \alpha_1 \left(\mu_1 \mu_2 \mu_3 +\mu_1 \mu_2 \mu_4
+ \mu_1 \mu_3 \mu_4 + \mu_2 \mu_3 \mu_4 \right). 
$$
We also compute 
$$
* F_{21}^4= 24 \mu_1 \mu_2 \mu_3 \mu_4. 
$$
From these equations, we obtain (3). 
\end{proof}

\begin{lemma} \label{lem:sol F1}
Suppose that $F$ is of the form of the right hand side of \eqref{eq:2form cong}. 
Then, $\pi^2_7 \left(F - * F^3/6 \right) = 0$ if and only if 
one of the following conditions holds. 
\begin{enumerate}
\item
$\alpha_3=\alpha_5=\alpha_7=0$ and 
\begin{align} \label{eq:solsol}
\left( - \sum_{j=1}^4 \mu_j^2 +4 \alpha_1^2-4 \right) \alpha_1 
+ \sum_{1 \leq i < j <k \leq 4} \mu_i \mu_j \mu_k
=0. 
\end{align}

\item
$\alpha_3 \neq 0, \alpha_5=\alpha_7=0$, 
$\mu_1 \mu_2 + \mu_3 \mu_4 + 2 |\alpha|^2-2 =0$, 
$\mu_3+\mu_4 \neq 0$
and 
$\alpha_1= - (\mu_2+\mu_3) (\mu_2+\mu_4)/(2 (\mu_3+\mu_4)).$

\item
$\alpha_5 \neq 0, \alpha_3=\alpha_7=0$, 
$\mu_1 \mu_3 + \mu_2 \mu_4 + 2 |\alpha|^2-2  =0$
and 
\begin{enumerate}
\item
$\mu_2+\mu_4=0$, or 
\item
$\mu_2+\mu_4 \neq 0$
and 
$\alpha_1= - (\mu_2+\mu_3) (\mu_3+\mu_4)/(2 (\mu_2+\mu_4)).$
\end{enumerate}

\item
$\alpha_7 \neq 0, \alpha_3=\alpha_5=0$, 
$\mu_1 \mu_4 + \mu_2 \mu_3 + 2 |\alpha|^2-2=0$
and 
\begin{enumerate}
\item
$\mu_2+\mu_3=0$, or 
\item
$\mu_2+\mu_3 \neq 0$
and 
$\alpha_1= - (\mu_2+\mu_4) (\mu_3+\mu_4)/(2 (\mu_2+\mu_3)).$
\end{enumerate}

\item
$\alpha_3, \alpha_5 \neq 0, \alpha_7=0$, 
$\mu_2+\mu_4 \neq 0$, 
$\mu_2=\mu_3, \alpha_1=-\mu_2, 
\alpha_3^2+\alpha_5^2=1$. 

\item
$\alpha_3, \alpha_7 \neq 0, \alpha_5=0$, 
$\mu_2+\mu_3 \neq 0$, 
$\mu_2=\mu_4, \alpha_1=-\mu_2, 
\alpha_3^2+\alpha_7^2=1$.

\item
$\alpha_5, \alpha_7 \neq 0, \alpha_3=0$ and 
\begin{enumerate}
\item
$(\mu_1, \mu_2, \mu_3, \mu_4)=(-\mu, -\mu, \mu, \mu)$
and 
$|\alpha|^2=1+ \mu^2$ for $\mu >0$, or 
\item
$\mu_2+\mu_3, \mu_2+\mu_4 \neq 0$, 
$\mu_3=\mu_4, \alpha_1=-\mu_3, 
\alpha_5^2+\alpha_7^2=1$, or
\item
$\mu_2+\mu_3, \mu_2+\mu_4 \neq 0$, 
$\mu_1=\mu_2, \alpha_1=-\mu_1, 
\alpha_5^2+\alpha_7^2=1$.  
\end{enumerate}

\item
$(\mu_1, \mu_2, \mu_3, \mu_4)=(-3 \mu, \mu, \mu, \mu), 
\alpha_1= - \mu$ for $\mu > 0$
and 
$\alpha_3^2+\alpha_5^2+\alpha_7^2=1$.

\item
$(\mu_1, \mu_2, \mu_3, \mu_4)=(-\mu, -\mu, -\mu, 3 \mu), 
\alpha_1= \mu$ for $\mu > 0$ 
and 
$\alpha_3^2+\alpha_5^2+\alpha_7^2=1$.

\item
$\mu_1=\mu_2=\mu_3=\mu_4=0$
and $|\alpha|^2=1$. 
\end{enumerate}
\end{lemma}

\begin{proof}
We rewrite equations 
in Lemma \ref{lem:computation} (2)
to prove the statement. 
The proof is done by division into possible cases systematically. 

\vspace{2mm}
\noindent
(I) \textbf{The case where $\alpha_3=\alpha_5=\alpha_7=0$. }
In this case, \eqref{eq:sol F12},\eqref{eq:sol F13} and \eqref{eq:sol F14} 
are satisfied and \eqref{eq:sol F11} is equivalent to \eqref{eq:solsol}. Thus, we obtain (1). 
Note that (1) includes (10) with $\alpha_3=\alpha_5=\alpha_7=0$. 

\vspace{2mm}
\noindent
(II) \textbf{The case where one of $\alpha_3$, $\alpha_5$ or $\alpha_7$ is $0$. }
\begin{itemize}
\item[(i)] \textbf{The case where $\alpha_3 \neq 0$ and $\alpha_5=\alpha_7=0$. }
In this case, \eqref{eq:sol F13} and \eqref{eq:sol F14} 
are satisfied 
and \eqref{eq:sol F12} are equivalent to 
\begin{align} \label{eq:sol F14.5}
2|\alpha|^2-2 = -(\mu_1 \mu_2 + \mu_3 \mu_4).  
\end{align}
Substituting this into \eqref{eq:sol F11}, we have 
\begin{align*}
\left \{ - (\mu_1 + \mu_2)^2 - (\mu_3 + \mu_4)^2 \right \} \alpha_1  
+ \sum_{1 \leq i < j <k \leq 4} \mu_i \mu_j \mu_k =0. 
\end{align*}
Since 
$(\mu_1 + \mu_2)^2 = (\mu_3+\mu_4)^2$ 
and 
\begin{align} \label{eq:muijk}
\begin{split}
\sum_{1 \leq i < j <k \leq 4} \mu_i \mu_j \mu_k 
&=
\mu_1 \mu_2 (\mu_3 + \mu_4) + (\mu_1 + \mu_2) \mu_3 \mu_4 \\
&=
(\mu_3 + \mu_4) (\mu_1 \mu_2 - \mu_3 \mu_4) \\
&=
(\mu_3 + \mu_4) (- (\mu_2+\mu_3+\mu_4) \mu_2 - \mu_3 \mu_4) \\
&=
- (\mu_3 + \mu_4) (\mu_2 + \mu_3) (\mu_2 + \mu_4), 
\end{split}
\end{align}
if follows that 
$$
(\mu_3 + \mu_4) \left \{ 2 (\mu_3 + \mu_4) \alpha_1 
+ (\mu_2 + \mu_3) (\mu_2 + \mu_4) \right \}=0.  
$$
Now, we further divide the situation. 
\begin{itemize}
\item[(A)] If $\mu_3 + \mu_4=0$, we have $\mu_1 + \mu_2=0$ by $\sum_{j=1}^4 \mu_j=0$. 
Then, $\mu_1 \leq \mu_2 \leq \mu_3 \leq \mu_4$ implies that 
$\mu_1 = - \mu_2 \geq - \mu_3 = \mu_4$. 
Thus, we obtain $\mu_1 = \mu_4$, 
and hence, $\mu_1 = \mu_2 = \mu_3 = \mu_4=0$.
This together with \eqref{eq:sol F14.5} implies that 
$|\alpha|^2=1$, which is (10) with $\alpha_3 \neq 0, \alpha_5=\alpha_7=0$. 
\item[(B)] If $\mu_3 + \mu_4 \neq 0$, we obtain (2). 
Note that (2) includes (8) and (9) with $\alpha_3 \neq 0, \alpha_5=\alpha_7=0$. 
\end{itemize}

\item[(ii)]  \textbf{The case where $\alpha_5 \neq 0$ and $\alpha_3=\alpha_7=0$. }
In this case, we obtain (3) similarly to the case (i). 

\item[(iii)] \textbf{The case where $\alpha_7 \neq 0$ and $\alpha_3=\alpha_5=0$. }
In this case, we obtain (4) similarly to the case (i). 
\end{itemize}

Note that 
we cannot deduce $\mu_1 = \mu_2 = \mu_3 = \mu_4=0$ 
from $\mu_2+\mu_4=0$ or $\mu_2+\mu_3=0$. 
A counter example is $(\mu_1, \mu_2, \mu_3, \mu_4)=(-1, -1, 1, 1)$. 
Further, note that 
(3)-(a), (3)-(b), (4)-(a), (4)-(b) include 
(10) with $\alpha_5 \neq 0, \alpha_3=\alpha_7=0$, 
(8) and (9) with $\alpha_5 \neq 0, \alpha_3=\alpha_7=0$, 
(10) with $\alpha_7 \neq 0, \alpha_3=\alpha_5=0$, 
(8) and (9) with $\alpha_7 \neq 0, \alpha_3=\alpha_5=0$, respectively.

\vspace{2mm}
\noindent
(III) \textbf{The case where two of $\alpha_3$, $\alpha_5$ or $\alpha_7$ are $0$. }
\begin{itemize}
\item[(i)] \textbf{ The case where $\alpha_3, \alpha_5 \neq 0$ and $\alpha_7=0$. }
In this case, \eqref{eq:sol F12} and \eqref{eq:sol F13} are equivalent to 
\begin{equation}\label{eq:sol F15}
\begin{aligned}
&2|\alpha|^2-2 = -(\mu_1 \mu_2 + \mu_3 \mu_4) \quad \mbox{and}\\
&2|\alpha|^2-2 = -(\mu_1 \mu_3 + \mu_2 \mu_4), 
\end{aligned}
\end{equation}
respectively. 
Substituting these into \eqref{eq:sol F11} as in the case (II)-(i), we obtain 
\begin{align} \label{eq:sol F16}
(\mu_3 + \mu_4) \left \{ 2 (\mu_3 + \mu_4) \alpha_1 
+ (\mu_2 + \mu_3) (\mu_2 + \mu_4) \right \} &=0, \\
\label{eq:sol F17}
(\mu_2 + \mu_4) \left \{ 2 (\mu_2 + \mu_4) \alpha_1 
+ (\mu_2 + \mu_3) (\mu_3 + \mu_4) \right \} &=0. 
\end{align}
Now, we further divide the situation. 
\begin{itemize}
\item[(A)] If $\mu_3 + \mu_4=0$, by the same argument as in the case (II)-(i)-(A), 
we have $\mu_1=\mu_2=\mu_3=\mu_4=0$ and $|\alpha|^2=1$, 
which is (10) with $\alpha_3, \alpha_5 \neq 0, \alpha_7=0$. 
\item[(B)] If $\mu_2 + \mu_4=0$, 
we have $\mu_1 + \mu_3=0$ by $\sum_{j=1}^4 \mu_j=0$. 
Then, since 
\[
\begin{aligned}
&\mu_1 \mu_2 + \mu_3 \mu_4 = 2 \mu_1 \mu_2, \\
&\mu_1 \mu_3 + \mu_2 \mu_4 = -\mu_1^2 - \mu_2^2, 
\end{aligned}
\]
\eqref{eq:sol F15} implies that $\mu_1= -\mu_2$. 
Then, we obtain $\mu_3+\mu_4=0$. 
Thus, this case is included in the previous case. 
\item[(C)] If $\mu_3 + \mu_4, \mu_2 + \mu_4 \neq 0$, \eqref{eq:sol F16} and \eqref{eq:sol F17}
imply that 
\begin{align} \label{eq:sol F18}
\alpha_1 = - \frac{(\mu_2 + \mu_3) (\mu_2 + \mu_4)}{2 (\mu_3 + \mu_4)}
= - \frac{(\mu_2 + \mu_3) (\mu_3 + \mu_4)}{2 (\mu_2 + \mu_4)}. 
\end{align}
Then, we have  
\begin{align*}
&(\mu_2 + \mu_3) \left \{ (\mu_2 + \mu_4)^2 - (\mu_3 + \mu_4)^2 \right \} \\
=&
(\mu_2 + \mu_3) (\mu_2 - \mu_3) (\mu_2 + \mu_3 + 2 \mu_4) \\
=& (\mu_2 + \mu_3) (\mu_2 - \mu_3) (- \mu_1 + \mu_4)
=0.
\end{align*}
We patiently divide this situation. 
\begin{itemize}
\item[(C1)] If $\mu_2 + \mu_3=0$, 
we have $\mu_1 + \mu_4=0$ by $\sum_{j=1}^4 \mu_j=0$. 
The equation \eqref{eq:sol F18} implies that $\alpha_1=0$. 
The equation \eqref{eq:sol F15} is equivalent to 
$$
2|\alpha|^2-2 = -2 \mu_1 \mu_2 = 2 \mu_1 \mu_2,  
$$
which implies that $|\alpha|^2=1$ and $\mu_1 \mu_2 =0$. 
If $\mu_1=0$, 
$\sum_{j=1}^4 \mu_j=0$ and $\mu_1 \leq \mu_2 \leq \mu_3 \leq \mu_4$ imply that 
$\mu_1 = \mu_2 = \mu_3 = \mu_4 =0$, 
which contradicts $\mu_3 + \mu_4 \neq 0$. 
If $\mu_2=0$, we have $\mu_3=-\mu_2=0$. 
Thus, we obtain (5) with $\mu_2=0$. 
\item[(C2)] If $\mu_2 - \mu_3=0$, 
we have $\mu_1 + \mu_4= -2 \mu_2$ by $\sum_{j=1}^4 \mu_j=0$. 
The equation \eqref{eq:sol F18} implies that $\alpha_1=-\mu_2$. 
Then, \eqref{eq:sol F15} is equivalent to 
\[
\begin{aligned}
|\alpha|^2
=&\alpha_1^2 + \alpha_3^2 + \alpha_5^2\\
=&1- \frac{\mu_2 (\mu_1+\mu_4)}{2}
\end{aligned}
\]
which implies (5). 
Note that (5) includes (8) and (9) with $\alpha_3, \alpha_5 \neq 0, \alpha_7=0$. 
\item[(C3)] If $- \mu_1 + \mu_4=0$, 
$\sum_{j=1}^4 \mu_j=0$ and $\mu_1 \leq \mu_2 \leq \mu_3 \leq \mu_4$ imply that 
$\mu_1 = \mu_2 = \mu_3 = \mu_4 =0$, 
which contradicts $\mu_3 + \mu_4 \neq 0$. 
\end{itemize}
\end{itemize}

\item[(ii)] \textbf{The case where $\alpha_3, \alpha_7 \neq 0$ and $\alpha_5=0$. }
In this case, we obtain (10) with $\alpha_3, \alpha_7 \neq 0, \alpha_5=0$, 
(6) and (9) with $\alpha_3, \alpha_7 \neq 0, \alpha_5=0$. 
Note that (6) includes (8) with $\alpha_3, \alpha_7 \neq 0, \alpha_5=0$. 
\item[(iii)] \textbf{The case where $\alpha_5, \alpha_7 \neq 0$ and $\alpha_3=0$. }
In this case, 
we obtain (10) with $\alpha_5, \alpha_7 \neq 0, \alpha_3=0$ and (7). 
Note that (7)-(b) and (7)-(c) include (8) and (9) with $\alpha_5, \alpha_7 \neq 0, \alpha_3=0$, respectively. 
\end{itemize}

\vspace{2mm}
\noindent
(IV) \textbf{The case where $\alpha_3, \alpha_5, \alpha_7 \neq 0$. }
In this case, \eqref{eq:sol F12}, \eqref{eq:sol F13}, \eqref{eq:sol F14} are equivalent to 
\begin{align} \label{eq:sol F19}
\begin{split}
2|\alpha|^2-2 &= -(\mu_1 \mu_2 + \mu_3 \mu_4), \\
2|\alpha|^2-2 &= -(\mu_1 \mu_3 + \mu_2 \mu_4), \\
2|\alpha|^2-2 &= -(\mu_1 \mu_4 + \mu_2 \mu_3), 
\end{split}
\end{align}
respectively. 
Substituting these into \eqref{eq:sol F11} as above, we obtain 
\begin{align} \label{eq:sol F110}
(\mu_3 + \mu_4) \left \{ 2 (\mu_3 + \mu_4) \alpha_1 
+ (\mu_2 + \mu_3) (\mu_2 + \mu_4) \right \} &=0, \\
\label{eq:sol F111}
(\mu_2 + \mu_4) \left \{ 2 (\mu_2 + \mu_4) \alpha_1 
+ (\mu_2 + \mu_3) (\mu_3 + \mu_4) \right \} &=0, \\
\label{eq:sol F112}
(\mu_2 + \mu_3) \left \{ 2 (\mu_2 + \mu_3) \alpha_1 
+ (\mu_2 + \mu_4) (\mu_3 + \mu_4) \right \} &=0. 
\end{align}
Now, we further divide the situation. 
\begin{itemize}
\item[(i)] Suppose that $\mu_3 + \mu_4=0$. 
Then, by the same argument as above, 
we have $\mu_1=\mu_2=\mu_3=\mu_4=0$ and $|\alpha|^2=1$, 
which is (10) with $\alpha_3, \alpha_5, \alpha_7 \neq 0$. 
\item[(ii)] If $\mu_2 + \mu_4=0$, 
we have $\mu_1 + \mu_3=0$. Then, 
\eqref{eq:sol F19} implies that 
$$
2|\alpha|^2-2=-2 \mu_1 \mu_2 = 2 \mu_1 \mu_2. 
$$
Hence, we obtain $|\alpha|^2=1$ and $\mu_1 \mu_2=0$. 
If $\mu_1=0$ or $\mu_2=-\mu_4=0$, 
$\sum_{j=1}^4 \mu_j=0$ and $\mu_1 \leq \mu_2 \leq \mu_3 \leq \mu_4$ imply that 
$\mu_1 = \mu_2 = \mu_3 = \mu_4 =0$. 
Hence, we obtain (10) with $\alpha_3, \alpha_5, \alpha_7 \neq 0$. 
\item[(iii)] Similarly, if $\mu_2 + \mu_3=0$, we obtain (10) with $\alpha_3, \alpha_5, \alpha_7 \neq 0$. 
\item[(iv)] Now, suppose that 
$\mu_2 + \mu_3, \mu_2 + \mu_4, \mu_3 + \mu_4 \neq 0$. 
Then, \eqref{eq:sol F110}, \eqref{eq:sol F111} and \eqref{eq:sol F112} 
imply that
\begin{equation}\label{eq:sol F113}
\begin{aligned}
\alpha_1
=& - \frac{(\mu_2 + \mu_3) (\mu_2 + \mu_4)}{2 (\mu_3 + \mu_4)}\\
=& - \frac{(\mu_2 + \mu_3) (\mu_3 + \mu_4)}{2 (\mu_2 + \mu_4)}
= - \frac{(\mu_2 + \mu_4) (\mu_3 + \mu_4)}{2 (\mu_2 + \mu_3)}. 
\end{aligned}
\end{equation}
Then, as in the case (III)-(i)-(C), 
\eqref{eq:sol F113} implies that 
\begin{align}
\label{eq:sol F114}
(\mu_2 + \mu_3) (\mu_2 - \mu_3) (- \mu_1 + \mu_4) &=0, \\
\label{eq:sol F115}
(\mu_2 + \mu_4) (\mu_2 - \mu_4) (- \mu_1 + \mu_3) &=0.
\end{align}
We patiently divide this situation. 
\begin{itemize}
\item[(A)] Suppose that $\mu_2 - \mu_3=0$ and $\mu_2 - \mu_4=0$. 
Then, we have 
$(\mu_1, \mu_2, \mu_3, \mu_4)=(-3 \mu, \mu, \mu, \mu)$
for $\mu > 0$. 
The equations \eqref{eq:sol F19} and \eqref{eq:sol F113} imply that 
$$
2 |\alpha|^2-2=2 \mu^2, \quad \mbox{and} \quad \alpha_1= -\mu. 
$$
Hence, we obtain (8) with $\alpha_3, \alpha_5, \alpha_7 \neq 0$. 
\item[(B)] Similarly, if we suppose $\mu_2 - \mu_3 = - \mu_1 + \mu_3=0$, we obtain (9) with $\alpha_3, \alpha_5, \alpha_7 \neq 0$. 
The equations 
$- \mu_1 + \mu_4=\mu_2 - \mu_4=0$ and  
$- \mu_1 + \mu_4=- \mu_1 + \mu_3=0$
implies that 
$\mu_1=\mu_2=\mu_3=\mu_4=0$, which contradicts 
$\mu_2 + \mu_3 \neq 0$. 
\end{itemize}
\end{itemize}
Then, the proof is completed. 
\end{proof}

Note that we can solve \eqref{eq:solsol} with respect to $\alpha_1$ explicitly. 

\begin{lemma} \label{lem:eplicit a1}
The equation \eqref{eq:solsol} holds if and only if 
$$
\alpha_1= \sqrt{\frac{\sum_{j=1}^4 \mu_j^2 + 4}{3}}
\cos \left( \frac{1}{3} \arccos 
\left( \frac{-3 \sqrt{3} \sum_{1 \leq i < j <k \leq 4} \mu_i \mu_j \mu_k}
{\left( \sum_{j=1}^4 \mu_j^2+4 \right)^{3/2}} \right) + \frac{2 \l \pi}{3} \right)
$$
for $\l=0,1,2$. 
\end{lemma}

\begin{proof}
We can obtain this by the Vi\`ete's formula for a cubic equation. 
To apply this, we have to show that 
$
3 \sqrt{3} \left | \sum_{1 \leq i < j <k \leq 4} \mu_i \mu_j \mu_k \right |
< \left | \sum_{j=1}^4 \mu_j^2+4 \right |^{3/2}. 
$
This follows from the following more general statement. 
\end{proof}

\begin{lemma} \label{lem:dDT estimate sub}
We have 
$$
3 \sqrt{3} \left | \sum_{1 \leq i < j <k \leq 4} \mu_i \mu_j \mu_k \right | 
\leq \left | \sum_{j=1}^4 \mu_j^2 \right |^{3/2}. 
$$
\end{lemma}

\begin{proof}
Set $\nu_1 = \mu_2 + \mu_3 , \nu_2 = \mu_2 + \mu_4$ and $\nu_3= \mu_3 + \mu_4$. 
Then, by \eqref{eq:muijk}, we have 
$
\sum_{1 \leq i < j <k \leq 4} \mu_i \mu_j \mu_k = - \nu_1 \nu_2 \nu_3. 
$
We also compute 
\begin{align*}
\sum_{j=1}^4 \mu_j^2
=& (\mu_2 + \mu_3 + \mu_4)^2 + \mu_2^2 + \mu_3^2 + \mu_4^2 \\
=& (\mu_2 + \mu_3)^2 + (\mu_2 + \mu_4)^2 + (\mu_3 + \mu_4)^2 
= \nu_1^2+\nu_2^2+\nu_3^2.
\end{align*}
By the inequality of arithmetic and geometric means, we have 
$$
27 |\nu_1 \nu_2 \nu_3| \leq (|\nu_1|+ |\nu_2| + |\nu_3|)^3. 
$$
We also have 
$3 (|\nu_1|^2 + |\nu_2|^2 + |\nu_3|^2) - (|\nu_1|+ |\nu_2| + |\nu_3|)^2 
= (|\nu_1|-|\nu_2|)^2 + (|\nu_2|-|\nu_3|)^2+ (|\nu_3|-|\nu_1|)^2 \geq 0$. 
Hence, we obtain 
$$
27 |\nu_1 \nu_2 \nu_3| \leq \left( 3 (|\nu_1|^2 + |\nu_2|^2 + |\nu_3|^2) \right)^{3/2}
= 3 \sqrt{3} \left( |\nu_1|^2 + |\nu_2|^2 + |\nu_3|^2 \right)^{3/2}. 
$$
\end{proof}

Now, we show Proposition \ref{prop:1+F 00 Spin7}.

\begin{proof}[Proof of Proposition \ref{prop:1+F 00 Spin7}]
Since 
$\pi^2_7 \left(F - * F^3/6 \right)$, 
$* F^4/24$ and $\pi^4_7 (F^2)$
are $\Sp$-equivariant, 
we may assume that 
$F$ is of the form of the right hand side of \eqref{eq:2form cong}. 
Thus, we may show that  
$1 - * F^4/24=0$ or $\pi^4_7 (F^2) =0$
for all cases given in Lemma \ref{lem:sol F1}.

We now show that $1 - * F^4/24=0$ for case (2). 
By Lemma \ref{lem:computation} (3), we have 
$$
1 - \frac{* F^4}{24}
= 1-|\alpha|^4 - \mu_1 \mu_2 \mu_3 \mu_4 +I_1 + I_2, 
$$
where 
\begin{align*}
I_1 &=- |\alpha|^2 (\mu_1 \mu_2 + \mu_3 \mu_4), \\
I_2 &=
- \alpha_1^2 (\mu_1 \mu_3 + \mu_2 \mu_4 + \mu_1 \mu_4 + \mu_2 \mu_3) 
- \alpha_1 \sum_{1 \leq i < j <k \leq 4} \mu_i \mu_j \mu_k \\
&=
- \alpha_1^2 (\mu_1 + \mu_2) (\mu_3 + \mu_4) 
+ \alpha_1 (\mu_2 + \mu_3) (\mu_2 + \mu_4) (\mu_3 + \mu_4), 
\end{align*}
The last equality holds by \eqref{eq:muijk}.
Then, we see that 
$I_1= |\alpha|^2 (2  |\alpha|^2-2)$ and 
\begin{align*}
I_2=&
\alpha_1^2 (\mu_3 + \mu_4)^2 
+ \alpha_1 (\mu_2 + \mu_3) (\mu_2 + \mu_4) (\mu_3 + \mu_4)\\
=&
- \frac{1}{4} (\mu_2 + \mu_3)^2 (\mu_2 + \mu_4)^2 \\
=&
- \frac{1}{4} (\mu_2 (\mu_2 + \mu_3 + \mu_4) + \mu_3 \mu_4)^2 \\
=&
- \frac{1}{4} (- \mu_1 \mu_2 + \mu_3 \mu_4)^2. 
\end{align*}
Hence, we have 
\begin{align*}
1 - \frac{* F^4}{24}
&=
1-|\alpha|^4 - \mu_1 \mu_2 \mu_3 \mu_4 +|\alpha|^2 (2  |\alpha|^2-2) 
- \frac{1}{4} (- \mu_1 \mu_2 + \mu_3 \mu_4)^2 \\
&=
(|\alpha|^2-1)^2 - \frac{1}{4} (\mu_1 \mu_2 + \mu_3 \mu_4)^2=0. 
\end{align*}
We can compute in the other cases similarly, 
and obtain the following Table \ref{table1}. 
By this table, the proof is completed. 
\begin{table}[htbp]
\begin{center}
\begin{tabular}{|l|c|r|}
\hline
Case & $1 - * F^4/24$ & $\pi^4_7 (F^2)$ \\ \hline \hline
(1)    & $1- (\alpha_1 + \mu_1) (\alpha_1 + \mu_2) (\alpha_1 + \mu_3) (\alpha_1 + \mu_4)$ 
& 0\\
(2)    & 0 &  $\neq 0$ \\
(3)-(a) & $(\mu_1+\mu_2)^2 (\mu_1-\mu_2)^2/4$ & 0 \\
(3)-(b) & 0 & $\neq 0$ \\
(4)-(a) & $(\mu_1+\mu_2)^2 (\mu_1-\mu_2)^2/4$ & 0 \\
(4)-(b) & 0 & $\neq 0$ \\
(5)     & 0 & $\neq 0$ \\
(6)     & 0 & $\neq 0$ \\
(7)-(a)  & 0 & 0 \\
(7)-(b)  & 0 & $\neq 0$ \\
(7)-(c)  & 0 & $\neq 0$ \\
(8)      & 0 & $\neq 0$ \\ 
(9)      & 0 & $\neq 0$ \\
(10)    & 0 & 0 \\
\hline
\end{tabular}
\end{center}
\caption{} \label{table1}
\end{table}
\end{proof}

By Lemma \ref{lem:sol F1}, we obtain the following simpler form 
of a 2-form $F \in \Lambda^2 W^*$ satisfying  
$\pi^2_7 \left(F - * F^3/6 \right) = 0$ and $\pi^4_7 (F^2)=0$. 

\begin{corollary} \label{cor:sol F1 F2}
If a 2-form $F \in \Lambda^2 W^*$ satisfies 
$\pi^2_7 \left(F - * F^3/6 \right) = 0$ and $\pi^4_7 (F^2)=0$, 
there exists $h \in \Sp$ such that 
\begin{align} \label{eq:sol F1 F2}
h^*F = 2 \lambda^2(\alpha_1 e^1) + \mu_1 e^{01} + \mu_2 e^{23} 
+ \mu_3 e^{45} + \mu_4 e^{67}, 
\end{align}
where $\alpha_1, \mu_j \in \R, \sum_{j=1}^4 \mu_j =0, 
\mu_1 \leq \mu_2 \leq \mu_3 \leq \mu_4$ and 
\eqref{eq:solsol} holds. 
\end{corollary}

\begin{proof}
By Lemma \ref{lem:sol F1} and table \ref{table1}, 
cases satisfying the assumption of this corollary 
are (1),(3)-(a), (4)-(a), (7)-(a), (10). 
We only have to show that cases 
(3)-(a), (4)-(a), (7)-(a) and (10) are reduced to (special cases of) (1) 
by the $\Sp$-action. 

Suppose that (3)-(a) holds. 
Since $\alpha_3=\alpha_7=0$ and $\mu_2+\mu_4=\mu_1+\mu_3=0$, we see that 
$$
F_7= 2 \lambda^2(\alpha_1 e^1 + \alpha_5 e^5), \qquad 
F_{21} = \mu_1 (e^{01} - e^{45}) + \mu_2 (e^{23} - e^{67}). 
$$
Set 
$\zeta_0=e^0 - \i e^4, \zeta_1=e^1 + \i e^5, \zeta_2=e^2 + \i e^6, \zeta_3=e^3 - \i e^7$. 
Then, it is straightforward to see that 
\begin{align*}
e^{01} - e^{45} &= {\rm Re} (\overline{\zeta_0} \wedge \zeta_1), \\
e^{23} - e^{67} &=  {\rm Im} (\zeta_2 \wedge \overline{\zeta_3}), \\
2 \lambda^2 (e^1) + 2 \i \lambda^2 (e^5) 
&= \zeta_0 \wedge \zeta_1 + \overline{\zeta_2 \wedge \zeta_3}
\end{align*}
and 
$$
\Phi 
= \frac{1}{2} \left( \frac{\i}{2} \sum_{j=0}^3 \zeta_j \wedge \overline {\zeta}_j \right)^2
+ {\rm Re} (\zeta_0 \wedge \zeta_1 \wedge \zeta_2 \wedge \zeta_3). 
$$
Hence, for $\theta \in \R$, setting $h:W \to W$ by 
$$
\left( h (\zeta_0), h (\zeta_1), h (\zeta_2), h (\zeta_3) \right)
= 
\left( e^{\i \theta} \zeta_0, \ e^{\i \theta} \zeta_1, \ 
e^{-\i \theta} \zeta_2, \ e^{-\i \theta} \zeta_3 \right), $$
we see that 
$h \in \Sp$, $h$ preserves $e^{01} - e^{45}, e^{23} - e^{67}$ and 
$$
h^* \left ( \lambda^2 (e^1) + \i \lambda^2 (e^5) \right)
= e^{2 \i \theta} \left ( \lambda^2 (e^1) + \i \lambda^2 (e^5) \right). 
$$
This implies that $h$ rotates the plane 
${\rm span} \{ \lambda^2 (e^1), \lambda^2 (e^5) \}$ fixing $F_{21}$, 
and hence, the case (3)-(a) is reduced to a special case of (1) by the $\Sp$-action. 

Since we can permute $\{ \mu_j \}$ by the ${\rm SU}(4)$-action, 
the case (4)-(a) is reduced to the case (3)-(a). 
Hence, by the same argument as above, 
the case (4)-(a) is reduced to a special case of (1) by the $\Sp$-action.

Suppose that (7)-(a) holds. Then, 
$$
F_7= 2 \lambda^2(\alpha_1 e^1 + \alpha_5 e^5 + \alpha_7 e^7), \qquad 
F_{21} = \mu (- e^{01} - e^{23} + e^{45} + e^{67}). 
$$
By the standard ${\rm SU}(2)$-action acting on ${\rm span} \{ e^4, e^5, e^6, e^7\}$, 
which preserves $F_{21}$, we can reduce the case (7)-(a) 
to the case of $\alpha_7=0$. 
Then, by the same argument as above, 
the case (7)-(a) is reduced to a special case of (1) by the $\Sp$-action. 

Suppose that (10) holds. Since $F_{21}=0$, 
there exists $h \in \Sp$ such that $h^* \alpha = e^1$. 
Hence, the case (10) is reduced to a special case of (1) by the $\Sp$-action. 
\end{proof}

\begin{remark}
By Lemma \ref{lem:eplicit a1} and Corollary \ref{cor:sol F1 F2}, 
we see that for any solution 
$F=F_7 + F_{21} \in \Lambda^2_7 W^* \oplus \Lambda^2_{21} W^*$ 
of $\pi^2_7 \left(F - * F^3/6 \right) = 0$ and $\pi^4_7 (F^2)=0$, 
$F_7$ is controlled by $F_{21}$. 
This fact is also reflected in Corollary \ref{cor:dDT estimate}. 
\end{remark}

\begin{theorem}\label{thm:1+F Spin7 neq}
For a 2-form $F \in \Lambda^2 W^*$ 
define $T_F: \Lambda^2 W^* \rightarrow \Lambda^2_7 W^*$ 
and $S_F: \Lambda^2 W^* \rightarrow \Lambda^4_7 W^*$ 
by 
$$
T_F (\beta) =\pi^2_7 \left( \beta - * \left( \frac{F^2}{2} \wedge \beta \right) \right), 
\qquad
S_F (\beta) =2 \pi^4_7 \left( F \wedge \beta \right). 
$$
\begin{enumerate}
\item
If $\pi^2_7 \left(F - * F^3/6 \right) = 0$ and $1-*F^4/24 \neq 0$ 
(then $\pi^4_7(F^2)=0$ is satisfied by Proposition \ref{prop:1+F 00 Spin7}), 
we have 
$$
\ker T_F = (I + F^\sharp)^* (\Lambda^2_{21} W^*). 
$$
In other words, 
there exists an isomorphism $P_F:\Lambda^2_7 W^* \rightarrow \Lambda^2_7 W^*$ 
such that 
$$
T_F 
= P_F \circ \pi^2_7 \circ ((I + F^\sharp)^{-1})^* 
= P_F \circ ((I + F^\sharp)^{-1})^* \circ \pi^2_{7, \Phi_F}, 
$$
where $\pi^2_{7, \Phi_F}$ is defined by \eqref{eq:proj2} for $\sigma=I + F^\sharp$. 
\item
If $\pi^2_7 \left(F - * F^3/6 \right) = 0$ and $\pi^4_7(F^2)=0$, we have 
$$
\ker T_F \cap \ker S_F = (I + F^\sharp)^* (\Lambda^2_{21} W^*). 
$$
In other words, 
there exists an isomorphism 
$Q_F:\Lambda^2_7 W^* \rightarrow \Im (T_F, S_F)$ 
such that 
$$
(T_F, S_F) 
= Q_F \circ \pi^2_7 \circ ((I + F^\sharp)^{-1})^* 
= Q_F \circ ((I + F^\sharp)^{-1})^* \circ \pi^2_{7, \Phi_F}. 
$$
\end{enumerate}
\end{theorem}

\begin{proof}
Since $F$ satisfies $\pi^2_7 \left(F - * F^3/6 \right) = 0$ and $\pi^4_7(F^2)=0$ 
in all cases, 
we may assume that $F$ is of the form of the right hand side of \eqref{eq:sol F1 F2} 
and satisfies \eqref{eq:solsol} by Corollary \ref{cor:sol F1 F2}. 

First, we show that 
$\ker T_F \cap \ker S_F \supset (I + F^\sharp)^* (\Lambda^2_{21} W^*)$
without any assumptions on $1-*F^4/24$. 
Set 
$$
C_j = \alpha_1 + \mu_j \qquad j =1, \cdots, 4. 
$$
Then, we see that 
$$
F=C_1 e^{01} + C_2 e^{23} + C_3 e^{45} + C_4 e^{67}. 
$$
Set 
$$
\xi= \sum_{1 \leq i<j<k \leq 4} C_i C_j C_k - \sum_{j=1}^4 C_j. 
$$
Note that \eqref{eq:solsol} is equivalent to 
\begin{align}\label{eq:xi=0}
\xi=0. 
\end{align}
Indeed, since 
\begin{align*}
C_i C_j C_k 
&=
\left( \alpha _{1}+\mu _{i}\right) \left( \alpha _{1}+\mu _{j}\right) \left( \alpha _{1}+\mu _{k}\right) \\
&= 
\alpha _{1}^{3}+\left( \mu _{i}+\mu _{j}+\mu _{k}\right) \alpha _{1}^{2}
+\left( \mu _{i}\mu _{k}+\mu _{j}\mu _{k}+\mu _{i}\mu _{j}\right) \alpha _{1} 
+\mu _{i}\mu _{j}\mu _{k}
\end{align*}
and 
\begin{align*}
\sum_{1 \leq i<j<k \leq 4} \left( \mu _{i}+\mu _{j}+\mu _{k}\right)
&= 3 \left( \mu_1+\mu _{2}+\mu _{3} + \mu_4 \right)
=0, \\
\sum_{1 \leq i<j<k \leq 4} \left( \mu _{i}\mu _{k}+\mu _{j}\mu _{k}+\mu _{i}\mu _{j}\right)
&= 2 \sum_{1 \leq i<j \leq 4} \mu_i \mu_j
= - \sum_{i=1}^4 \mu_i^2,  
\end{align*}
we have 
$$
\xi = 4 \alpha_1^3 - \left( \sum_{i=1}^4 \mu_i^2 \right) \alpha_1 
+ \sum_{1 \leq i<j<k \leq 4} \mu_i \mu_j \mu_k
-4 \alpha_1, 
$$
which is the left hand side of \eqref{eq:solsol}. 

Fix $u,v \in V=\R^7$ and set 
$$
\tau =e^{0}\wedge \varphi \left( u,v,\cdot \right) -u^{b}\wedge v^{b}. 
$$
By Lemma \ref{lem:lambda221}, $\Lambda^2_{21} W^*$ 
is spanned by elements of the form $\tau$. 
Set
\begin{equation}\label{202012271829}
T_\mu= 
\left \langle \left( I+F^{\sharp}\right) ^{\ast }\tau -* \left( \dfrac{F^{2}}{2}\wedge \left( I+F^{\sharp}\right) ^{\ast }\tau \right), 2\lambda ^{2}\left( e^{\mu }\right) \right \rangle 
\end{equation}
for $\mu=1,\cdots,7$. 
Then, $\left( I+F^{\sharp}\right) ^{\ast }\tau \in \ker T_F$ 
if and only if $T_\mu=0$ for any $\mu=1,\cdots,7$. 
Using 
$\left( I+F^{\sharp}\right) ^{\ast }\alpha =\alpha -i\left( \alpha ^{\sharp} \right) F$
for $\alpha \in W^*$, direct but lengthy computations show that 
\begin{align*}
T_{1} &=\xi ( -C_{1}\varphi \left( e_{1},u,v\right) +C_{2}e^{23}\left( u,v\right) 
 +C_{3}e^{45}\left( u,v\right) +C_{4}e^{67}\left( u,v\right) ), \\
T_{\mu} &=  -\xi * \varphi \left( e_1, e_\mu, u,v \right)
\end{align*}
for $\mu=2, \cdots,7$. 
Then, by \eqref{eq:xi=0}, all of these vanish and hence, 
$\ker T_F \supset (I + F^\sharp)^* (\Lambda^2_{21} W^*)$. 
Similarly, set 
\begin{equation}\label{202012271830}
S_\mu= 
\left \langle F \wedge \left( I+F^{\sharp}\right) ^{\ast }\tau, 
2\lambda ^{4} \left( e^{\mu }\right) \right \rangle 
\end{equation}
for $\mu=1,\cdots,7$. 
Then, $\left( I+F^{\sharp}\right) ^{\ast }\tau \in \ker S_F$ 
if and only if $S_\mu=0$ for any $\mu=1,\cdots,7$. 
Direct but lengthy computations show that 
\begin{align*}
S_1=0, \qquad S_\mu=T_\mu=0 \quad \mbox{for} \quad \mu=2, \cdots,7. 
\end{align*}
Hence, we see that $\ker S_F \supset (I + F^\sharp)^* (\Lambda^2_{21} W^*)$. 
\\

Next, we compute $\ker T_F \cap \left( I+F^{\sharp}\right) ^{\ast } \Lambda^2_7 W^*$
and $\ker T_F \cap \ker S_F \cap \left( I+F^{\sharp}\right) ^{\ast } \Lambda^2_7 W^*$. 
We still do not assume that $*F^4/24 \neq 1$. 
Since 
$\Lambda^2 W^* = \left( I+F^{\sharp}\right) ^{\ast } \Lambda^2_7 W^* 
\oplus \left( I+F^{\sharp}\right) ^{\ast } \Lambda^2_{21} W^*$ 
and 
$\ker T_F \cap \ker S_F \supset (I + F^\sharp)^* \Lambda^2_{21} W^*$, 
we see that 
\begin{align*}
\ker T_F 
=&\left( \ker T_F \cap \left( I+F^{\sharp}\right) ^{\ast } \Lambda^2_7 W^* \right)
\oplus \left( I+F^{\sharp} \right)^{\ast } \Lambda^2_{21} W^*,  \\
\ker T_F \cap \ker S_F 
=& \left( \ker T_F \cap \ker S_F \cap(I + F^\sharp)^* \Lambda^2_{7} W^* \right) 
\oplus \left( I+F^{\sharp} \right)^{\ast } \Lambda^2_{21} W^*. 
\end{align*}
Fix $\eta =\sum ^{7}_{\mu=1}\eta _{\mu}e^{\mu}$ for $\eta_\mu \in \R$. 
Similarly to \eqref{202012271829} and \eqref{202012271830}, set 
\begin{align*}
T'_\mu &= 
\left \langle \left( I+F^{\sharp}\right) ^{\ast } \left( 2 \lambda^2(\eta) \right) 
-* \left( \dfrac{F^{2}}{2}\wedge \left( I+F^{\sharp}\right) ^{\ast } \left( 2 \lambda^2(\eta) \right) \right), 2\lambda ^{2}\left( e^{\mu }\right) \right \rangle, \\
S'_\mu &= 
\left \langle F \wedge \left( I+F^{\sharp}\right) ^{\ast } \left( 2 \lambda^2(\eta) \right), 
2\lambda ^{4} \left( e^{\mu }\right) \right \rangle 
\end{align*}
for $\mu=1,\cdots,7$. 
Note that $\ker T_F \cap (I + F^\sharp)^* \Lambda^2_{7} W^* = \{\,0\,\}$ 
if and only if $T'_{\mu}=0$ for any $\mu=1,\cdots,7$ implies $\eta=0$. 
Similarly, 
$\ker T_F \cap \ker S_F \cap (I + F^\sharp)^* \Lambda^2_{7} W^* = \{\,0\,\}$ 
if and only if $T'_{\mu}=S'_{\mu}=0$ for any $\mu=1, \cdots ,7$ implies $\eta=0$. 
By direct but lengthy computations, we obtain 
\footnotesize
\begin{align}
T_1'&= 4 \eta _{1} \left( C_1 C_2 C_3 C_4 - \sum_{1 \leq i<j \leq 4} C_i C_j +1 \right), 
\label{eq:T1'}\\
T_2'&= 
4 (C_1 C_2-1) (C_3 C_4 -1) \eta_2 
+ 2 \left \{ (C_1 C_2-1) (C_3+C_4) + (C_1 + C_2) (-C_3 C_4+1) \right \} \eta_3, 
\label{eq:T2'}\\
T_3'&= 
- 2 \left \{ (C_1 C_2-1) (C_3+C_4) + (C_1 + C_2) (-C_3 C_4+1) \right \} \eta_2 
+ 4 (C_1 C_2-1) (C_3 C_4 -1) \eta_3, \label{eq:T3'}\\
T_4'&= 
4 (C_1 C_3-1) (C_2 C_4 -1) \eta_4 
+ 2 \left \{ (C_1 C_2+1) (C_3-C_4) + (C_1 - C_2) (C_3 C_4+1) \right \} \eta_5, \label{eq:T4'}\\
T_5'&= 
- 2 \left \{ (C_1 C_2+1) (C_3-C_4) + (C_1 - C_2) (C_3 C_4+1) \right \} \eta_4 
+ 4 (C_1 C_3-1) (C_2 C_4 -1) \eta_5, \label{eq:T5'}\\
T_6'&= 
4 (C_1 C_4-1) (C_2 C_3 -1) \eta_6 
+ 2 \left \{ (C_1 C_2+1) (-C_3+C_4) + (C_1 - C_2) (C_3 C_4+1) \right \} \eta_7, \label{eq:T6'}\\
T_7'&= 
- 2 \left \{ (C_1 C_2+1) (-C_3+C_4) + (C_1 - C_2) (C_3 C_4+1) \right \} \eta_6 
+ 4 (C_1 C_4-1) (C_2 C_3 -1) \eta_7. \label{eq:T7'}
\end{align}
\normalsize
Note that we use \eqref{eq:xi=0} for the computation of $T_1'$. 
Similarly, direct but lengthy computations show that
\footnotesize
\begin{align}
S_1'&= 0, \nonumber \\
S_2'&= 
- 4 (C_1 + C_2) (C_3 + C_4) \eta_2 
- 2 \left \{ (C_1 C_2-1) (C_3+C_4) + (C_1 + C_2) (-C_3 C_4+1) \right \} \eta_3, 
\label{eq:S2'}\\
S_3'&= 
2 \left \{ (C_1 C_2-1) (C_3+C_4) + (C_1 + C_2) (-C_3 C_4+1) \right \} \eta_2 
- 4 (C_1 + C_2) (C_3 + C_4) \eta_3, \label{eq:S3'}\\
S_4'&= 
-4 (C_1 + C_3) (C_2 + C_4) \eta_4 
- 2 \left \{ (C_1 C_2+1) (C_3-C_4) + (C_1 - C_2) (C_3 C_4+1) \right \} \eta_5, \label{eq:S4'}\\
S_5'&= 
2 \left \{ (C_1 C_2+1) (C_3-C_4) + (C_1 - C_2) (C_3 C_4+1) \right \} \eta_4 
-4 (C_1 + C_3) (C_2 + C_4) \eta_5, \label{eq:S5'}\\
S_6'&= 
-4 (C_1 + C_4) (C_2 + C_3) \eta_6 
- 2 \left \{ (C_1 C_2+1) (-C_3+C_4) + (C_1 - C_2) (C_3 C_4+1) \right \} \eta_7, \label{eq:S6'}\\
S_7'&= 
2 \left \{ (C_1 C_2+1) (-C_3+C_4) + (C_1 - C_2) (C_3 C_4+1) \right \} \eta_6 
- 4 (C_1 + C_4) (C_2 + C_3) \eta_7. \label{eq:S7'}
\end{align}
\normalsize
First, we show the following. 

\begin{lemma} \label{lem:nonzero}
We have 
$C_1 C_2 C_3 C_4 - \sum_{1 \leq i<j \leq 4} C_i C_j +1 \neq 0$. 
\end{lemma}

\begin{proof}
If $C_2 C_3 + C_2 C_4 + C_3 C_4 -1 \neq 0$, we have 
$$
C_1 = \frac{-C_2 C_3 C_4 + C_2 + C_3 + C_4}{C_2 C_3 + C_2 C_4 + C_3 C_4 -1} 
$$
by \eqref{eq:xi=0}. 
Then, 
\begin{align*}
&C_1 C_2 C_3 C_4 - \sum_{1 \leq i<j \leq 4} C_i C_j +1 \\
=&
\frac{-(C_2 C_3 C_4 - C_2 - C_3 - C_4)^2 - (C_2 C_3 + C_2 C_4 + C_3 C_4 -1)^2}
{C_2 C_3 + C_2 C_4 + C_3 C_4 -1}, 
\end{align*}
which is nonzero since $C_2 C_3 + C_2 C_4 + C_3 C_4 -1 \neq 0$. 

Suppose that $C_2 C_3 + C_2 C_4 + C_3 C_4 -1 = 0$. 
Then, by \eqref{eq:xi=0}, we have 
$
\xi= C_2 C_3 C_4 - (C_2 + C_3 + C_4) =0. 
$
Thus, we obtain 
$$
\begin{pmatrix}
C_{3}+C_{4} & C_{3} C_{4}-1 \\
-\left( C_{3} C_{4}-1\right)  & C_{3}+C_{4}
\end{pmatrix}\begin{pmatrix}
C_{2} \\
1
\end{pmatrix}=\begin{pmatrix}
0 \\
0
\end{pmatrix}. 
$$
Since $(C_{3}+C_{4})^2 + \left( C_{3} C_{4}-1\right)^2 >0$, 
we have $(C_2, 1) = (0,0)$, which is a contradiction. 
\end{proof}
Hence, by \eqref{eq:T1'}, 
$T_1'=0$ if and only if $\eta_1=0$. 
By \eqref{eq:T2'} and \eqref{eq:T3'}, 
$T_2'=T_3'=0$ implies $\eta_2=\eta_3=0$ 
if and only if 
\begin{align}
(C_1 C_2-1) (C_3 C_4 -1) &\neq 0 \quad \mbox{or}, \label{202012271902} \\
(C_1 C_2-1) (C_3+C_4) + (C_1 + C_2) (-C_3 C_4+1)  &\neq 0. \label{202012271903}
\end{align}
Now, we show that $T_2'=T_3'=0$ implies $\eta_2=\eta_3=0$ if $*F^4/24 \neq 1$. 
To see this, suppose that the left hand sides of \eqref{202012271902} and \eqref{202012271903} vanish.
If $C_1 C_2=1$, the assumption that the left hand side of \eqref{202012271903} vanishes implies that 
$$
(C_1 + C_2) (-C_3 C_4+1) = \frac{(C_1^2+1) (-C_3 C_4+1)}{C_1}=0. 
$$
Thus, we obtain $C_3 C_4=1$. In particular, we have 
$*F^4/24 = C_1 C_2 C_3 C_4=1$. 
Similarly, 
$C_3 C_4=1$ implies that $C_1 C_2=1$ and $*F^4/24 = C_1 C_2 C_3 C_4=1$. 
Hence, we see that 
$T_2'=T_3'=0$ implies $\eta_2=\eta_3=0$ if $*F^4/24 \neq 1$. 

Similarly, we show that $T_2'=T_3'=S_2'=S_3'=0$ implies $\eta_2=\eta_3=0$ without any assumptions on $*F^4/24$. 
By \eqref{eq:T2'}, \eqref{eq:T3'}, \eqref{eq:S2'} and \eqref{eq:S3'}, 
this is the case 
if and only if 
one of the following is nonzero. 
\begin{align*}
&(C_1 C_2-1) (C_3 C_4 -1), \\ 
&(C_1 + C_2) (C_3 + C_4), \\
&(C_1 C_2-1) (C_3+C_4) + (C_1 + C_2) (-C_3 C_4+1).
\end{align*}
To see that one of these is nonzero, suppose that all of these are 0. 
Then, by the argument above, the first and the third equation imply that 
$C_1 C_2=1$ and $C_3 C_4=1$. 
Then, 
$$
(C_1 + C_2) (C_3 + C_4) = \frac{(C_1^2+1) (C_4^2+1)}{C_1 C_3} \neq 0, 
$$
which is a contradiction.

We can discuss for $\eta_4, \eta_5, \eta_6, \eta_7$ similarly 
and we see that 
$\ker T_F = (I + F^\sharp)^* (\Lambda^2_{21} W^*)$ if 
$1-*F^4/24 \neq 0$ 
and 
$\ker T_F \cap \ker S_F = (I + F^\sharp)^* (\Lambda^2_{21} W^*)$. 
\\

Finally, we prove the last statement of (1). 
As proved above, we see that 
$$
P_F:=T_F \circ (I + F^\sharp)^*|_{\Lambda^2_7 W^*} 
: \Lambda^2_7 W^* \rightarrow \Lambda^2_7 W^*
$$
is an isomorphism if $* F^4/24 \neq 1$. 
Then, for any $\beta_7 \in \Lambda^2_7 W^*$ and 
$\beta_{21} \in \Lambda^2_{21} W^*$, 
we see that 
\begin{align*}
\left( P_F \circ \pi^2_7 \circ ((I + F^\sharp)^{-1})^* \right) 
\left( (I + F^\sharp)^* (\beta_7) \right)
&=P_F (\beta_7)
= T_F \left( (I + F^\sharp)^* (\beta_7) \right), \\
\left( P_F \circ \pi^2_7 \circ ((I + F^\sharp)^{-1})^* \right) 
\left( (I + F^\sharp)^* (\beta_{21}) \right)
&=0
= T_F \left( (I + F^\sharp)^* (\beta_{21}) \right). 
\end{align*}
Hence, we obtain $T_F = P_F \circ \pi^2_7 \circ ((I + F^\sharp)^{-1})^*$.  

Similarly, for the last statement of (2), an isomorphism 
$$
Q_F:= (T_F, S_F) \circ (I + F^\sharp)^*|_{\Lambda^2_7 W^*} 
: \Lambda^2_7 W^* \rightarrow \Im (T_F, S_F)
$$ 
satisfies the desired property. 
\end{proof}

\begin{remark}
Since 
$C_1 C_2 C_3 C_4 - \sum_{1 \leq i<j \leq 4} C_i C_j +1 =  
1 - \la F^2, \Phi \ra/2 + * F^4/24$, 
we see that 
$$
1 - \frac{1}{2} \la F^2, \Phi \ra + \frac{* F^4}{24} \neq 0
$$
for $F \in \Lambda^2 W^*$ satisfying 
$\pi^2_7 \left(F - * F^3/6 \right) = 0$ and $\pi^4_7(F^2)=0$ 
by Lemma \ref{lem:nonzero}. 
We can also prove this by the ``mirror" of the Cayley equality in \cite{KYvol}. 
\end{remark}

Finally, we give other applications of Corollary \ref{cor:sol F1 F2}. 
The following proves Proposition \ref{prop:dDT norm}. 

\begin{corollary} \label{cor:dDT estimate}
Suppose that $F \in \Lambda^2 W^*$ satisfies 
$\pi^2_7 \left(F - * F^3/6 \right) = 0$ and $\pi^4_7(F^2)=0$. 
Set $F_7= \pi^2_7 (F)$ and $F_{21} = \pi^2_{21} (F)$.  
\begin{enumerate}
\item If $F_{21}=0$, we have $F_7=0$ or $|F_7| = 2$. 
\item
We have 
\[
|F_7| \leq 2 \sqrt{\frac{|F_{21}|^2 + 4}{3}} 
\cos \left( \frac{1}{3} \arccos \left( \frac{|F_{21}|^3}{(|F_{21}|^2+4)^{3/2}} \right) \right). 
\]
\end{enumerate}
\end{corollary}

\begin{proof}
By Corollary \ref{cor:sol F1 F2}, 
we may assume that $F$ is of the form of the right hand side of \eqref{eq:sol F1 F2} 
satisfying \eqref{eq:solsol}. 

If $F_{21}=0$, \eqref{eq:solsol} implies that $(\alpha_1^2- 1) \alpha_1 =0$. 
Since $|F_7|=2 |\alpha_1|$, we obtain (1). 
Next, we show (2). Recall that $|F_{21}|^2 = \sum_{j=1}^4 \mu_j^2$. 
By \eqref{eq:solsol}, we have 
\begin{align*}
\left(4 \alpha_1^2 - \sum_{j=1}^4 \mu_j^2 - 4 \right) |\alpha_1| 
\leq 
\left | 4 \alpha_1^2 - \sum_{j=1}^4 \mu_j^2 -4 \right | |\alpha_1| 
= 
\left | \sum_{1 \leq i < j <k \leq 4} \mu_i \mu_j \mu_k \right |. 
\end{align*}
Then, by Lemma \ref{lem:dDT estimate sub}, we obtain 
$$
4 |\alpha_1|^3 - (|F_{21}|^2+4) |\alpha_1| - \frac{\sqrt{3}}{9} |F_{21}|^3 \leq 0. 
$$
Thus, if we define a cubic polynomial $f(x)$ (with a parameter $\lambda \geq 0$) by 
\[
f(x) = 4 x^3 - \left( \lambda^2 + 4 \right) x - \frac{\sqrt{3}}{9} \lambda^3, 
\]
$x =|\alpha_1|$ satisfies $f(x) \leq 0$ for $\lambda =|F_{21}|$. 
By the Vi\`ete's formula for a cubic equation, 
the largest solution of $f(x)=0$ is given by 
\[
x = x_0 := \sqrt{ \frac{\lambda^2 + 4}{3}} 
\cos \left( \frac{1}{3} \arccos \left( \frac{\lambda^3}{(\lambda^2+4)^{3/2}} \right) \right). 
\]
Hence, we see that 
$f(x) \leq 0$ implies that $x \leq x_0$. 
This together with the equation $|F_7| = 2 |\alpha_1|$ implies (2). 
\end{proof}

The following is useful in Section \ref{sec:moduli generic}. 

\begin{lemma} \label{lem:generic pt}
Suppose that $F \in \Lambda^2 W^*$ satisfies 
$\pi^2_7 \left(F - * F^3/6 \right) = 0$, $*F^4/24 \neq 1$ and $F \neq 0$. 
Then, if $\gamma \in \Lambda^2_7 W^*$ satisfies 
\begin{align}
\pi^4_1 (F \wedge \gamma) &=0, \label{eq:L2 on2 pt}\\
\pi^4_7 \left( \left( 6F - * F^3 \right) \wedge \gamma \right) &=0, \label{eq:L2 on3 pt}\\
\pi^4_{35} \left( \left(6 F + * F^3 \right) \wedge \gamma \right) &=0,  \label{eq:L2 on4 pt}
\end{align}
we have $\gamma=0$. 
\end{lemma}

\begin{proof}
By Corollary \ref{cor:sol F1 F2}, 
we may assume that $F$ is of the form of the right hand side of \eqref{eq:sol F1 F2} 
satisfying \eqref{eq:solsol}. 
Set $C_j=\alpha_1 + \mu_j$. That is, 
$$
F=C_1 e^{01} + C_2 e^{23} + C_3 e^{45}  +C_4 e^{67}.  
$$
Note that we have the equation \eqref{eq:xi=0}. 
Set 
$$
\gamma = 2 \lambda^2(\eta) = e^0 \wedge \eta + i(\eta^\sharp) \varphi, 
$$
where $\eta = \sum_{j=1}^7 \eta_j e^j \in V^*$. 
Since  
$* \left( F \wedge \gamma \wedge \Phi \right) = 3 \la F, \gamma \ra$
and 
\begin{align*}
\la F, \gamma \ra 
&= C_1 \eta_1 + C_2 \varphi (\eta, e_2, e_3) 
+ C_3 \varphi (\eta, e_4, e_5) + C_4 \varphi (\eta, e_6, e_7) \\
&=
(C_1+C_2+C_3+C_4) \eta_1, 
\end{align*}
\eqref{eq:L2 on2 pt} is equivalent to 
\begin{align} \label{eq:L2 on2 pt 1}
(C_1+C_2+C_3+C_4) \eta_1 =0. 
\end{align}

Next, we rewrite \eqref{eq:L2 on3 pt} and \eqref{eq:L2 on4 pt}. 
Since $\pi^2_7 \left(F - * F^3/6 \right) = 0$ and \eqref{eq:F2decomp}, we see that 
$$
\pi^4_7 \left( \left( F - \frac{* F^3}{6} \right) \wedge \gamma \right) 
= \frac{1}{2} \left( \left( F - \frac{* F^3}{6} \right) + * \left( F - \frac{* F^3}{6} \right) \right). 
$$
Since $\Lambda^4_{35} W^*$ is the space of anti self dual 4-forms, we have 
$$
\pi^4_{35} \left( \left( F + \frac{* F^3}{6} \right) \wedge \gamma \right) 
= \frac{1}{2} \left( \left( F + \frac{* F^3}{6} \right) - * \left( F + \frac{* F^3}{6} \right) \right). 
$$
By a direct computation, we see that 
$$
F \wedge \gamma \pm *(F \wedge \gamma) = \sum_{j=1}^7 \eta_j \Theta_j, 
$$
where 
\begin{align*}
\Theta_1 
=&
\left( (C_1 + C_2) \pm (C_3+C_4) \right) \left(e^{0123} \pm e^{4567} \right) \\
&+
\left( (C_1 + C_3) \pm (C_2+C_4) \right) \left(e^{0145} \pm e^{2367} \right) \\
&+
\left( (C_1 + C_4) \pm (C_2+C_3) \right) \left(e^{0167} \pm e^{2345} \right), \\
\Theta_2 
=&
(C_1 \pm C_2) \left(e^{0146} -e^{0157} \pm e^{2346} \mp e^{2357} \right) \\
&+
(C_3 \pm C_4) \left(e^{0245} \pm e^{0267} - e^{1345} \mp e^{1367} \right), \\
\Theta_3 
=&
(C_1 \pm C_2) \left(-e^{0147} -e^{0156} \mp e^{2347} \mp e^{2356} \right) \\
&+
(C_3 \pm C_4) \left(e^{0345} \pm e^{0367} + e^{1245} \pm e^{1267} \right), \\
\Theta_4 
=&
(C_1 \pm C_3) \left(-e^{0126} + e^{0137} \mp e^{2456} \pm e^{3457} \right) \\
&+
(C_2 \pm C_4) \left(e^{0234} \pm e^{0467} - e^{1235} \mp e^{1567} \right), \\
\Theta_5 
=&
(C_1 \pm C_3) \left(e^{0127} + e^{0136} \pm e^{2457} \pm e^{3456} \right) \\
&+
(C_2 \pm C_4) \left(e^{0235} \pm e^{0567} + e^{1234} \pm e^{1467} \right), \\
\Theta_6 
=&
(C_1 \pm C_4) \left(e^{0124} - e^{0135} \pm e^{2467} \mp e^{3567} \right) \\
&+
(C_2 \pm C_3) \left(e^{0236} \pm e^{0456} - e^{1237} \mp e^{1457} \right), \\
\Theta_7
=&
(C_1 \pm C_4) \left(-e^{0125} - e^{0134} \mp e^{2567} \mp e^{3467} \right) \\
&+
(C_2 \pm C_3) \left(e^{0237} \pm e^{0457} + e^{1236} \pm e^{1456} \right). 
\end{align*}
Since 
$$
\frac{* F^3}{6} = C_2 C_3 C_4 e^{01} + C_1 C_3 C_4 e^{23} + C_1 C_2 C_4 e^{45} + C_1 C_2 C_3 e^{67}, 
$$
\eqref{eq:L2 on3 pt} is equivalent to 
\begin{align} \label{eq:L2 on3 pt 1}
\begin{split}
(C_1+C_2) (C_3 C_4 -1) \eta_a &=(C_3+C_4) (C_1 C_2 -1) \eta_a=0, \\
(C_1+C_3) (C_2 C_4 -1) \eta_b &=(C_2+C_4) (C_1 C_3 -1) \eta_b=0, \\
(C_1+C_4) (C_2 C_3 -1) \eta_c &=(C_2+C_3) (C_1 C_4 -1) \eta_c=0
\end{split}
\end{align}
for $a=2,3, b=4,5$ and $c=6,7$. 
Similarly, \eqref{eq:L2 on4 pt} is equivalent to 
\begin{align} \label{eq:L2 on4 pt 1}
\begin{split}
\left \{ (C_1+C_2) (C_3 C_4+1) - (C_3+C_4) (C_1 C_2 + 1) \right \} \eta_1 &=0, \\
\left \{ (C_1+C_3) (C_2 C_4+1) - (C_2+C_4) (C_1 C_3 + 1) \right \} \eta_1 &=0, \\
\left \{ (C_1+C_4) (C_2 C_3+1) - (C_2+C_3) (C_1 C_4 + 1) \right \} \eta_1 &=0, 
\end{split}
\end{align}
\begin{align} \label{eq:L2 on4 pt 2}
\begin{split}
(C_1-C_2) (C_3 C_4 -1) \eta_a &=(C_3-C_4) (C_1 C_2 -1) \eta_a=0, \\
(C_1-C_3) (C_2 C_4 -1) \eta_b &=(C_2-C_4) (C_1 C_3 -1) \eta_b=0, \\
(C_1-C_4) (C_2 C_3 -1) \eta_c &=(C_2-C_3) (C_1 C_4 -1) \eta_c=0
\end{split}
\end{align}
for $a=2,3, b=4,5$ and $c=6,7$. 
It is straightforward to see that 
all the coefficients of $\eta_1$ in \eqref{eq:L2 on2 pt 1} and \eqref{eq:L2 on4 pt 1}  
vanish if and only if $C_1=C_2=C_3=C_4=0$, 
which violates the assumption $F \neq 0$. 
Hence, we obtain $\eta_1=0$. 

It is also straightforward to see that 
all the coefficients of $\eta_2$ and $\eta_3$ in \eqref{eq:L2 on3 pt 1} and \eqref{eq:L2 on4 pt 2} 
vanish if and only if $C_1=C_2=C_3=C_4=0$ or 
$C_1 C_2= C_3 C_4 =1$. 
Since $*F^4/24 = C_1 C_2 C_3 C_4$, both cases are contrary to the assumptions. 
Hence, we obtain $\eta_2=\eta_3=0$. 

Similarly, we obtain $\eta_4=\eta_5=\eta_6=\eta_7=0$ 
and the proof is completed. 
\end{proof}

%%%%%%%%%%%%%%%%%%%%%%%%%%%%%%%%%%%%%%%%%%%%%%%%%%%%%%%%%
%%%%%%%%%%%%%%%%%%%%%%%%%%%%%%%%%%%%%%%%%%%%%%%%%%%%%%%%%
%%%%%%%%%%%%%%%%%%%%%%%%%%%%%%%%%%%%%%%%%%%%%%%%%%%%%%%%%
%%%%%%%%%%%%%%%%%%%%%%%%%%%%%%%%%%%%%%%%%%%%%%%%%%%%%%%%%
%%%%%%%%%%%%%%%%%%%%%%%%%%%%%%%%%%%%%%%%%%%%%%%%%%%%%%%%%
%%%%%%%%%%%%%%%%%%%%%%%%%%%%%%%%%%%%%%%%%%%%%%%%%%%%%%%%%
\section{Notation} \label{app:notation}
We summarize the notation used in this paper. 
We use the following for a manifold $X$ with a $G_2$- or $\Sp$-structure. 
Denote by $g$ the associated Riemannian metric.

\vspace{0.5cm}
\begin{center}
\begin{tabular}{|lc|l|}
\hline
Notation                        & & Meaning \\ \hline \hline
 $i(\,\cdot\,)$                &  & The interior product \\
$\Gamma(X, E)$            & &  The space of all smooth sections of  a vector bundle $E \rightarrow X$\\
$\Om^k$                       & & $\Om^k = \Om^k (X) = \Gamma
(X, \Lambda^k T^*X)$ \\
$S^2 T^*X$                   & & The space of symmetric (0,2)-tensors on $X$ \\
$S^2_0 T^*X$                & & The space of traceless symmetric (0,2)-tensors on $X$ \\
$\Ss^2$                        & & $\Ss^2= \Gamma (X, S^2 T^*X)$ \\
$\Ss^2_0$                     & & $\Ss^2_0= \Gamma (X, S^2_0 T^*X)$ \\
$b^k$                            & & the $k$-th Betti number  \\
$H^k_{dR}$                     & & the $k$-th de Rham cohomology \\
$H^k (\#)$                      & & the $k$-th cohomology of a complex $(\#)$ \\
$Z^1$                            & &  the space of closed 1-forms \\
$v^{\flat} \in T^* X$       & & $v^{\flat} = g(v, \,\cdot\,)$ for $v \in TX$ \\ 
$\alpha^{\sharp} \in TX$ & & $\alpha = g(\alpha^{\sharp}, \,\cdot\,)$ for $\alpha \in T^* X$ \\
$\vol$                           & & The volume form induced from $g$ \\
$\Lambda^k_\l T^*X$      & & The subspace of $\Lambda^k T^*X$ corresponding to  \\
                                  & & an $\l$-dimensional irreducible subrepresentation  \\
$\Om^k_\l$                    & &  $\Om^k_\l = \Gamma (X, \Lambda^k_\l T^*X)$ \\
$\pi^k_\l$                       & & The projection $\Lambda^k T^*X \rightarrow \Lambda^k_\l T^*X$ or 
                                           $\Om^k \rightarrow \Om^k_\l$ \\
\hline
\end{tabular}
\end{center}
%%%%%%%%%%%%%%%%%%%%%%%%%%%%%%%%%%%%%%%%%%%%%%%%%%%%%%%%%
%%%%%%%%%%%%%%%%%%%%%%%%%%%%%%%%%%%%%%%%%%%%%%%%%%%%%%%%%
%%%%%%%%%%%%%%%%%%%%%%%%%%%%%%%%%%%%%%%%%%%%%%%%%%%%%%%%%
%%%%%%%%%%%%%%%%%%%%%%%%%%%%%%%%%%%%%%%%%%%%%%%%%%%%%%%%%
%%%%%%%%%%%%%%%%%%%%%%%%%%%%%%%%%%%%%%%%%%%%%%%%%%%%%%%%%
%%%%%%%%%%%%%%%%%%%%%%%%%%%%%%%%%%%%%%%%%%%%%%%%%%%%%%%%%

\vspace{0.5cm}
\noindent{{\bf Acknowledgements}}: 
The authors would like to thank an anonymous referee for the careful reading 
of an earlier version of this paper and useful comments on it.

%%%%%%%%%%%%%%%%%%%%%%%%%%%%%%%%%%%%%%%%%%%%%%%%%%%%%%%%%
%%%%%%%%%%%%%%%%%%%%%%%%%%%%%%%%%%%%%%%%%%%%%%%%%%%%%%%%%
%%%%%%%%%%%%%%%%%%%%%%%%%%%%%%%%%%%%%%%%%%%%%%%%%%%%%%%%%
%%%%%%%%%%%%%%%%%%%%%%%%%%%%%%%%%%%%%%%%%%%%%%%%%%%%%%%%%
%%%%%%%%%%%%%%%%%%%%%%%%%%%%%%%%%%%%%%%%%%%%%%%%%%%%%%%%%
%%%%%%%%%%%%%%%%%%%%%%%%%%%%%%%%%%%%%%%%%%%%%%%%%%%%%%%%%
\end{document}